\documentclass{article}
\usepackage[utf8]{inputenc}
\usepackage{amsmath}
\usepackage{amssymb}
\usepackage{amsthm}
\usepackage{color}
\usepackage{hyperref}
\hypersetup{
	hidelinks,
	pdftitle={Semiparametric Bernstein-von Mises theorems from Stein's method},
	pdfauthor={Paul Rosa}
}
\usepackage{bbm}
\usepackage{graphicx}
\usepackage{subcaption}
\usepackage[margin=0.8in]{geometry}
\usepackage[backend=biber,
style=alphabetic,
]{biblatex}
\newcommand{\RR}{\mathbb{R}}

\newcommand{\Pro}{\mathbb{P}}

\renewcommand{\AA}{\mathcal{A}}

\newcommand{\CC}{\mathcal{C}}

\newcommand{\FF}{\mathcal{F}}

\newcommand{\HH}{\mathcal{H}}

\newcommand{\MM}{\mathcal{M}}

\newcommand{\NN}{\mathcal{N}}

\newcommand{\XX}{\mathcal{X}}

\newcommand{\norm}[1]{\left\Vert #1\right\Vert}

\newcommand{\floor}[1]{\lfloor #1\rfloor}

\renewcommand{\b}[1]{\left\{ #1\right\}}
\renewcommand{\c}[1]{\left[ #1\right]}
\newcommand{\p}[1]{\left( #1\right)}
\newcommand{\abs}[1]{\left| #1\right|}
\newcommand{\iid}{\overset{i.i.d}{\sim}}
\newcommand{\ind}[1]{\mathbbm{1}_{#1}}

\newcommand{\inner}[1]{\left\langle#1\right\rangle}
\DeclareMathOperator{\Span}{span}
\DeclareMathOperator{\Supp}{supp}
\DeclareMathOperator{\Tr}{Tr}
\DeclareMathOperator{\Cov}{Cov}
\DeclareMathOperator{\Var}{Var}
\DeclareMathOperator{\Div}{Div}
\DeclareMathOperator{\Diag}{Diag}
\DeclareMathOperator{\Dir}{Dir}

\DeclareMathOperator{\Lip}{Lip}
\DeclareMathOperator{\vol}{vol}

\title{Semiparametric Bernstein-von Mises theorems from Stein's method}
\author{Paul Rosa \\ Statistical Laboratory, University of Cambridge}
\date{}

\newtheorem{theorem}{Theorem}
\newtheorem{lemma}{Lemma}
\newtheorem{definition}{Definition}
\newtheorem{proposition}{Proposition}
\newtheorem{corollary}{Corollary}
\newtheorem{remark}{Remark}

\newtheorem{assumption}{Assumption}
\newtheorem{prior}{Prior}

\begin{document}
	
\maketitle

\begin{abstract}
We introduce a novel proof strategy for semiparametric Bernstein–von Mises theorems based on Stein's method combined with an information-geometric framework. Rather than controlling the effect of the prior on the asymptotic marginal posterior distribution of the functional of interest through the stability of an integrated likelihood under a suitable perturbation, we characterise its influence through the prior-weighted divergence of suitable vector fields over the statistical model. Applying Stein's method in this setting instead of the usual Laplace-transform approach yields explicit non-asymptotic upper bounds on the bounded-Lipschitz distance between marginal posterior distributions and the corresponding Gaussian limits predicted by semiparametric efficiency theory. These bounds consist entirely of local scalar differential quantities associated with the functional, the prior and a chosen vector field—typically related to the efficient influence function—evaluated over posterior contraction sets. We apply the theory to quadratic functionals in Gaussian white-noise models and to linear, quadratic and general integral functionals in histogram density models, with both conjugate and non-conjugate priors. For linear and quadratic functionals, the same theory identifies the differential term responsible for posterior bias and allows us to remove it through a natural functional correction, yielding Bernstein–von Mises-type theorems in regimes where they may fail for the original uncorrected functional.
\end{abstract}

\tableofcontents

\section{Introduction}\label{section:intro}

Bernstein--von Mises (hereafter BvM) theorems provide a frequentist justification for Bayesian uncertainty quantification by showing that suitably centred and rescaled posterior distributions are asymptotically Gaussian. While this phenomenon is classical in regular finite-dimensional models \cite{vaartAsymptoticStatistics1998}, its extension to infinite-dimensional models is substantially more delicate: posterior contraction alone does not in general guarantee that Bayesian credible sets have asymptotically correct frequentist coverage, see e.g. \cite{freedmanWaldLectureBernsteinvon1999} or \cite{rivoirardBernsteinMisesTheorem2012,castilloBernsteinVonMisesTheorem2015} for negative results. Semiparametric BvM theorems address this question for low-dimensional functionals $\Psi(f)$ of an infinite-dimensional parameter $f$, establishing Gaussian limits for marginal posterior distributions centred at efficient estimators \cite{shenAsymptoticNormalitySemiparametric2002,castilloSemiparametricBernsteinvonMises2012a,rivoirardBernsteinMisesTheorem2012,castilloBernsteinVonMisesTheorem2015} and in particular the approach of \cite{castilloBernsteinVonMisesTheorem2015} became the standard one in the literature. Related nonparametric BvM results establish Gaussian approximations for the full posterior in suitably weak infinite-dimensional topologies \cite{castilloNonparametricBernsteinMises2013,castilloBernsteinVonMisesPhenomenon2014}, and rely on a combination of the approach of \cite{rivoirardBernsteinMisesTheorem2012,castilloBernsteinVonMisesTheorem2015} together with a suitable tightness criterion. These theories have since been developed in a variety of nonparametric and semiparametric models; see, for instance, \cite{rousseauFrequentistPropertiesBayesian2016,ghosalFundamentalsNonparametricBayesian2017,castilloBayesianNonparametricStatistics2024} and the references therein.

The proof strategy for semiparametric BvM theorems outlined in \cite{rivoirardBernsteinMisesTheorem2012,castilloBernsteinVonMisesTheorem2015} relies on pointwise convergence in probability of the posterior Laplace transform of $\sqrt{n}\p{\Psi(f)-\hat\Psi_0}$ to that of the limiting Gaussian distribution where here $\hat \Psi_0$ is an efficient estimator of the quantity of interest $\Psi\p{f_0}$ and $f$ is a posterior draw, from which weak convergence in probability of the corresponding marginal posterior distribution can be deduced. In the weak nonparametric setting \cite{castilloNonparametricBernsteinMises2013,castilloBernsteinVonMisesPhenomenon2014}, a related argument is used to establish convergence of finite-dimensional posterior marginals, in combination with an additional tightness criterion. A central ingredient of the proof strategy is a local deformation of the statistical model: given a local parametrisation $\eta$ by a Hilbert space $\mathcal H$, one considers a change of variables of the form $\eta \mapsto \eta_t = \eta - ts / \sqrt{n}$ for a suitably chosen direction $s$, typically related to an approximation of the efficient influence function. Taylor expansions of the likelihood and the functional along this path reduce the problem, among other requirements, to conditions $(2.12)$ and $(2.14)$ of \cite{castilloBernsteinVonMisesTheorem2015}. While the former controls the relevant Taylor remainders, the latter is the condition in which the prior enters explicitly. Often referred to as the ``change of variables'' or ``ratio of integrals'' condition, it requires a suitable asymptotic stability of the integrated likelihood under the local deformation and in turn controls the influence of the prior on the limiting posterior distribution.

The present paper develops a complementary, infinitesimal alternative to this proof technique when considering high-dimensional priors. A ratio of integrals under a local change of variables measures the effect of a small but macroscopic deformation of the statistical model; from a differential-geometric perspective and provided that a suitable Riemannian metric over the statistical model has been fixed, an infinitesimal analog can be naturally described by a vector field $V$ (encoding the direction of perturbation) and the corresponding change of prior mass by its weighted divergence $\Div_\pi V := \Div V + \inner{\nabla \ln \pi | V}$, where $\pi$ is the density of the prior with respect to the Riemannian volume.\footnote{Under suitable assumptions, for a distribution $\Pi$, a vector field $V$ and a family of diffeomorphisms $\varphi_t$ over a Riemannian manifold such that $\varphi_0 = I$ and $\dot \varphi_0 = V$, using Jacobi's formula one can see that the induced Radon-Nikodym derivative $L_t = d\p{\Pi \circ \varphi_t^{-1}}/d\Pi$ satisfies $\dot L_0 = - \Div_\pi V$, see Proposition \ref{proposition:liouvilles_formula} for a proof and an extension to the case $t \neq 0$ when considering vector flows.} This viewpoint also removes the dependence of the formulation on a particular representation of the local deformation. Indeed, the property that $t\mapsto\eta_t$ is linear depends on the chosen parametrisation. For example, in a density model one may parametrise $f$ around $f_0$ either through the score $\eta = \ln \p{f/f_0} - \int \ln \p{f/f_0} f_0$ or $\eta' = f/f_0 - 1$, and a path which is linear in the first coordinate system is generally nonlinear in the second. This coordinate dependence is not an obstruction to the classical approach, since convenient or canonical parametrisations are available in many models, but it motivates asking whether the relevant prior-stability mechanism can be formulated intrinsically in terms of the infinitesimal direction of perturbation itself. This is closely related to the discussion following Corollary \(1\) of \cite{castilloBernsteinVonMisesTheorem2015}, where nonlinear perturbations are mentioned as a possible route towards more general results.

Information geometry provides a natural framework for such a formulation. Tangent vectors describe first-order deformations independently of the particular curves used to realise them, while the Fisher information equips the statistical model with a canonical Riemannian metric. In the approach developed below, the prior-dependent contribution is encoded by the weighted divergence $\Div_{\pi}V$ of an appropriate vector field $V$, where $\pi$ denotes the density of the prior with respect to the Riemannian volume (i.e. Jeffreys' prior in this context). The resulting sufficient conditions therefore have a different structure from the classical ratio-of-integrals condition: rather than controlling a ratio of integrated likelihood under a specified macroscopic perturbation, they involve explicit local differential quantities which can be controlled directly on sets over which the posterior concentrates. The weighted-divergence formulation can thus be viewed as an infinitesimal counterpart to the prior-stability mechanism captured by the classical ratio-of-integrals condition.

A second motivation for this differential formulation is quantitative. Laplace-transform arguments are particularly well suited to establishing weak convergence of posterior distributions, but pointwise convergence of the Laplace transform does not by itself provide a finite-$n$ bound on a probability metric measuring the quality of the Gaussian approximation guaranteed by the BvM theorem. By combining the preceding geometric formulation with Stein's method \cite{steinBoundErrorNormal1972,rossFundamentalsSteinsMethod2011}, we instead obtain explicit upper bounds on the bounded-Lipschitz distance (a common choice within the Bayesian nonparametric literature) between the posterior distribution of $\sqrt n\p{\Psi(f)-\hat\Psi_0}$ and its Gaussian limit. These bounds are expressed in terms of local differential quantities that can be directly controlled on posterior contraction sets. Besides yielding BvM theorems whenever these bounds vanish asymptotically, their explicit form makes it possible to quantify the quality of the Gaussian approximation. This quantitative perspective is related in spirit to \cite{katsevichImprovedDimensionDependence2025}, where the author obtains high-probability bounds in total variation for high-dimensional parametric BvM theorems; the present work instead focuses on semiparametric functionals and on bounds whose structure explicitly reflects posterior concentration, the geometry of the functional and the prior.

Finally, the differential formulation also provides a natural perspective on the biases induced by Bayesian methods in semiparametric contexts. The weighted-divergence term explicitly identifies a contribution of the prior to the centring error and, when this contribution does not vanish at the required scale, suggests a corresponding prior-dependent correction of the functional. In Section \ref{section:applications}, we make this mechanism explicit for linear and quadratic functionals and relate the resulting correction to the bias-correction technique introduced in \cite{castilloBernsteinVonMisesTheorem2015}.

The rest of the paper is organised as follows: in Section \ref{section:general_results} we introduce the necessary information geometry formalism (Section \ref{section:information_geometry}) before presenting the main results of this paper (Section \ref{section:general_bvm_stein}), a new proof technique for the semiparametric BvM theorem that relies on an application of Stein's method. We then present several applications in Section \ref{section:applications} to white noise (Section \ref{section:white_noise}) and histogram (Section \ref{section:histograms}) models. Section \ref{section:simulation_study} contains numerical illustrations for some of the theoretical claims in Section \ref{section:applications}.

\section{General results}\label{section:general_results}

\subsection{Information geometry}\label{section:information_geometry}

We consider a statistical model $\MM = \b{P_\theta : \theta \in \Theta}$ over a measurable space $\p{\XX,\AA}$, which we assume to be dominated by some measure $\lambda$. This allows us to define the density $p_\theta := \frac{dP_\theta}{d\lambda}$ and the log-likelihood $\ell = \ell(\theta,x) = \ln p_\theta\p{x}$. We assume that $\Theta$ is a manifold without boundary of finite dimension $p$ (although in practice $p$ will be allowed to grow with the sample size $n$). We equip $\Theta$ with the Fisher information metric $g$ (which we assume to be well-defined) given in local coordinates as $g_{\mu \nu} = P_\theta \c{\p{\partial_\mu \ell} \p{\partial_\nu \ell}}$ where $\partial_\mu \ell = \frac{\partial}{\partial \theta^\mu} \ell\p{\theta,\cdot}$, as well as the resulting Levi-Civita connection $\nabla$. In the last display we have used the notation $\mu\c{f} := \int f d\mu$ for any measure $\mu$ and integrable or nonnegative measurable function $f$ over some measurable space that we also use throughout the rest of this paper. We also drop the explicit dependence of the log-likelihood in the observed $x$ and simply consider it as a function of $\theta$ alone. The tangent space $T_\theta$ of $\Theta$ at $\theta \in \Theta$ can be described concretely as
\[
T_\theta = \Span \b{\partial_\mu \ell : 1 \leq \mu \leq p} \subset H_\theta := \b{s \in L^2\p{P_\theta} : P_\theta \c{s} = 0},
\]
and the Christoffel symbols $\Gamma_{~\mu \nu}^\rho$ of the Levi-Civita connection are then given by
\[
\Gamma_{~\mu \nu}^\rho = g^{\rho \tau}P_\theta\c{\p{\partial_\tau \ell}\p{\partial_{\mu \nu} \ell + \frac{1}{2} \p{\partial_\mu \ell} \p{\partial_\nu \ell}}} = \frac{1}{2} g^{\rho \tau} \p{\partial_\mu g_{\tau \nu} + \partial_\nu g_{\mu \tau} - \partial_\tau g_{\mu \nu}},
\]
where $\partial_{\mu \nu} \ell := \frac{\partial^2}{\partial \theta^\mu \partial \theta^\nu} \ell\p{\theta,\cdot}$, see e.g. \cite{amariDifferentialGeometryStatistical1987}.\footnote{In particular the connection $\nabla$ is induced by the ambient connection $\bar \nabla_{\partial_\mu \ell} \partial_\nu \ell = \partial_{\mu \nu} \ell + \frac{1}{2} \p{\partial_\mu \ell}\p{\partial_\nu \ell} + \frac{1}{2} g_{\mu \nu}$ over the Hilbert bundle $\bigsqcup_{\theta \in \Theta} H_\theta$ by the orthogonal projection of $H_\theta$ onto $T_\theta$.} In the last display and throughout the rest of this paper we employ Einstein's summation convention on repeated indices. In this context if $V = V^\mu \partial_\mu \ell$ is a vector field its squared norm is given by $\abs{V}^2 = \Var \p{V} = V^\mu V^\nu g_{\mu \nu}$, in the sense that $\abs{V}^2(\theta) = P_\theta\c{V^2}$ at any $\theta \in \Theta$, where $\Var_\theta$ denotes the variance under $P_\theta$. By definition we also have $\nabla f =\p{ \nabla^\mu f} \p{\partial_\mu \ell} = g^{\mu \nu} \p{\partial_\nu f} \partial_\mu \ell = \inner{\nabla f | \nabla \ell(\cdot)}$ for any differentiable function $f$ on $\Theta$. This is consistent as $\nabla f$ is an element of $T_\theta = \Span \b{\partial_\mu \ell : 1 \leq \mu \leq p}$ and the inner product between two vector fields $X$ and $Y$ is the scalar given by the expression $\inner{X|Y} = g_{\mu \nu} X^\mu Y^\nu = g^{\mu \nu} X_\mu Y_\nu$ in local coordinates. In particular for any given sample $x^n = \p{x_1,\ldots,x_n} \in \XX^n$ of size $n \geq 1$, by definition of $\nabla \ell_n$ we have
\[
\frac{1}{n} \nabla \ell_n = g^{\mu \nu} \p{\frac{1}{n} \sum_{i=1}^n \p{\partial_\nu \ell}(x_i) } \partial_\mu \ell = \frac{1}{n} \sum_{i=1}^n g^{\mu \nu} \p{\partial_\nu \ell}(x_i) \partial_\mu \ell = \frac{1}{n} \sum_{i=1}^n \nabla \ell\p{x_i} = \Pro_n \c{\nabla \ell},
\]
as equality between vector fields over $\Theta$, where $\Pro_n = \frac{1}{n} \sum_{i=1}^n \delta_{x_i}$ is the empirical measure.\\

We equip $\Theta$ with a prior $\Pi$, i.e. a measure on $\Theta$, and we assume that $\Pi$ admits a density $\pi$ with respect to the Riemannian volume element $\vol_{\Theta}\p{d\theta} = \sqrt{\abs{g(\theta)}} d\theta$ i.e. Jeffreys' prior (here $\abs{g(\theta)}$ denotes the determinant of the metric tensor $g(\theta)$ in coordinate $\theta$); we refer to such priors as ``absolutely continuous''. Since $\sqrt{\abs{g}}$ is everywhere positive it follows that $\Pi$ also admits a density in coordinates $\theta$ with respect to the Lebesgue measure $d\theta$, that we simply write as $\Pi(\theta)$. Notice that the prior can be improper. We assume that $\pi$ is everywhere positive on $\Theta$ (priors that vanish outside a domain---for instance priors over log-densities that are uniformly bounded by a fixed constant---can also be considered in this framework by replacing $\Theta$ with the support $\Supp \Pi := \b{\theta : \pi\p{\theta} > 0}$ and taking care of the boundary effects). Using Bayes' formula we define the Bayesian posterior $\Pi_n$ as the probability distribution
\[
\Pi_n = \Pi \c{\cdot | X^n} = \frac{e^{\ell_n}}{\Pi \c{e^{\ell_n}}} \cdot \Pi
\]
and $\pi_n = \pi e^{\ell_n}/\Pi\c{e^{\ell_n}}$ its density with respect to the Riemannian volume form. Implicitly we assume that the normalising factor $\Pi \c{e^{\ell_n}}$ is well defined, i.e. $e^{\ell_n} \in L^1\p{\Pi}$, which is the case in all the examples investigated in this paper. In contrast to $\pi$, $\theta \mapsto \Pi\p{\theta}$ is not an invariant scalar (i.e. a function of $\theta$) but rather a scalar density on $\Theta$, in the sense that for two coordinate systems $\theta,\eta$ it satisfies the change of variable formula
\[
\Pi(\eta) = \Pi(\theta) \abs{\frac{\partial \theta}{\partial \eta}}.
\]
The same observation goes for $\Pi_n$. The covector field $\nabla \ln \pi$ can be expressed in local coordinates as
\[
\nabla_\mu \ln \pi = \partial_\mu \ln \Pi - \partial_\mu \ln \sqrt{\abs{g}} = \partial_\mu \ln \Pi - \Gamma_{~\mu \nu}^\nu = \partial_\mu \ln \Pi - g^{\rho \nu} \Gamma_{\rho \mu \nu},
\]
where $\Gamma_{\rho \mu \nu}$ (resp. $\Gamma_{~\mu \nu}^\rho$) are the Christoffel symbols of the first (resp. second) kind of the connection $\nabla$ and where we have used the identity $\partial_\mu \ln \sqrt{\abs{g}} = \Gamma_{~\mu \nu}^\nu$ (see e.g. Chapter $3$ in \cite{carrollSpacetimeGeometryIntroduction2019}, this can easily be proved using Jacobi's formula for the log-derivative of the determinant and the definition of the Christoffel symbols). We use the rules of Ricci calculus in order to raise and lower indices of tensors using the metric, for instance $\nabla \ln \pi$ can also be considered as a vector field with components $\nabla^\mu \ln \pi = g^{\mu \nu} \nabla_\nu \ln \pi$. We refer to \cite{schoutenRiccicalculusIntroductionTensor1954,carrollSpacetimeGeometryIntroduction2019} for an introduction to Ricci calculus in Riemannian and pseudo-Riemannian geometry. We also define the tensor $\nabla^2 \ln \pi = \nabla \nabla \ln \pi$, i.e. the Hessian of $\ln \pi$. For a vector field $V$ we define its divergence $\Div V$ by
\[
\Div V = \nabla_\mu V^\mu = \partial_\mu V^\mu + \Gamma_{~\mu \nu}^\mu V^\nu
\]
and given a positive function $h$ its $h$-divergence $\Div_h V$ by $\Div_h V = h^{-1} \Div\p{hV} = \Div V + \inner{V|\nabla \ln h}$. When $h = \pi$ the expression in local coordinates simplifies to
\[
\Div_\pi V = \Div V + \inner{V|\nabla \ln \pi} = \partial_\mu V^\mu + \b{\Gamma_{~\mu \nu}^\mu + \partial_\nu \ln \pi } V^\nu = \partial_\mu V^\mu + \p{\partial_\mu \ln \Pi}V^\mu,
\]
where $\inner{U|V} := g\p{U,V}$ for any vector fields $U,V$. Crucially if the density $\Pi\p{\cdot}$ of the prior in a set of local coordinates is available then the last display provides a tractable formula for the value of $\Div_\pi V$ that does not require the preliminary calculation of the Christoffel symbols but only of the components $V^\mu$ of $V$ and their partial derivatives; we will use this property in Section \ref{section:histograms}. Similarly for a functional $\Phi$ we define its Laplacian $\Delta \Phi$ by
\[
\Delta \Phi = \Div \nabla \Phi = \Tr \nabla^2 \Phi = g^{\mu \nu} \nabla_{\mu \nu}^2 \Phi
\]
and given a positive function $h$ its $h$-Laplacian by
\[
\Delta_h \Phi = \Div_h \nabla \Phi = h^{-1} \Div\p{h \nabla \Phi} = \Div \nabla \Phi + \inner{\nabla \ln h | \nabla \Phi} = \Delta \Phi + \inner{\nabla \ln h | \nabla \Phi} ,
\]
which in local coordinates and when $h = \pi$ simplifies to $\Delta_\pi \Phi = \partial_\mu \nabla^\mu \Phi + \p{\partial_\mu \ln \Pi}\p{\nabla^\mu \Phi}$. The operators $\Div_\pi$ and $\Delta_\pi$ are well known in analysis over weighted manifolds, see e.g. \cite{grigoryanHeatKernelAnalysis2013}. Finally, we define the precise meaning of BvM theorems that will be used throughout this paper:

\begin{definition}
We say that a sequence of functionals $\Phi_n : \Theta \to \RR$ satisfies a BvM phenomenon with centring $\hat \Psi_{n0}$ (an efficient estimator of the semiparametric functional of interest) and asymptotic variance $\sigma^2 > 0$ if and only if
\begin{equation}\label{equation:bvm_type_theorem}
	d_{BL}\p{\Pi_n \circ Z_n^{-1},\NN\p{0,\sigma^2}} = o_p(1), \quad Z_n := \sqrt{n}\p{\Phi_n - \hat \Psi_{n0}}.
\end{equation}
\end{definition}
Obviously, such a definition finds its usefulness when $\hat \Psi_{n0}$ is itself an estimator which is asymptotically normal around the quantity of interest $\Psi\p{f_0}$ with variance $\sigma^2/n$. Depending on the case, the functional $\Phi$ may or may not be equal to the original one $\Psi$, which is not per se a problem as long as $\Phi$ can be explicitly computed and $\hat \Psi_{n0}$ is still an efficient estimator of $\Psi\p{f_0}$. The statistician will then simply base its inference about $\Psi\p{f_0}$ on the posterior credible intervals of $\Phi(f)$ and not the ones of $\Psi(f)$. This idea of corrected functionals for semiparametric and nonparametric inference has a precedent in the Bayesian literature, for instance in \cite{castilloBernsteinVonMisesTheorem2015} where a bias-corrected quadratic functional in the white noise model is considered, or in \cite{castilloMultiscaleBayesianSurvival2021} where a projected posterior is used for weak nonparametric inference in a survival analysis model. See also \cite{yiuSemiparametricPosteriorCorrections2025} for an alternative technique that relies on a Bayesian analog of one-step estimators for efficient semiparametric inference.

We make the frequentist assumption $X^n = \p{X_1,\ldots,X_n} \iid \Pro_0$ for some probability distribution $\Pro_0$ on $\p{\XX,\AA}$ (that need not belong to the model $\MM$), and we study the behaviour of the posterior distribution $\Pi_n \c{d\theta} \propto e^{\ell_n\p{\theta,X^n}} \Pi\c{d\theta}$ in the frequentist sense, i.e. under $X^n \sim \Pro_0^n$. All $o_p$ or $O_p$ terms are understood with respect to $X^n \sim \Pro_0^n$. Finally, for a measure $\Xi$ and a nonnegative integrable function $\chi$ for which $\Xi \c{\chi} > 0$, we define $\Xi^\chi$ as the probability distribution $A \mapsto \Xi^\chi\c{A} = \Xi\c{\chi \ind{A}}/\Xi\c{\chi}$. In the important case where $\Xi = \Pi \c{\cdot | X^n}$ is a Bayesian posterior distribution, notice that $\Pi^\chi\c{\cdot | X^n} = \Pi \c{\cdot | X^n}^\chi$. If $\chi = \ind{B}$ for some measurable set $B$ we simply write $\Xi^B$ for $\Xi^{\ind{B}}$.

\subsection{Bernstein-von Mises theorems from Stein's method}\label{section:general_bvm_stein}

We now present the main abstract results of this paper. We first give an integration-by-parts argument motivating the sufficient conditions of our main result \ref{corollary:bvm_without_bias}. This argument can be formulated in terms of posterior Laplace transforms and provides a useful connection with classical proofs of semiparametric BvM theorems. It is, however, only motivational here: the BvM results of this section are proved directly using Stein's method, which additionally yields quantitative bounds on the quality of the Gaussian approximation. Additionally, a rigorous version of the arguments presented here is used later for proving posterior contraction results (see Lemma \ref{lemma:laplace_transform}).

Throughout this section, the functional $\Phi$, the estimator $\hat\Psi_0$ and the vector field $V$ may implicitly depend on $n$. We suppress this dependence from the notation for readability.

\paragraph{Heuristic derivation.} Consider the posterior cumulant generating function
\[
g\p{t} := \ln \Pi_n \c{\exp\p{t\sqrt{n}\p{\Phi - \hat \Psi_0}}}, \quad t \in \RR.
\]
Then at least formally we have
\[
g'\p{t} = \Pi_{n,t} \c{\sqrt{n}\p{\Phi - \hat \Psi_0}}, \quad \frac{d\Pi_{n,t}}{d\Pi_n} \propto \exp\p{t\sqrt{n}\p{\Phi - \hat \Psi_0}}.
\]
Let $V$ be a $\CC^1$ vector field over $\Theta$ and define the data-dependent functional $\hat\Phi := \Phi+n^{-1}\inner{V|\nabla\ell_n}$. The choice of vector field $V$ is specified in the applications below and will always be either $\nabla \Phi$ or a perturbation thereof. For now, it should be viewed as an infinitesimal direction along which the likelihood and the functional can be differentiated in order to relate the posterior distribution of $\Phi$ to the local geometry of the statistical model. Then we have
\[
\Div_{\pi_{n,t}} V = \Div_\pi V + \inner{V | \nabla \ell_n} + \inner{V | \nabla t\sqrt{n}\p{\Phi - \hat \Psi_0}} = \Div_\pi V + \inner{V | \nabla \ell_n} + t\sqrt{n}\inner{V | \nabla \Phi},
\]
so that
\begin{align*}
	g'\p{t} & = \sqrt{n} \Pi_{n,t} \c{\hat \Phi - \frac{1}{n} \inner{V | \nabla \ell_n} - \hat \Psi_0} \\
	& = t \Pi_{n,t} \c{\inner{V | \nabla \Phi}} + \sqrt{n}\Pi_{n,t} \c{\hat \Phi - \hat \Psi_0 + \frac{1}{n} \Div_\pi V} - \frac{1}{\sqrt{n}} \Pi_{n,t} \c{\Div_{\pi_{n,t}}V}.
\end{align*}
But using the divergence theorem, under suitable growth assumptions on $V$ we have
\[
\Pi_{n,t} \c{\Div_{\pi_{n,t}}V} = \int \p{\Div_{\pi_{n,t}}V}(\theta) \pi_{n,t}(\theta) \vol_\Theta \p{d\theta} = \int \p{\Div \p{\pi_{n,t} V}}(\theta) \vol_\Theta \p{d\theta} = 0.
\]
As a result if
\begin{equation}\label{equation:heuristic_condition}
	\sup_\Theta \sqrt{n}\abs{\hat \Phi - \hat \Psi_0 + \frac{1}{n} \Div_\pi V} + \abs{\sigma^2 - \inner{V | \nabla \Phi}} = o_p(1)
\end{equation}
then we obtain $\sup_t \abs{g'\p{t} - \sigma^2 t} = o_p(1)$, which yields $g\p{t} = \sigma^2 t^2/2 + o_p(1)$ for every $t$ by integration and therefore implies the BvM theorem as in \cite{castilloBernsteinVonMisesTheorem2015}. The derivation above is made rigourous in Lemma \ref{lemma:laplace_transform}, which we use in our proofs of posterior contraction by applying it to sufficently many linear functionals, see also \cite{castilloBayesianSupremumNorm2014} and \cite{castilloNonparametricBernsteinMises2013,castilloBernsteinVonMisesPhenomenon2014,castilloMultiscaleBayesianSurvival2021} for related ideas. When Condition (\ref{equation:heuristic_condition}) holds only on a posterior contraction set $B_n \subsetneq \Theta$, the argument must be localised. This produces an additional boundary contribution in the analysis which does not appear to be easy to control using this technique. Nevertheless it is possible to prove a quantitative, non-asymptotic BvM theorem using Stein's method combined with localisation by multiplication against a smooth cutoff function (instead of a hard thresholding on a posterior contraction set), which relies on the same condition (\ref{equation:heuristic_condition}), with two additional ones related to the localisation penalty.

\paragraph{Formal argument.} We now formulate a rigorous proof. For any probability distributions $\mu$ and $\nu$ over $\RR$ define their bounded Lipschitz distance as
\[
d_{BL}\p{\mu,\nu} := \sup_{f \in \Lip_1} \abs{\int fd\p{\mu - \nu}}, \quad \Lip_1 := \b{f \in \CC_b\p{\RR,\c{-1,1}}, \quad \sup_{x \neq y} \frac{\abs{f(x) - f(y)}}{\abs{x-y}} \leq 1},
\]
and for any probability distributions $\mu,\nu$ over a measurable set $\p{\Omega,\AA}$ define their total variation distance as
\[
d_{TV}\p{\mu,\nu} := \sup_{A \in \AA} \abs{\p{\mu - \nu}\c{A}}.
\]
Notice that $d_{BL}(\mu,\nu) \leq 2d_{TV}(\mu,\nu)$ for any probability distributions $\mu,\nu$ over $\RR$.\footnote{This follows by the triangle inequality using the classical representation of $d_{TV}\p{\mu,\nu}$ as $\frac{1}{2} \int \abs{\frac{d\mu}{d\p{\mu + \nu}} - \frac{d\nu}{d\p{\mu + \nu}}} d\p{\mu + \nu}$.} For any smooth compactly supported cutoff function $\chi_n : \Theta \to \c{0,1}$, using $d_{TV}\p{\Pi_n^{\chi_n},\Pi_n} \leq \Pi_n\c{1-\chi_n}$\footnote{Exchanging the roles of $A$ and $A^c$ we have $d_{TV}\p{\Pi_n^{\chi_n},\Pi_n} = \sup_A \abs{\Pi_n\c{A} - \Pi_n^{\chi_n}\c{A}} = \sup_A \Pi_n \c{A} - \Pi_n^{\chi_n}\c{A} = \sup_A \Pi_n \c{\p{1-\chi_n}\ind{A}} - \p{1 - \Pi_n \c{\chi_n}}\Pi_n^{\chi_n} \c{A} \leq \Pi_n\c{1-\chi_n}$} and the data processing inequality $d_{TV}\p{\Pi_n^{\chi_n} \circ Z_n^{-1},\Pi_n \circ Z_n^{-1}} \leq d_{TV}\p{\Pi_n^{\chi_n},\Pi_n}$ we have
\begin{align*}
d_{BL}\p{\Pi_n \circ Z_n^{-1},\NN\p{0,\sigma^2}} & \leq d_{BL}\p{\Pi_n^{\chi_n} \circ Z_n^{-1},\NN\p{0,\sigma^2}} + d_{BL}\p{\Pi_n^{\chi_n} \circ Z_n^{-1},\Pi_n \circ Z_n^{-1}} \\
& \leq d_{BL}\p{\Pi_n^{\chi_n} \circ Z_n^{-1},\NN\p{0,\sigma^2}} + 2d_{TV}\p{\Pi_n^{\chi_n},\Pi_n} \\
& \leq d_{BL}\p{\Pi_n^{\chi_n} \circ Z_n^{-1},\NN\p{0,\sigma^2}} + 2\Pi_n \c{1-\chi_n}.
\end{align*}
By definition of the bounded Lipschitz metric we have
\[
d_{BL}\p{\Pi_n^{\chi_n} \circ Z_n^{-1}, \NN\p{0,\sigma^2}} = \sup_{h \in \Lip_1} \abs{\Pi_n^{\chi_n} \c{h\p{Z_n}} - \NN\p{0,\sigma^2}\c{h}}.
\]
Now, following Stein's method \cite{steinBoundErrorNormal1972,rossFundamentalsSteinsMethod2011}, for any $h \in \Lip_1$ let $f_h : \RR \to \RR$ be the function
\[
f_h : x \mapsto - e^{\frac{x^2}{2\sigma^2}} \int_x^\infty e^{-\frac{s^2}{2\sigma^2}} \p{h(s) - \NN\p{0,\sigma^2}\c{h}} ds,
\]
which satisfies the Stein's equation
\[
\sigma^2 f_h'\p{x} - x f_h(x) = \sigma^2 \p{h(x) - \NN\p{0,\sigma^2}\c{h}}, \quad x \in \RR.
\]
Then
\begin{align*}
	& d_{BL}\p{\Pi_n^{\chi_n} \circ Z_n^{-1}, \NN\p{0,\sigma^2}} = \sup_{h \in \Lip_1} \abs{\Pi_n^{\chi_n} \c{h\p{Z_n}} - \NN\p{0,\sigma^2}\c{h}} = \sup_{h \in \Lip_1} \abs{\Pi_n^{\chi_n} \c{h\p{Z_n} - \NN\p{0,\sigma^2}\c{h}}} \\
	& = \sigma^{-2} \sup_{h \in \Lip_1} \abs{\Pi_n^{\chi_n} \c{\sigma^2 f_h'\p{Z_n} - Z_n f_h\p{Z_n}}} = \sigma^{-2} \sup_{h \in \Lip_1} \frac{\abs{\Pi_n \c{\chi_n\p{\sigma^2 f_h'\p{Z_n} - Z_n f_h\p{Z_n}}}}}{\Pi_n \c{\chi_n}}.
\end{align*}
The numerator in the supremum of the last display can be expressed as
\begin{align*}
	\Pi_n \c{\chi_n \p{\sigma^2 f_h'\p{Z_n} - Z_nf_h\p{Z_n}}} & = \Pi_n \c{\chi_n \p{f_h'\p{Z_n}\p{\sigma^2 - \inner{\nabla \Phi | V}} + f_h'\p{Z_n} \inner{\nabla \Phi | V} - Z_nf_h\p{Z_n}}} \\
	& = \Pi_n \c{\chi_n \p{f_h'\p{Z_n}\p{\sigma^2 - \inner{\nabla \Phi | V}} - \p{Z_n + \frac{1}{\sqrt{n}}\Div_{\pi_n} V}f_h\p{Z_n}}} \\
	& + \Pi_n\c{\chi_n \p{f_h'\p{Z_n} \inner{\nabla \Phi | V} + f_h\p{Z_n}\frac{1}{\sqrt{n}}\Div_{\pi_n} V}},
\end{align*}
and since
\[
Z_n + \frac{1}{\sqrt{n}}\Div_{\pi_n} V = \sqrt{n}\p{\Phi - \hat \Psi_0 + \frac{1}{n} \Div_\pi V + \frac{1}{n} \inner{V|\nabla \ell_n}} = \sqrt{n}\p{\hat \Phi - \hat \Psi_0 + \frac{1}{n} \Div_\pi V},
\]
we obtain
\begin{align*}
	\Pi_n \c{\chi_n \p{\sigma^2 f_h'\p{Z_n} - Z_nf_h\p{Z_n}}} & = \Pi_n \c{\chi_n \p{f_h'\p{Z_n}\p{\sigma^2 - \inner{\nabla \Phi | V}} - \sqrt{n}\p{\hat \Phi - \hat \Psi_0 + \frac{1}{n} \Div_\pi V}f_h\p{Z_n}}} \\
	& + \Pi_n\c{\chi_n \p{f_h'\p{Z_n} \inner{\nabla \Phi | V} + f_h\p{Z_n}\frac{1}{\sqrt{n}}\Div_{\pi_n} V}}.
\end{align*}
Moreover because
\begin{align*}
	\Div_{\pi_n} \p{\chi_n f_h\p{Z_n} V} & = \chi_n f_h\p{Z_n} \Div_{\pi_n} V + \inner{\nabla \p{\chi_n f_h\p{Z_n}} | V} \\
	& = \chi_n f_h\p{Z_n} \Div_{\pi_n} V + \sqrt{n} \chi_n f_h'\p{Z_n} \inner{\nabla \Phi | V} + f_h\p{Z_n} \inner{\nabla \chi_n | V},
\end{align*}
by Stokes' theorem (which applies since $\chi$ has compact support) we obtain
\begin{align*}
	0 & = \int \Div \p{\pi_n \chi_n f_h\p{Z_n} V} = \Pi_n \c{\Div_{\pi_n}\p{\chi_n f_h\p{Z_n} V}} \\
	& = \Pi_n \c{\chi_n f_h\p{Z_n} \Div_{\pi_n} V + \sqrt{n} \chi_n f_h'\p{Z_n} \inner{\nabla \Phi | V} + f_h\p{Z_n} \inner{\nabla \chi_n | V}}.
\end{align*}
Therefore
\[
\Pi_n\c{\chi_n \p{f_h'\p{Z_n} \inner{\nabla \Phi | V} + f_h\p{Z_n}\frac{1}{\sqrt{n}}\Div_{\pi_n} V}} = - \Pi_n \c{f_h\p{Z_n} \frac{1}{\sqrt{n}}\inner{\nabla \chi_n | V}},
\]
which yields
\begin{align*}
	& \Pi_n \c{\chi_n \p{\sigma^2 f_h'\p{Z_n} - Z_nf_h\p{Z_n}}} \\
	& = \Pi_n \c{\chi_n \p{f_h'\p{Z_n}\p{\sigma^2 - \inner{\nabla \Phi | V}} - \sqrt{n}\p{\hat \Phi - \hat \Psi_0 + \frac{1}{n} \Div_\pi V}f_h\p{Z_n}} - f_h\p{Z_n} \frac{1}{\sqrt{n}}\inner{\nabla \chi_n | V}}.
\end{align*}
Finally the function $f_h$ satisfies (see e.g. \cite{rossFundamentalsSteinsMethod2011})
\[
\norm{f_h}_\infty \leq \sigma \sqrt{\frac{\pi}{2}} \norm{h - \NN\p{0,\sigma^2}\c{h}}_\infty \leq \sigma\sqrt{2\pi}, \quad \norm{f_h'}_\infty \leq 2\norm{h - \NN\p{0,\sigma^2}\c{h}}_\infty \leq 4,
\]
which implies the following fundamental estimate by the triangle inequality, taking the supremum over all $h \in \Lip_1$:

\begin{lemma}\label{lemma:stein}
If $\sigma>0$, $\Phi : \Theta \to \RR$ is a $\CC^1$ functional, $V$ is a $\CC^1$ vector field over $\Theta$, $\chi_n : \Theta \to \c{0,1}$ is a sequence of $\CC^1$ functions over $\Theta$ with compact supports, $\hat \Psi_0$ is an estimator (i.e. real valued and $X^n$-measurable random variable) and $Z_n = \sqrt{n}\p{\Phi - \hat \Psi_0}$ we have
\begin{align*}
& d_{BL}\p{\Pi_n \circ Z_n^{-1}, \NN\p{0,\sigma^2}} \leq \sigma^{-1} \sqrt{2n\pi} \Pi_n^{\chi_n} \c{\abs{\hat \Phi - \hat \Psi_0 + \frac{1}{n} \Div_\pi V}} + 4\sigma^{-2}\Pi_n^{\chi_n} \c{\abs{\sigma^2 - \inner{\nabla \Phi | V}}} \\
& + \sigma^{-1} \sqrt{\frac{2\pi}{n}} \frac{\Pi_n \c{\abs{\inner{\nabla \chi_n | V}}}}{\Pi_n\c{\chi_n}} + 2\Pi_n \c{1-\chi_n},
\end{align*}
where $\hat \Phi$ is the random functional $\hat \Phi = \Phi + n^{-1} \inner{V | \nabla \ell_n}$.
\end{lemma}

As we can see, Lemma \ref{lemma:stein} provides a quantitative upper bound on the bounded Lipschitz distance in terms of the same analytical quantities introduced heuristically above (Condition (\ref{equation:heuristic_condition})). It directly implies the following theorem using the triangle inequality. Its first item gives sufficient conditions for a BvM theorem for the functional $\Phi$ centred at $\hat \Psi_0$. Moreover, the form of the condition $(a)$ suggests a natural correction when it fails. Indeed, even when $\hat \Phi-\hat \Psi_0=o_p(n^{-1/2})$ uniformly on the posterior concentration set, the term $n^{-1}\Div_\pi V$ may prevent $(a)$ from holding. In that case it is natural to consider instead the ``divergence corrected'' functional $\tilde \Phi := \Phi-n^{-1}\Div_\pi V$. Applying the same Stein bound to $\tilde \Phi$ removes $\Div_\pi V$ from $(a)$, which is then somehow ``transferred'' to $(b)$. This leads to the second item of the theorem below. As the applications in Section \ref{section:applications} and the simulations in Section \ref{section:simulation_study} will show, the conditions of the second item of Theorem \ref{theorem:generic_bvm} can hold in regimes where the ones of the first item may fail, so that the divergence term simultaneously identifies a possible obstruction for the BvM to hold for the functional $\Phi$ and suggests a corresponding prior-dependent correction to the functional.

\begin{theorem}\label{theorem:generic_bvm}
If $\sigma>0$, $\Phi : \Theta \to \RR$ is a $\CC^1$ functional, $\hat \Psi_0$ is an estimator, $V$ is a $\CC^1$ vector field, $\pi$ is $\CC^1$ and $\chi_n : \Theta \to \c{0,1}$ is a sequence of compactly supported $\CC^1$-functions such that $A_0 = 1 - \Pi_n\c{\chi_n} = o_p(1)$ and $A_1 := \Pi_n \c{\abs{\inner{\nabla \chi_n | V}}} = o_p\p{\sqrt{n}}$ then:
\begin{enumerate}
	\item if
	\begin{enumerate}
		\item $A_2 := \sup_{\Supp \p{\chi_n}} \abs{\hat \Phi - \hat \Psi_0 + n^{-1} \Div_\pi V} = o_p\p{n^{-1/2}}$ where $\hat \Phi = \Phi + n^{-1}\inner{V | \nabla \ell_n}$, 
		\item $A_3 := \sup_{\Supp \p{\chi_n}} \abs{\sigma^2 - \inner{\nabla \Phi | V}} = o_p(1)$,
	\end{enumerate}
	then with $Z_n = \sqrt{n}\p{\Phi - \hat \Psi_0}$ we have
	\[
	d_{BL}\p{\Pi_n \circ Z_n^{-1}, \NN\p{0,\sigma^2}} \leq 4\p{A_0 + \sigma^{-1} \frac{A_1}{\sqrt{n}\p{1-A_0}} + \sigma^{-1}\sqrt{n}A_2 + \sigma^{-2} A_3} = o_p(1),
	\]
	\item if in addition $\Div_\pi V \in \CC^1$
	\begin{enumerate}
		\item $A_2 := \sup_{\Supp \p{\chi_n}} \abs{\hat \Phi - \hat \Psi_0} = o_p\p{n^{-1/2}}$ where $\hat \Phi = \Phi + n^{-1}\inner{V | \nabla \ell_n}$,
		\item $A_3 := \sup_{\Supp \p{\chi_n}} \abs{\sigma^2 - \inner{\nabla \Phi | V} + n^{-1} \inner{\nabla \Div_\pi V|V}} = o_p(1)$,
	\end{enumerate}
	then with $Z_n = \sqrt{n}\p{\Phi - n^{-1} \Div_\pi V - \hat \Psi_0}$ we have
	\[
	d_{BL}\p{\Pi_n \circ Z_n^{-1}, \NN\p{0,\sigma^2}} \leq 4\p{A_0 + \sigma^{-1} \frac{A_1}{\sqrt{n}\p{1-A_0}} + \sigma^{-1} \sqrt{n}A_2 + \sigma^{-2} A_3} = o_p(1).
	\]
\end{enumerate}
\end{theorem}

\begin{remark}
The upper bound obtained in Theorem \ref{theorem:generic_bvm} is non-asymptotic and in the examples considered in this paper the convergence in probability to $0$ may be strengthened into an almost sure convergence using Borel-Cantelli's lemma (for every example considered here, for any $K>0$ the bounds on the $d_{BL}$-distances hold non-asymptotically with probability $N^{-K}$ up to a multiplicative constant that depends on $K$).
\end{remark}

The ``divergence correction'' step $\Phi \leftrightarrow \Phi - n^{-1} \Div_\pi V$ can also be motivated by a direct analysis of the posterior bias. Indeed, if a vector field $V$ and a functional $\Phi$ are such that $\sup_\Theta \abs{\hat \Phi - \hat \Psi_0} = o_p\p{n^{-1/2}}$ (as will actually be the case for the linear and quadratic functionals considered in Section \ref{section:applications}, where $\hat \Phi = \hat \Psi_0$ identically over $\Theta$) then
\begin{align*}
\Pi_n \c{\Phi - n^{-1} \Div_\pi V} - \hat \Psi_0 & = \Pi_n \c{ - n^{-1} \inner{V | \nabla \ell_n} - n^{-1} \Div_\pi V} + o_p\p{n^{-1/2}} = -n^{-1} \Pi_n \c{\Div_{\pi_n} V} + o_p\p{n^{-1/2}} \\
& = o_p\p{n^{-1/2}},
\end{align*}
where the identity $\Pi_n \c{\Div_{\pi_n} V} = 0$ can be justified by the divergence theorem using localisation by compactly supported functions combined with an approximation argument (as done for instance in the proof of Lemma \ref{lemma:laplace_transform}). Thus considering the divergence corrected functional $\Phi - n^{-1} \Div_\pi V$ instead of $\Phi$ allows in that case to remove every prior induced asymptotic bias at the relevant scale $n^{-1/2}$. In the application to the quadratic functional in Section \ref{section:white_noise} we discuss how such a correction step offers an alternative to the approach suggested in \cite{castilloBernsteinVonMisesTheorem2015} that is valid in a wider range of regimes by considering non-constant and prior-dependent correction terms.

For proving BvM theorems without functional correction, i.e. when $\Phi = \Psi$, the choice $V = \nabla \Psi$ appears to suffice for the examples considered in this paper. In that particular case we obtain the following formulation:

\begin{corollary}\label{corollary:bvm_without_bias}
If $\sigma>0$, $\Phi : \Theta \to \RR$ is a $\CC^1$ functional, $\pi$ is $\CC^1$, $\hat \Psi_0$ is an estimator and $\chi_n : \Theta \to \c{0,1}$ is a sequence of compactly supported $\CC^1$-functions such that $A_0 = 1 - \Pi_n\c{\chi_n} = o_p(1)$ and $A_1 := \Pi_n \c{\abs{\inner{\nabla \chi_n | \nabla \Psi}}} = o_p\p{\sqrt{n}}$ and
\begin{enumerate}
	\item $A_2 := \sup_{\Supp \p{\chi_n}} \abs{\hat \Psi - \hat \Psi_0 + n^{-1} \Delta_\pi \Psi} = o_p\p{n^{-1/2}}$ where $\hat \Psi = \Psi + n^{-1} \inner{\nabla \Psi | \nabla \ell_n}$, 
	\item $A_3 := \sup_{\Supp \p{\chi_n}} \abs{\sigma^2 - \abs{\nabla \Psi}^2} = o_p(1)$,
\end{enumerate}
then with $Z_n = \sqrt{n}\p{\Psi - \hat \Psi_0}$ we have
\[
d_{BL}\p{\Pi_n \circ Z_n^{-1}, \NN\p{0,\sigma^2}} \leq 4\p{A_0 + \sigma^{-1} \frac{A_1}{\sqrt{n}\p{1-A_0}} + \sigma^{-1}\sqrt{n}A_2 + \sigma^{-2} A_3} = o_p(1).
\]
\end{corollary}

The next section presents applications of Theorem \ref{theorem:generic_bvm} and Corollary \ref{corollary:bvm_without_bias} to specific models and functionals.

\section{Applications}\label{section:applications}

\subsection{White noise model with independent priors}\label{section:white_noise}

As a first application of Theorem \ref{theorem:generic_bvm} for which the geometry is trivial we consider the case of an observation in the Gaussian sequence model

\begin{equation}\label{equation:gaussian_sequence_model}
Y^\mu = \theta_0^\mu + N^{-1/2} \varepsilon^\mu, \quad \varepsilon^\mu \iid \NN\p{0,1}, \quad \mu \geq 1,
\end{equation}
for some $N \geq 1$ and an unknown $\theta_0 \in \ell^2 = \ell^2\p{\mathbb N^*}$. Let $p \geq 1$ and
\[
\pi_p : \eta = \p{\eta^\mu}_{\mu \geq 1} \in \ell^2 \mapsto \pi_p\c{\eta} = \p{\eta^1,\ldots,\eta^p} \in \Theta := \RR^p
\]
be the orthogonal projection on the first $p$ coordinates. For any $\theta \in \Theta$ we then define
\[
P_\theta := \bigotimes_{\mu \geq 1} \NN\p{\pi_p^* \c{\theta}^\mu,N^{-1}},
\]
where $\pi_p^*\c{\theta} = \p{\theta^1,\ldots,\theta^p,0,\ldots} \in \ell^2$, i.e. the law of the observation
\[
Y^\mu = \begin{cases}
	\theta^\mu + N^{-1/2} \varepsilon^\mu, \quad 1 \leq \mu \leq p, \\
	N^{-1/2} \varepsilon^\mu, \quad \mu > p.
\end{cases}
\]
Notice that we only observe $n = 1$ observations even if $N$ has a natural sample size interpretation, and that Model (\ref{equation:gaussian_sequence_model}) is equivalent to the continuous white noise model on $\c{0,1}$ by a choice of orthonormal basis of $L^2\c{0,1}$, see e.g. Chapter $1$ in \cite{gineMathematicalFoundationsInfiniteDimensional2015}. Moreover following Chapter $6$ in \cite{gineMathematicalFoundationsInfiniteDimensional2015} the distributions $\p{P_\theta}_{\theta \in \Theta}$ are all absolutely continuous with respect to $P_0$ and satisfy, for any $\theta \in \Theta$,
\[
p_\theta\p{Y} = \frac{dP_\theta}{dP_0}\p{Y} = \exp\p{N\inner{\pi_p\c{Y}|\theta}_2  - \frac{N}{2} \norm{\theta}_2^2}, \quad P_0-\text{a.s.}
\]
From this we immediately deduce that $g_{\mu \nu} = N\delta_{\mu \nu}$ and $\nabla^\mu \ell = Y^\mu - \theta^\mu$. In particular the coordinates $\theta^\mu$ are affine for $\nabla$, in the sense that $\Gamma_{~\mu\nu}^\rho \p{\theta} = 0$. Throughout this section we assume that the unknown signal $\theta_0$ belongs to the sequence Sobolev space $H^s$ defined as
\[
H^s := \b{\theta \in \ell^2 : \sum_{\mu \geq 1} \mu^{2s} \theta^{\mu2} < +\infty}, \quad s \geq 0.
\]
We consider two types of priors over $\theta$ :

\begin{prior}{\label{prior:independent_series_white_noise}}
The prior $\Pi$ is of the form $\Pi \c{d\theta} = \prod_{\mu = 1}^p \sigma_\mu^{-1} \exp\p{-H\p{\theta^\mu/\sigma_\mu}} d\theta^\mu$ for some $\sigma \in \p{0,\infty}^p, \quad 1 \ll p \ll N/\ln N$ and differentiable function $H : \RR \to \RR$ with $\norm{H'}_\infty <+\infty$.
\end{prior}

\begin{prior}\label{prior:gaussian}
The prior $\Pi$ is of the form $\Pi = \NN\p{0,\Diag\p{\sigma^2}}$ for some $\sigma \in \p{0,\infty}^p$.
\end{prior}

\begin{remark}
Using the special independence structure of the Gaussian sequence model we could extend the analysis to independent priors of the form \ref{prior:independent_series_white_noise} but with non Lipschitz potentials $H$ (e.g. the ones considered in \cite{castilloNonparametricBernsteinMises2013}) but we abstain from such refinements in the present paper.
\end{remark}

\paragraph{Squared norm functional.} We consider the squared norm functional $\Psi(\theta_0) = \norm{\theta_0}_2^2$. When $\theta_0 \in H^s$ for some $s > \frac{1}{4}$ and $N^{\frac{1}{4s}} \ll p \ll N$ the estimator $\hat \Psi_0 = \sum_{\mu = 1}^p \c{Y^{\mu 2} - \frac{1}{N}}$ satisfies
\[
\sqrt{N}\p{\hat \Psi_0 - \norm{\theta_0}_2^2} \xrightarrow[N \to \infty]{(d)} \NN\p{0,4\norm{\theta_0}_2^2}
\]
and the asymptotic variance $\sigma^2 = 4\norm{\theta_0}_2^2$ is optimal in a semiparametric sense, see e.g. \cite{caiOptimalAdaptiveEstimation2006,castilloBernsteinVonMisesTheorem2015}. We now present an application of the second item of Theorem \ref{theorem:generic_bvm} to this setting, compare the results to those of \cite{castilloBernsteinVonMisesTheorem2015} and also discuss what Corollary \ref{corollary:bvm_without_bias} (BvM without functional correction) would imply instead.

\begin{corollary}\label{corollary:squared_norm_functional_white_noise}
Under Model (\ref{equation:gaussian_sequence_model}) and if $p \gg N^{\frac{1}{4s}}$:
\begin{enumerate}
	\item if $\Pi$ is a prior of type \ref{prior:independent_series_white_noise} such that $\norm{H''}_\infty < +\infty$ and $\sigma_- = \min_{1 \leq \mu \leq p} \sigma_\mu \gg N^{-1/2}$ then
	\[
	d_{BL}\p{\Pi_1 \circ Z_N^{-1}, \NN\p{0,\sigma^2}} = O_p\p{p^{-s} + \sqrt{\frac{p \ln N}{N}} + \frac{1}{N\sigma_-^2} + \frac{\sqrt{p}}{N\sigma_-}} = o_p(1)
	\]
	\[
	\text{where }Z_N := \sqrt{N}\p{\sum_{\mu = 1}^p \b{\theta^{\mu 2} + N^{-1} \sigma_\mu^{-1} H'\p{\theta^\mu/\sigma_\mu}\p{Y^\mu + \theta^\mu}} - \frac{2p}{N} - \hat \Psi_0},
	\]
	and in particular if in addition $s > 1/4$ and $N^{\frac{1}{4s}} \ll p \ll N/\ln N$ then the BvM phenomenon with centring $\hat \Psi_0$ holds for the functional
	\[
	\theta \mapsto \sum_{\mu = 1}^p \b{\theta^{\mu 2} + N^{-1} \sigma_\mu^{-1} H'\p{\theta^\mu/\sigma_\mu}\p{Y^\mu + \theta^\mu}} - \frac{2p}{N}.
	\]
	\item If $\Pi$ is a prior of type \ref{prior:gaussian} with $\sigma_\mu \geq c\p{N^{-1/2} \vee \abs{\theta_0^\mu}}$ for a fixed $c>0$ and $p \ll N/\ln N$ then
	\[
	d_{BL}\p{\Pi_1 \circ Z_N^{-1}, \NN\p{0,\sigma^2}} = O_p\p{p^{-s} + \sqrt{\frac{p \ln N}{N}}} = o_p(1),
	\]
	where
	\[
	Z_N := \sqrt{N}\p{\sum_{\mu = 1}^p \b{\theta^{\mu 2} + N^{-1} \sigma_\mu^{-2} \theta^\mu\p{Y^\mu + \theta^\mu}} - \frac{2p}{N} - \hat \Psi_0},
	\]
	and in particular if in addition $s>1/4$ and $N^{\frac{1}{4s}} \ll p \ll N/\ln N$ then the BvM phenomenon with centring $\hat \Psi_0$ holds for the functional
	\[
	\theta \mapsto \sum_{\mu = 1}^p \b{\theta^{\mu 2} + N^{-1} \sigma_\mu^{-2} \theta^\mu\p{Y^\mu + \theta^\mu}} - \frac{2p}{N}.
	\]
\end{enumerate}
\end{corollary}

In the case of Gaussian priors it is instructive to compare the results of Corollary \ref{corollary:squared_norm_functional_white_noise} to those of Theorem $(3.1)$ in \cite{castilloBernsteinVonMisesTheorem2015}: there it is stated that in the case of Gaussian priors \ref{prior:gaussian} a BvM phenomenon holds for the functional $\theta \mapsto \norm{\theta}_2^2 - 2p/N$ and with centring $\hat \Psi_0$ whenever $\sigma_\mu \geq c\p{N^{-1/2} \vee \abs{\theta_0^\mu}}, \quad p = \floor{N/\ln N}$ and
\begin{equation}\label{equation:incorrect_condition_CR15}
\sum_{\mu = 1}^p \sigma_\mu^{-2} \ll N^{3/2}.
\end{equation}
However we prove at the end of Section \ref{section:proofs_white_noise} that this is not exactly correct as stated, probably due to a minor bookkeeping issue in the calculations: under the only assumptions $s>1/4, \quad \sigma_\mu \geq c\p{N^{-1/2} \vee \abs{\theta_0^\mu}}$ and $p = \floor{N/\ln N}$ we have
\[
d_{BL}\p{\Pi_1 \circ Z_1^{-1}, \NN\p{0,4\norm{\theta_0}_2^2}} = o_p(1) \text{ where } Z_1 := \sqrt{N}\p{\norm{\theta}_2^2 - \frac{2p}{N} + \frac{1}{2\sqrt{N}} \sum_{\mu = 1}^p \b{A_\mu + \theta_0^{\mu 2} B_\mu} - \hat \Psi_0},
\]
and
\[
A_\mu := \frac{2\p{3N\sigma_\mu^2 + 2}}{\sqrt{N}\p{N\sigma_\mu^2 + 1}^2} \asymp \frac{\sigma_\mu^{-2}}{N^{3/2}}, \quad B_\mu := \frac{2}{\sqrt{N}} \frac{2\sigma_\mu^2 + N^{-1}}{\p{\sigma_\mu^2 + N^{-1}}^2} \asymp \frac{\sigma_\mu^{-2}}{\sqrt{N}}.
\]
As a result, under the assumption $\sigma_\mu \geq c\p{N^{-1/2} \vee \abs{\theta_0^\mu}}, \quad p = \floor{N/\ln N}$ and the fact that both bias terms are positive, the BvM phenomenon with centring $\hat \Psi_0$ holds for the functional $\theta \mapsto \norm{\theta}_2^2 - 2p/N$ if and only if
\begin{equation}\label{equation:correct_condition_CR15}
\sum_{\mu = 1}^p \sigma_\mu^{-2} \ll N^{3/2} \text{ and } \sum_{\mu = 1}^p \sigma_\mu^{-2} \theta_0^{\mu 2} \ll N^{1/2}.
\end{equation}
With the same proof one can also formulate a variant of this result: under the assumptions
\[
\sigma_\mu \geq c\p{N^{-1/2} \vee \abs{\theta_0^\mu}}, \quad p = \floor{N/\ln N} \text{ and } \sum_{\mu = 1}^p \sigma_\mu^{-2} \theta_0^{\mu 2} \ll N^{1/2}
\]
then a BvM phenomenon with centring $\hat \Psi_0$ holds for the functional
\[
\theta \mapsto \norm{\theta}_2^2 - 2p/N + \p{2\sqrt{N}}^{-1} \sum_{\mu = 1}^p A_\mu
\]
(which is tractable since it does not depend on the truth $\theta_0$) without assuming that $\sum_{\mu = 1}^p \sigma_\mu^{-2} \ll N^{3/2}$. This is the corrected functional that we compare against in Section \ref{section:simulation_study}. Condition (\ref{equation:correct_condition_CR15}) is interestingly consistent with the predictions of Corollary \ref{corollary:squared_norm_functional_white_noise}: indeed one can prove that
\[
\sup_{\theta \in A_N} \abs{\sum_{\mu = 1}^p N^{-1} \sigma_\mu^{-2} \theta^\mu\p{Y^\mu + \theta^\mu}} = O_p\p{\sum_{\mu = 1}^p \frac{\sigma_\mu^{-2}}{N^2} \ln N + \frac{\sigma_\mu^{-2} \theta_0^{\mu 2}}{N}}
\]
where $A_N$ is the set defined in Lemma \ref{lemma:contraction_rates_white_noise} and satisfies $\Pi_1 \c{A_N^c} = o_p(1)$ (actually it is $O_p\p{N^{-K}}$ for any $K$ up to the choice of constant $M>0$ in the definition on $A_N$). Moreover the $\ln N$ term could be avoided by replacing the rectangular contraction set with an ellipsoidal one and using more specific Gaussian/$\chi^2$ concentration inequalities. As a result, using the conclusions of Corollary \ref{corollary:squared_norm_functional_white_noise} under Conditions (\ref{equation:correct_condition_CR15}) the BvM phenomenon holds for the functional $\norm{\theta}_2^2 - 2p/N$ with centring $\hat \Psi_0$ (since in that case the non-constant part of the correction becomes negligible at the $N^{-1/2}$ scale). But importantly considering the full ``divergence corrected'' functional
\[
\theta \mapsto \sum_{\mu = 1}^p \b{\theta^{\mu 2} + N^{-1} \sigma_\mu^{-2} \theta^\mu\p{Y^\mu + \theta^\mu}} - \frac{2p}{N}
\]
does not require any assumption other than $\sigma_\mu \geq c\p{N^{-1/2} \vee \abs{\theta_0^\mu}}$, as a result it allows us to perform Bayesian uncertainty quantification on the value of $\Psi\p{\theta_0} = \norm{\theta_0}_2^2$ for a wider range of priors. We confirm empirically these theoretical findings in Section \ref{section:simulation_study}.

Furthermore, in the case $s>1/2$ one could apply Corollary \ref{corollary:bvm_without_bias} instead of the second item of Theorem \ref{theorem:generic_bvm}, which would allow us to prove a BvM phenomenon for the uncorrected functional $\theta \mapsto \norm{\theta}_2^2$ with centring $\hat \Psi_0$ for instance under the conditions $p = \floor{\sqrt{N}/\ln N}, \quad \sigma_\mu \geq c\p{N^{-1/2} \vee \abs{\theta_0^\mu}}$, as already mentioned in \cite{castilloBernsteinVonMisesTheorem2015}. Similarly for priors of type \ref{prior:independent_series_white_noise} one could prove using Corollary \ref{corollary:bvm_without_bias} that a BvM phenomenon for the uncorrected functional $\theta \mapsto \norm{\theta}_2^2$ with centring $\hat \Psi_0$ holds under the assumptions $s>1/2, \quad p = \floor{\sqrt{N}/\ln N}, \quad \sigma_- \gg \p{N \ln N}^{-1/2}$ and $\sum_{\mu = 1}^p \sigma_\mu^{-1} \abs{\theta_0^\mu} \ll N^{1/2}$ which are stronger than the ones of Corollary \ref{corollary:squared_norm_functional_white_noise}.

\begin{proof}{Corollary \ref{corollary:squared_norm_functional_white_noise}}\\
	We apply the second item of Theorem \ref{theorem:generic_bvm} to the functional $\tilde \Phi = \sqrt{N} \Phi - \frac{p}{\sqrt{N}}, \quad \Phi = \norm{\theta}_2^2$, the estimator $\hat{\tilde \Psi}_0 = \sqrt{N} \hat \Psi_0$ and the vector field
	\[
	V = \nabla \tilde \Phi + \frac{1}{2} \nabla^2 \tilde \Phi.\nabla \ell_1 = \sqrt{N}\p{\nabla \Phi + \frac{1}{2} \nabla^2 \Phi.\nabla \ell_1}
	\]
	which using Proposition \ref{proposition:quadratic_functional_white_noise} can be expressed as $V_\mu = \sqrt{N}\p{Y^\mu + \theta^\mu}$. With this choice of $V$ the random functional $\hat {\tilde \Phi}$ then satisfies
	\begin{align*}
		\hat{\tilde \Phi} & = \tilde \Phi + \frac{1}{1} \inner{V | \nabla \ell_1} = \tilde \Phi + \p{V_\mu}\p{\nabla^\mu \ell_1} = \sqrt{N} \Phi - \frac{p}{\sqrt{N}} + \sum_{\mu = 1}^p \sqrt{N} \p{Y^\mu + \theta^\mu} \p{Y^\mu - \theta^\mu} \\
		& = - \frac{p}{\sqrt{N}} + \sqrt{N}\sum_{\mu = 1}^p \b{\theta^{\mu 2} + Y^{\mu 2} - \theta^{\mu 2}} = \hat{\tilde \Psi}_0 .
	\end{align*}
	To define the cutoff function $\chi_1$ we consider
	\[
	\chi_1\p{\theta} = \begin{cases}
		\prod_{\mu = 1}^p \chi\p{\frac{\theta^\mu - Y^\mu}{\delta_N}} \text{ for a prior of type \ref{prior:independent_series_white_noise}}, \\
		\prod_{\mu = 1}^p \chi\p{\frac{\theta^\mu - \rho_\mu^2 Y^\mu}{\rho_\mu \delta_N}} \text{ for a prior of type \ref{prior:gaussian}},
	\end{cases}
	\]
	where $\delta_N = M \sqrt{\frac{\ln N}{N}}, \quad M>0, \quad \rho_\mu^2 = \frac{N\sigma_\mu^2}{N \sigma_\mu^2 + 1}$ and $\chi : \RR \to \c{0,1}$ is any $\CC^1$ function such that $\Supp\p{\chi} \subset \p{-2,2}$ and $\chi \equiv 1$ on $\c{-1,1}$. Lemma \ref{lemma:contraction_rates_white_noise} provides the existence of $M>0$ such that $\Pi_1 \c{\chi_1} = 1 + o_p(1)$. In particular with $R = \Diag\p{\rho}$ and
	\[
	A_N = \begin{cases}
		\b{\theta : \norm{\theta - \pi_p\c{Y}}_\infty \leq M \sqrt{\frac{\ln N}{N}}} \text{ for a prior of type \ref{prior:independent_series_white_noise}} \\
		\b{\theta : \norm{R^{-1}\c{\theta - R^2 \pi_p\c{Y}}}_\infty \leq M \sqrt{\frac{\ln N}{N}}} \text{ for a Gaussian prior \ref{prior:gaussian}},
	\end{cases}
	\]
	\[
	A_N' = \begin{cases}
		\b{\theta : \norm{\theta - \pi_p\c{Y}}_\infty \leq 2M \sqrt{\frac{\ln N}{N}}} \text{ for a prior of type \ref{prior:independent_series_white_noise}} \\
		\b{\theta : \norm{R^{-1} \c{\theta - R^2 \pi_p\c{Y}}}_\infty \leq 2M \sqrt{\frac{\ln N}{N}}} \text{ for a Gaussian prior \ref{prior:gaussian}},
	\end{cases}
	\]
	we have $A_N \subset \Supp\p{\chi_1} \subset A_N', \quad \chi_1 = 1$ on $A_N$, $\Supp\p{\chi_1} \subset A_N'$ in each case. Moreover for a prior of type \ref{prior:independent_series_white_noise} we have
	\begin{align*}
		\abs{\nabla \chi_1}^2 & = \abs{g^{\mu \nu} \p{\partial_\mu \chi_1} \p{\partial_\nu \chi_1}} = N^{-1} \sum_{\mu = 1}^p \b{\p{\prod_{\nu \neq \mu} \chi\p{\frac{\theta^\nu - Y^\nu}{\delta_N}}} \delta_N^{-1} \chi'\p{\frac{\theta^\mu - Y^\mu}{\delta_N}}}^2 \\
		& \leq N^{-1} \norm{\chi'}_\infty^2 \delta_N^{-2} p \lesssim p/\ln N,
	\end{align*}
	and a similar argument for a prior of type \ref{prior:gaussian} yields
	\[
	\abs{\nabla \chi_1}^2 \leq N^{-1} \norm{\chi'}_\infty^2 \delta_N^{-2} \sum_{\mu = 1}^p 1 + \frac{1}{N\sigma_\mu^2} \leq \p{1+c^{-2}}pN^{-1} \norm{\chi'}_\infty^2 \delta_N^{-2}
	\]
	(because $\sigma_\mu \geq cN^{-1/2}$) so that in any case (and because $p \leq N$) we can find $a > 0$ such that $\sup_\Theta \abs{\nabla \chi_1} \leq N^a$. Because we have shown above that $\hat{\tilde \Phi} = \hat{\tilde \Psi}_0$ identically over $\Theta$, to conclude by applying Theorem \ref{theorem:generic_bvm} we need to appropriately bound
	\[
	\Pi_1 \c{\abs{\inner{\nabla \chi_1 | V}}} \text{ and } \sup_{\Supp\p{\chi_1}} \abs{\sigma^2 - \inner{V|\nabla \tilde \Phi} + \inner{\nabla \Div_\pi V | V}}
	\]
	for both priors. The norm of $V$ is given by $\abs{V}^2 = g^{\mu \nu} V_\mu V_\nu = N^{-1} \sum_{\mu = 1}^p V_\mu^2 = \norm{\pi_p\c{Y} + \theta}_2^2$, so that
	\[
	\ind{E_N} \sup_{\theta \in \Supp\p{\chi_1}} \abs{V} \leq \ind{E_N} \sup_{\theta \in A_N'} \norm{\pi_p\c{Y} + \theta}_2 \leq \ind{E_N} \sup_{\theta \in A_N'} \norm{\pi_p\c{Y} + \theta}_2 \leq b
	\]
	for some fixed $b>0$ in each case, using the conclusion $A_N' \subset \b{\theta : \norm{\pi_p^*\c{\theta} - \theta_0}_2 \leq c\sqrt{\frac{p \ln N}{N}}}$ of Lemma \ref{lemma:contraction_rates_white_noise} in both cases and the definition of the event $E_N$, for some fixed $c>0$. Moreover because $\chi_1 \equiv 1$ on $A_N$ we also have $\nabla \chi_1 = 0$ on $A_N$, which using Cauchy-Schwarz inequality implies that
	\[
	\ind{E_N} \Pi_1 \c{\abs{\inner{\nabla \chi_1 | V}}} \leq \ind{E_N} \Pi_1 \c{\ind{A_N^c} \sup_{\Supp\p{\chi_1}} \abs{\nabla \chi_1} \abs{V}} \leq \ind{E_N} \Pi_1 \c{\ind{A_N^c} N^a b} \leq bN^{a-K}.
	\]
	Thus if $K$ is chosen strictly larger than say $2a$ (which is always possible provided that $M$ is initially chosen large enough itself), we obtain $\ind{E_N} \Pi_1 \c{\abs{\inner{\nabla \chi_1 | V}}} \leq bN^{-K/2} \to 0$, which implies that $\Pi_1 \c{\abs{\inner{\nabla \chi_1 | V}}} = O_p\p{N^{-K/2}}$ since $\Pro_0 \c{E_N} \to 1$. Furthermore because
	\[
	\inner{\nabla \tilde \Phi | V} = g^{\mu \nu} \p{\partial \tilde \Phi} V_\nu = N^{-1} \sum_{\mu = 1}^p 2N^{1/2} \theta_\mu \times N^{1/2}\p{Y^\mu + \theta^\mu} = 2\inner{\theta | \pi_p\c{Y} + \theta}_2
	\]
	and $\sigma^2 = 4\norm{\theta_0}_2^2$, using $\Supp\p{\chi_1} \subset A_N' \subset \b{\theta : \norm{\theta - \pi_p\c{\theta_0}}_2 \leq C \varepsilon_N}$ as shown in Lemma \ref{lemma:contraction_rates_white_noise} we obtain
	\[
	\ind{E_N} \sup_{\Supp\p{\chi_1}} \abs{\sigma^2 - \inner{\nabla \tilde \Phi | V}} \leq c\varepsilon_N,
	\]
	for some fixed $c>0$, proving that $\sup_{\Supp\p{\chi_1}} \abs{\sigma^2 - \inner{\nabla \tilde \Phi | V}} = O_p\p{\varepsilon_N}$ since $\Pro_0 \c{E_N} \to 1$. To conclude we thus only need to show that $\sup_{\Supp\p{\chi_1}} \abs{\inner{\nabla \Div_\pi V | V}} = o_p(1)$ by distinguishing between the two priors.
	\begin{enumerate}
		\item For a prior of type \ref{prior:independent_series_white_noise} using Proposition \ref{proposition:quadratic_functional_white_noise} we have
		\begin{align*}
			\abs{\inner{\nabla \Div_\pi V | V}} & = \abs{g^{\mu \nu} \p{\partial_\mu \Div_\pi V} V_\nu} \\
			& = N^{-1} \abs{\sum_{\mu = 1}^p N^{-1/2} \sigma_\mu^{-1} \b{\p{Y^\mu + \theta^\mu} \sigma_\mu^{-1} H''\p{\theta^\mu/\sigma_\mu} + H'\p{\theta^\mu/\sigma_\mu}} \times \sqrt{N}\p{Y^\mu + \theta^\mu}} \\
			& \leq N^{-1} \sum_{\mu = 1}^p \b{ \sigma_\mu^{-2} \p{Y^\mu + \theta^\mu}^2 \norm{H''}_\infty + \sigma_\mu^{-1} \norm{H'}_\infty \abs{Y^\mu + \theta^\mu}} \\
			& \leq N^{-1} \b{\norm{H''}_\infty \p{\sum_{\mu = 1}^p \sigma_\mu^{-2} \p{Y^\mu + \theta^\mu}^2} + \norm{H'}_\infty \sqrt{p} \p{\sum_{\mu = 1}^p \sigma_\mu^{-2} \p{Y^\mu + \theta^\mu}^2}^{1/2}} \\
			& \leq N^{-1} \b{\norm{H''}_\infty \sigma_-^{-2} \norm{\pi_p\c{Y}+\theta}_2^2 + \norm{H'}_\infty \sqrt{p} \sigma_-^{-1} \norm{\pi_p\c{Y}+\theta}_2}.
		\end{align*}
		where $\sigma_- = \min_{1 \leq \mu \leq p} \sigma_\mu$. Because $p \ll N/\ln N, \quad \Supp\p{\chi_1} \subset A_N' \subset \b{\theta : \norm{\theta}_2 \leq \norm{\theta_0}_2 + \varepsilon_N} \subset \b{\theta : \norm{\theta}_2 \leq 2\norm{\theta_0}_2}$ and as shown above $\ind{E_N} \sup_{\theta \in A_N'} \norm{\pi_p\c{Y} + \theta}_2$ is bounded by a fixed constant, we obtain that
		\[
		\sup_{\Supp\p{\chi_1}} \abs{\inner{\nabla \Div_\pi V | V}} = O_p\p{\frac{1}{N\sigma_-^2} \vee \frac{\sqrt{p}}{N\sigma_-}},
		\]
		which is also $o_p(1)$ since $\sigma_- \gg N^{-1/2} \gg \frac{\sqrt{p}}{N}$. 
		\item For a prior of type \ref{prior:gaussian} we have using Proposition \ref{proposition:quadratic_functional_white_noise}
		\begin{align*}
			& \abs{\inner{\nabla \Div_\pi V | V}} \\
			& = \abs{g^{\mu \nu} \p{\partial_\mu \Div_\pi V} V_\nu} = N^{-1} \abs{\sum_{\mu = 1}^p N^{-1/2} \sigma_\mu^{-2} \p{Y^\mu + 2\theta^\mu} \times \sqrt{N}\p{Y^\mu + \theta^\mu}} \\
			& \leq N^{-1} \sum_{\mu = 1}^p \sigma_\mu^{-2} \abs{\p{Y^\mu + 2 \theta^\mu}\p{Y^\mu + \theta^\mu}} \leq 2N^{-1} \sum_{\mu = 1}^p \sigma_\mu^{-2} \p{\abs{Y^\mu} + \abs{\theta^\mu}}^2 \\
			& \leq 4N^{-1} \sum_{\mu = 1}^p \sigma_\mu^{-2} \p{Y^{\mu 2} + \theta^{\mu 2}},
		\end{align*}
		where we have used the inequality $\p{a+b}^2 \leq 2\p{a^2 + b^2}$, which we continue to use repeatedly below. First because $Y^\mu = \theta_0^\mu + N^{-1/2} \varepsilon^\mu$ and because $\sigma_\mu \geq cN^{-1/2}$ we have
		\[
		\sum_{\mu = 1}^p \sigma_\mu^{-2} Y^{\mu 2} \leq 2\sum_{\mu = 1}^p \sigma_\mu^{-2} \p{\theta_0^{\mu 2} + N^{-1} \varepsilon^{\mu 2}} \leq 2c^{-2}p + 2 \norm{\pi_p\c{\varepsilon}}_2^2,
		\]
		which (since $\pi_p\c{\varepsilon}_2^2 \sim \chi_p^2$ under $\Pro_0$) is $O_p\p{p}$. Moreover decomposing $\theta^\mu = \theta^\mu - \rho_\mu^2 Y^\mu + \rho_\mu^2 Y^\mu$ and $Y^\mu = \theta_0^\mu + N^{-1/2} \varepsilon^\mu$ we obtain
		\begin{align*}
		\sup_{\theta \in A_N'} \sum_{\mu = 1}^p \sigma_\mu^{-2} \theta^{\mu 2} & \leq \frac{8M^2\ln N}{N} \sum_{\mu = 1}^p \sigma_\mu^{-2} \rho_\mu^2 + 4\sum_{\mu = 1}^p \sigma_\mu^{-2} \rho_\mu^4 \theta_0^{\mu 2} + 4N^{-1}\sum_{\mu = 1}^p \sigma_\mu^{-2} \rho_\mu^4 \varepsilon^{\mu 2} \\
		& \leq 8M^2 p \ln N + 4c^{-2}p + 4\norm{\pi_p\c{\varepsilon}}_2^2,
		\end{align*}
		(where we have also used $\rho_\mu^2 \leq 1$), which is $O_p\p{p \ln N}$ for the same reason as above. Thus
		\[
		\sup_{\Supp\p{\chi_1}} \abs{\inner{\nabla \Div_\pi V | V}} \leq \sup_{A_N'} \abs{\inner{\nabla \Div_\pi V | V}} = O_p\p{\frac{p \ln N}{N}} = o_p\p{1}.
		\]
	\end{enumerate}
	We deduce the desired result by applying Theorem \ref{theorem:generic_bvm}, using the expressions
	\[
	\Div_\pi V = \sqrt{N}\p{\frac{p}{N} -N^{-1} \sum_{\mu = 1}^p \sigma_\mu^{-1} \p{Y^\mu + \theta^\mu} H'\p{\theta^\mu/\sigma_\mu}}
	\]
	(in the prior \ref{prior:independent_series_white_noise} case) and
	\[
	\Div_\pi V = \sqrt{N}\p{\frac{p}{N} - N^{-1} \sum_{\mu = 1}^p \sigma_\mu^{-2} \p{Y^\mu + \theta^\mu} \theta^\mu}
	\]
	(in the prior \ref{prior:gaussian} case) proved in Proposition \ref{proposition:quadratic_functional_white_noise}. The quantitative estimates on the bounded Lipschitz distances follow by examining the respective orders of magnitude of the terms $A_0,A_1,A_2$ and $A_3$ when applying Theorem \ref{theorem:generic_bvm}.
\end{proof}

\subsection{Density model with priors based on histograms}\label{section:histograms}

We now consider the case of density models with histograms, where the information geometry of the problem is non-trivial. Let $p = 2^{J+1}, J \geq 0$ and consider the exponential family $\MM = \b{P_\theta : \theta \in \RR^{p-1}}$ on $\XX = \c{0,1}$ defined via
\[
\frac{dP_\theta}{dx} = f_\theta, \quad f_\theta = \exp \p{\theta^\mu u_\mu - \psi\p{\theta}}, \quad \psi\p{\theta} = \ln \int \exp \p{\theta^\mu u_\mu},
\]
where $\p{u_\mu}_{\mu = 1}^{p-1}$ are the Haar wavelets $\p{\psi_{jk}}_{0 \leq j \leq J, 1 \leq k \leq 2^j}$ (see e.g. \cite{gineMathematicalFoundationsInfiniteDimensional2015}) that we have relabeled for convenience. Notice that according to our parametrisation we have $\theta^\mu = \inner{u_\mu | \ln f_\theta}_2$. We define $\FF_p$ and $\HH_p$ as
\[
\FF_p = \Span \b{\ind{I_\mu} : 1 \leq \mu \leq p}, \quad \HH_p = \b{f \in \FF_p : f : \c{0,1} \to \RR_+, \quad \int f = 1},
\]
where
\[
I_\mu = \left[\frac{\mu-1}{p},\frac{\mu}{p}\right), \quad 1 \leq \mu \leq p-1, \quad I_p = \c{\frac{p-1}{p},1}.
\]
When the parametrisation is not relevant for the statements we will often consider a generic element $f \in \HH_p$ without explicitly indicating the parameter $\theta$ in the notation. $\FF_p$ admits two natural orthonormal bases: $\p{\sqrt{p}\ind{I_1},\ldots,\sqrt{p}\ind{I_p}}$ and $\p{\ind{},\p{\psi_{jk}}_{0 \leq j \leq J, 1 \leq k \leq 2^j}}$. We define $\pi_p : L^2 \to \FF_p$ as the orthogonal projection
\[
\pi_p\c{f} = p\sum_{\mu = 1}^p \inner{\ind{I_\mu}|f}_2 \ind{I_\mu} = \int f + \sum_{j = 0}^J \sum_{k = 1}^{2^j} \inner{\psi_{jk} | f}_2 \psi_{jk}, \quad f\in L^2.
\]
We recall that (see e.g. \cite{gineMathematicalFoundationsInfiniteDimensional2015}) any $f \in \CC^s, 0 < s \leq 1$ (classical H\"older space) satisfies
\[
\norm{f - \pi_p\c{f}}_\infty \lesssim \norm{f}_{\CC^s} p^{-s}, 
\quad \sup_{j \geq 0} 2^{j\p{s + 1/2}}\max_{1 \leq k \leq 2^j} \abs{\inner{\psi_{jk}|f}_2} \lesssim \norm{f}_{\CC^s}.
\]

We consider the distribution $P_\theta^n$  of an iid sample $X^n = \p{X_i}_{i=1}^n \sim P_\theta^n$, with associated log-likelihood
\[
\ell_n(x^n,\theta) = \sum_{i = 1}^n \ell\p{x_i,\theta}, \quad \ell\p{\cdot,\theta} = \ln f_\theta = \theta^\mu u_\mu - \psi(\theta),
\]
and score
\[
\partial_\mu \ell = u_\mu - \frac{\partial \psi}{\partial \theta^\mu} = u_\mu - P_\theta \c{u_\mu},
\]
which yields the Fisher information metric
\[
g_{\mu \nu}\p{\theta} = P_\theta \c{\p{\partial_\mu \ell} \p{\partial_\nu \ell}} = \Cov_\theta \p{u_\mu,u_\nu} = \frac{\partial^2 \psi}{\partial \theta^\mu \partial \theta^\nu}\p{\theta},
\]
where $\Cov_\theta$ is the covariance operator under $P_\theta$. Since the functions $\p{\ind{},u_1,\ldots,u_{p-1}}$ are linearly independent it is easy to check that the Fisher information matrix is invertible at every $\theta$. Following \cite{amariDifferentialGeometryCurved1982} the model $\MM$ can be equivalently described using the dual coordinates
\[
\eta_\mu := P_\theta \c{u_\mu} = \frac{\partial \psi}{\partial \theta^\mu}, \quad 1 \leq \mu \leq p - 1,
\]
for which we have
\[
g_{\mu \nu}\p{\theta} = \frac{\partial \eta_\mu}{\partial \theta_\nu}, \quad g^{\mu \nu}\p{\theta} = \frac{\partial \theta^\mu}{\partial \eta_\nu},
\]
where $g^{\mu \nu}\p{\theta}$ is the $\p{\mu,\nu}-$entry of the inverse of the matrix $\c{g_{\mu \nu}\p{\theta}}_{\mu \nu = 1}^{p-1}$. We define the dual differentiation operators $\partial_\mu$ and $\partial^\mu$ as
\[
\partial_\mu := \frac{\partial}{\partial \theta^\mu}, \quad \partial^\mu := \frac{\partial}{\partial \eta_\mu}, \quad 1 \leq \mu \leq p - 1.
\]
With these definitions it can be checked that
\[
g_{\mu \nu}\p{\eta} = g^{\mu \nu}\p{\theta} = \partial^{\mu \nu} \chi, \quad \chi = \int f \ln f.
\]
Moreover following \cite{amariDifferentialGeometryCurved1982} the Christoffel symbols $\Gamma_{~\mu \nu}^\rho$ of the Levi-Civita connection $\nabla$ induced by the Fisher information metric are given in natural coordinates (i.e. $\theta$ coordinates) in terms of the Amari-\v{C}encov cubic tensor $T$:
\[
\Gamma_{~\mu \nu}^\rho\p{\theta} = \frac{1}{2} T_{\mu \nu}^{~~\rho}\p{\theta}, \quad T_{\mu \nu \rho} := P_\theta \c{\p{\partial_\mu \ell} \p{\partial_\nu \ell} \p{\partial_\rho \ell}}.
\]
Notice that in natural coordinates the Amari-\v{C}encov tensor can be expressed as $T_{\mu \nu \rho}\p{\theta} = \partial_{\mu \nu \rho} \psi\p{\theta}$. A histogram $f_\theta \in \HH_p$ can equivalently be represented as
\[
f_\theta = h_\omega := p \sum_{\mu = 1}^p \omega_\mu \ind{I_\mu}, \quad \omega_\mu > 0, \quad \sum_{\mu = 1}^p \omega_\mu = 1.
\]
These can be expressed as an exponential family with sufficient statistics $\p{\ind{I_\mu}}_{\mu = 1}^{p-1}$ instead of $\p{u_\mu}_{\mu = 1}^{p-1}$ :
\begin{align*}
	h_\omega & = \exp\p{\sum_{\mu = 1}^p \ln \p{p \omega_\mu} \ind{I_\mu}} = \exp\p{\sum_{\mu = 1}^{p-1} \ln \p{p \omega_\mu} \ind{I_\mu} + \ln\p{p \omega_p} \p{1 - \sum_{\mu = 1}^{p-1} \ind{I_\mu}}} \\
	& = \exp\p{\sum_{\mu = 1}^{p-1} \ln \p{\omega_\mu/\omega_p} \ind{I_\mu} + \ln \p{p\omega_p}},
\end{align*}
so that with the $\alpha$ coordinates defined as
\[
\alpha^\mu = \ln \frac{\omega_\mu}{1 - \sum_{\nu = 1}^{p-1} \omega_\nu} \iff \omega_\mu = \frac{e^{\alpha^\mu}}{1 + \sum_{\nu = 1}^{p-1} e^{\alpha^\nu}}, \quad 1 \leq \mu \leq p-1,
\]
we obtain
\[
f_\theta = h_\omega := \exp\p{\alpha^\mu \ind{I_\mu} - S\p{\alpha}}, \quad S\p{\alpha} = \ln \p{\frac{1}{p} \p{1 + \sum_{\mu = 1}^{p-1} e^{\alpha^\mu}}},
\]
which parametrises $\HH_p$ as another exponential family with natural parameter $\p{\alpha^\mu}_{\mu = 1}^{p-1}$ and dual parameter $\p{\omega_\mu}_{\mu = 1}^{p-1}$. As for the $\theta$ coordinate system we will often consider a generic $h$ or $f \in \HH_p$ when the parametrisation $\omega$ is not relevant for the statement. By convention we define $\alpha^p = 0$ so that
\[
\alpha^\mu = \ln \frac{\omega_\mu}{1 - \sum_{\nu = 1}^{p-1} \omega_\nu} \iff \omega_\mu = \frac{e^{\alpha^\mu}}{1 + \sum_{\nu = 1}^{p-1} e^{\alpha^\nu}}, \quad 1 \leq \mu \leq p, \quad S\p{\alpha} = \ln \p{\frac{1}{p} \sum_{\mu = 1}^p e^{\alpha^\mu}}.
\]

We consider two different classes of priors over $\HH_p$ :

\begin{prior}\label{prior:independent_haar_series_histograms}
$\Pi \c{d\theta} \propto \prod_{\mu = 1}^{p-1} \exp\p{-H\p{\theta^\mu/\sigma_\mu}}d\theta^\mu$ for some differentiable function $H : \RR \to \RR$ with $\norm{H'}_\infty < +\infty$ and scaling parameters $\sigma_\mu > 0, \quad 1 \leq \mu \leq p-1$. For simplicity we assume that $\sigma_\mu = \sigma_{\p{j,k}} = \sigma_j$ only depends on $j$.
\end{prior}

\begin{prior}\label{prior:dirichlet}
$\Pi$ is a Dirichlet prior with parameter $m \in \p{0,\infty}^p$, i.e. in variables $\p{\omega_\mu}_{\mu = 1}^{p-1}$ the density $\Pi\p{\omega}$ of $\Pi$ with respect to the Lebesgue measure over the admissibility set
\[
\b{\p{\omega_\mu}_{\mu = 1}^{p-1} \in \p{0,\infty}^{p-1} : \sum_{\mu = 1}^{p-1} \omega_\mu < 1}
\]
is given by (see e.g. \cite{ghosalFundamentalsNonparametricBayesian2017})
\[
\Pi\p{\omega_1,\ldots,\omega_{p-1}} = \frac{\prod_{\mu = 1}^p \omega_\mu^{m_\mu - 1}}{B(m)}, \quad \omega_p := 1- \sum_{\mu = 1}^{p-1} \omega_\mu.
\]
\end{prior}

In the remainder of this section we consider iid observations $X^n = \p{X_i}_{i = 1}^n \iid f_0(x)dx$ for some $f_0 \in \CC^s, 0 < s \leq 1$ that satisfies $c_- \leq f_0 \leq c_+$ for some $c_-,c_+ > 0$ and we define
\[
\hat f := 1 + \sum_{j = 0}^{J} \sum_{k = 1}^{2^j} \hat f_{jk} \psi_{jk}, \quad \hat f_{jk} := \Pro_n \c{\psi_{jk}}.
\]
and recall that, since $f_0$ integrates to $1$,
\[
\pi_p\c{f_0} = 1 + \sum_{j = 0}^J \sum_{k = 1}^{2^j} f_{0jk} \psi_{jk}, \quad f_{0jk} := \inner{\psi_{jk} | f_0}_2.
\]

\paragraph{Linear functionals.} We first investigate the consequences of Theorem \ref{theorem:generic_bvm} and Corollary \ref{corollary:bvm_without_bias} for inference on the value of a functional of the form $\inner{a | f_0}_2 = \int a f_0, \quad a \in L^\infty$. In that particular case the simple estimator $\Pro_n \c{a} = \frac{1}{n} \sum_{i = 1}^n a\p{X_i}$ satisfies $\sqrt{n}\p{\Pro_n - \Pro_0} \c{a} \xrightarrow[n \to \infty]{(d)} \NN\p{0,\sigma^2}$ where $\sigma^2 = \int \p{a - \int a f_0}^2 f_0$, which is optimal in the semiparametric sense, see e.g. \cite{vaartAsymptoticStatistics1998}. Moreover because by orthogonality
\[
\Pro_n \c{\pi_p\c{a} - a} = \p{\Pro_n - \Pro_0}\c{\pi_p\c{a} - a} + \Pro_0 \c{\pi_p\c{a} - a} = \p{\Pro_n - \Pro_0}\c{\pi_p\c{a} - a} + \int \p{f_0 - \pi_p\c{f_0}} \p{\pi_p\c{a} - a},
\]
we obtain
\begin{align*}
\Pro_n \c{\pi_p\c{a} - a} & = O_p\p{n^{-1/2} \Pro_0 \c{\p{\pi_p\c{a} - a - \Pro_0 \c{\pi_p\c{a} - a}}^2}^{1/2} + \norm{\pi_p\c{f_0} - f_0}_\infty \norm{\pi_p\c{a} - a}_\infty} \\
& = O_p\p{n^{-1/2} \norm{a - \pi_p\c{a}}_\infty + \norm{\pi_p\c{f_0} - f_0}_\infty \norm{\pi_p\c{a} - a}_\infty} \\
& = O_p\p{n^{-1/2} \norm{a}_{\CC^\gamma} p^{-\gamma} + \norm{a}_{\CC^\gamma} \norm{f_0}_{\CC^s} p^{-(\gamma + s)}},
\end{align*}
which is $o_p\p{n^{-1/2}}$ as soon as $f_0 \in \CC^s,a \in \CC^\gamma$ for some $s,\gamma > 0$ and $p \gg n^{\frac{1}{2(\gamma + s)}}$. As a result the estimator $\hat \Psi_0 := \Pro_n \c{\pi_p\c{a}} = \Pro_n \c{a} + o_p\p{n^{-1/2}}$ also satisfies
\[
\sqrt{n}\p{\hat \Psi_0 - \inner{a|f_0}_2} \xrightarrow[n \to \infty]{(d)} \NN\p{0,\sigma^2}
\]
whenever $p \gg n^{\frac{1}{2(\gamma + s)}}$. We present below an application of (the second item of) Theorem \ref{theorem:generic_bvm} to this setting and then briefly discuss what Corollary \ref{corollary:bvm_without_bias} (BvM theorem without functional correction) would imply instead.

\begin{corollary}\label{corollary:linear_functionals_histograms}
If $X^n = \p{X_i}_{i = 1}^n \iid f_0$ for some probability $f_0 \in \CC^s, 0 < s \leq 1$ and $p \gg n^{\frac{1}{2\p{\gamma + s}}}$ then:
\begin{enumerate}
	\item if $\Pi$ is a prior of type \ref{prior:independent_haar_series_histograms} with $p \ll n/\ln n$, $H$ twice differentiable and $\norm{H''}_\infty < + \infty$,
	\[
	\sum_{j = 0}^J \sigma_j^{-2} \sum_{k = 1}^{2^j} \inner{\psi_{jk} | a}_2^2 \ll n \text{ and } \sum_{j = 0}^J \sigma_j^{-1} 2^{j/2} \leq c\sqrt{np\ln n}
	\]
	for a fixed constant $c>0$, we have
	\[
	d_{BL}\p{\Pi_n \circ Z_n^{-1},\NN\p{0,\sigma^2}} = O_p\p{p^{-\p{\gamma \wedge s}} + \sqrt{\frac{p \ln n}{n}} + \frac{1}{n} \sum_{j = 0}^J \sigma_j^{-2} \sum_{k = 1}^{2^j} \inner{\psi_{jk} | a}_2^2 } = o_p(1)
	\]
	\[
	\text{where }Z_n := \sqrt{n} \p{\int a f + \frac{1}{n} \sum_{j = 0}^J \sigma_j^{-1} \sum_{k = 1}^{2^j} H'\p{\sigma_j^{-1} \inner{\psi_{jk} | \ln f}_2} \inner{\psi_{jk} | a}_2 - \hat \Psi_0},
	\]
	and in particular the BvM phenomenon with centring $\hat \Psi_0$ holds for the functional
	\[
	f \mapsto \int a f + \frac{1}{n} \sum_{j = 0}^J \sigma_j^{-1} \sum_{k = 1}^{2^j} H'\p{\sigma_j^{-1} \inner{\psi_{jk} | \ln f}_2} \inner{\psi_{jk} | a}_2
	\]
	whenever $\sum_{j = 0}^J \sigma_j^{-2} \sum_{k = 1}^{2^j} \inner{\psi_{jk} | a}_2^2 \ll n, \quad \gamma + s > 1/2$ and $n^{\frac{1}{2(\gamma + s)}} \ll p \ll n/\ln n$.
	\item If $\Pi$ is a prior of type \ref{prior:dirichlet} with $p \ll n/\ln n$ and $\norm{m}_\infty = O(1)$ then
	\[
	d_{BL}\p{\Pi_n \circ Z_n^{-1},\NN\p{0,\sigma^2}} = O_p\p{p^{-\p{\gamma \wedge s}} + \sqrt{\frac{p \ln n}{n}}} = o_p(1)
	\]
	\[
	\text{where }Z_n := \sqrt{n} \p{\int a f - \frac{\norm{m}_1}{n} \inner{a | h_{\bar m} - f}_2 - \hat \Psi_0}, \quad \bar m := m/\norm{m}_1,
	\]
	and in particular the BvM phenomenon with centring $\hat \Psi_0$ holds for the functional
	\[
	f \mapsto \int a f - \frac{\norm{m}_1}{n} \inner{a | h_{\bar m} - f}_2
	\]
	whenever $\norm{m}_\infty = O(1), \quad \gamma + s > 1/2$ and $n^{\frac{1}{2(\gamma + s)}} \ll p \ll n/\ln n$.
\end{enumerate}
\end{corollary}

We now comment on these results before presenting a proof. First the conditions for priors of type \ref{prior:independent_haar_series_histograms} are met for instance when $s + \gamma > 1/2, \quad p \asymp n/\ln^2 n$ and $\sigma_j \geq c \abs{\inner{\psi_{jk} | a}_2} \vee c \p{n\ln n}^{-1/2}$ for some fixed $c>0$ (e.g. $\sigma_j = 2^{-j/2}$ since $a \in \CC^\gamma, \gamma \geq 0$). For Dirichlet priors \ref{prior:dirichlet} they are met for instance if $s+\gamma > 1/2, \quad p \asymp n/\ln^2 n$ and $\norm{m}_\infty = O(1)$. These conditions are weaker than the ones considered in \cite{castilloBernsteinVonMisesTheorem2015} for the regime $s + \gamma > 1/2$ as it is assumed there that $\norm{m}_1 \ll n^{1/2}$. In contrast, applying Corollary \ref{corollary:bvm_without_bias} with priors of type \ref{prior:independent_haar_series_histograms} would allow us to prove a BvM with centring $\hat \Psi_0$ for the original uncorrected functional $\Psi : f \mapsto \int a f$ under the more restrictive conditions $s + \gamma > 1$ with the choices $p \asymp \p{n/\ln^2 n}^{1/2}$ and $\sigma_j \geq c \abs{\inner{\psi_{jk} | a}_2} \vee c \p{n\ln n}^{-1/2}$ (e.g. $\sigma_j = 2^{-j/2}$). For Dirichlet priors \ref{prior:dirichlet} this would be possible under the conditions $s + \gamma > 1, \quad p \asymp \p{n/\ln^2 n}^{1/2}$ and $\norm{m}_\infty = O(1), \quad \norm{m}_1 \ll n^{1/2}$, which allows us to recover the conclusions of \cite{castilloBernsteinVonMisesTheorem2015}. As a consequence and as in Section \ref{section:white_noise} considering the divergence corrected functional in the present setting allows us to extend the range of validity of previously known BvM theorems.

Moreover, in the case of Dirichlet priors \ref{prior:dirichlet} the divergence corrected functional has a very intuitive form: indeed it is equal to $\inner{a | h_\omega - \frac{\norm{m}_1}{n} \p{h_{\bar m} - h_\omega}}_2$, and by conjugacy the posterior distribution of $h_\omega - \frac{\norm{m}_1}{n} \p{h_{\bar m} - h_\omega}$ can be shown to have a mean given by the empirical histogram estimator $\hat f$. As a result the posterior mean of the divergence corrected functional is exactly $\Pro_n \c{\pi_p\c{a}} = \hat \Psi_0 + o_p\p{n^{-1/2}}$ (since $p \gg n^{\frac{1}{2(\gamma + s)}}$). Thus in this case and as previously discussed in Section \ref{section:general_results} in more generality, the divergence correction step allows us to remove all prior induced posterior bias at the relevant $n^{-1/2}$ scale. What is interesting and not obvious is that the same correction step $\Psi \leftrightarrow \Psi - n^{-1} \Delta_\pi \Psi$ actually produces the same effect even for nonconjugate priors of type \ref{prior:independent_haar_series_histograms}.

\begin{proof}{Corollary \ref{corollary:linear_functionals_histograms}}\\
In both cases the proof follows by an application of Theorem \ref{theorem:generic_bvm} to the functional $\Phi_a = \int a h_\omega = \int a f_\theta$, the vector field $V = \nabla \Phi_a$ and the estimator $\hat \Psi_0 = \Pro_n \c{\pi_p\c{a}}$. We first need to define appropriate cutoff functions, which we do by invoking Lemma \ref{lemma:supremum_norm_contraction_rates_histograms} : for any $K>0$ we have $\ind{E_n}\Pi_n \c{A_n^c} \leq n^{-K}$ and $\Pro_0\c{E_n^c} \leq n^{-K}$ for some $M>0$ where
\[
A_n = \b{f \in \HH_p :  \norm{f - \hat f}_\infty \leq \varepsilon_n}, \quad E_n = \b{\norm{\hat f - \pi_p\c{f_0}}_\infty \leq \varepsilon_n}, \quad \varepsilon_n := M \sqrt{\frac{p \ln n}{n}}.
\]
Then given a $\CC^1$ function $\chi : \RR \to \c{0,1}$ supported on $\p{-2,2}$ and identically equal to $1$ on $\c{-1,1}$ we set
\[
\chi_n\p{\omega} := \prod_{\mu = 1}^p \chi\p{\frac{p\p{\omega_\mu - \hat \omega_\mu}}{\varepsilon_n}}, \quad \hat \omega_\mu := \Pro_n \c{I_\mu}.
\]
Using $c_- \leq f_0 \leq c_+$ one can check that $\chi_n$ is indeed compactly supported, even supported on a set of the form $\b{\omega : a \leq h_\omega \leq b}$ for some fixed $a,b>0$, at least on the event $E_n$. Since $0 \leq \chi_n \leq 1$ and $\chi_n$ is identically equal to $1$ on $A_n$ we have in particular $1 + o_p(1) = \Pi_n \c{A_n} = \Pi_n \c{\chi_n \ind{A_n}} \leq \Pi_n \c{\chi_n} \leq 1$, proving that $\Pi_n \c{\chi_n} = 1 + o_p(1)$. Notice that because both priors are supported on $\HH_p$ by orthogonality we actually have $\Phi_a = \int \pi_p\c{a} h_\omega$ so that Proposition \ref{proposition:linear_functionals_histograms} implies $\nabla \Phi_a = \pi_p\c{a} - \int \pi_p\c{a} h_\omega$. Moreover, in the $\alpha-\omega$ coordinate system,
\begin{align*}
\nabla^\mu \chi_n & = \partial^\mu \chi_n = \frac{\partial \chi_n}{\partial \omega_\mu} \\
& = \frac{p}{\varepsilon_n} \chi'\p{\frac{p\p{\omega_\mu - \hat \omega_\mu}}{\varepsilon_n}} \times \prod_{\nu \neq \mu} \chi\p{\frac{p\p{\omega_\nu - \hat \omega_\nu}}{\varepsilon_n}} - \frac{p}{\varepsilon_n} \chi'\p{\frac{p\p{\omega_p - \hat \omega_p}}{\varepsilon_n}} \times \prod_{\nu \neq p} \chi\p{\frac{p\p{\omega_\nu - \hat \omega_\nu}}{\varepsilon_n}} \\
& =: Q_n^\mu - Q_n^p,
\end{align*}
(remembering that $\omega_p = 1 - \sum_{\mu = 1}^{p-1} \omega_\mu$), therefore
\[
\nabla \chi_n = \p{\partial^\mu \chi_n} \p{\partial_\mu \ell} = \sum_{\mu = 1}^{p-1} \p{Q_n^\mu - Q_n^p} \p{\ind{I_\mu} - \omega_\mu} = \sum_{\mu = 1}^p Q_n^\mu \p{\ind{I_\mu} - \omega_\mu},
\]
where we have used $\sum_{\mu = 1}^{p-1} \ind{I_\mu} = 1 - \ind{I_p}$. Thus $\abs{\nabla \chi_n}^2 = P_\omega \c{\p{\nabla \chi_n}^2} \leq \norm{\p{\sum_{\mu = 1}^p Q_n^\mu \ind{I_\mu}}^2}_\infty = \max_{1 \leq \mu \leq p} \abs{Q_n^\mu}^2 \leq \p{\frac{p}{\varepsilon_n}}^2 \norm{\chi'}_\infty^2$ and since we also have $\abs{\nabla \Phi_a}^2 = P_\omega \c{\p{\pi_p\c{a} - \int \pi_p\c{a} h_\omega}^2} \leq \norm{\pi_p\c{a}}_\infty^2 \leq \norm{a}_\infty^2$, using Cauchy--Schwarz inequality we obtain $\abs{\inner{\nabla \chi_n | \nabla \Phi_a}} \leq \norm{a}_\infty \norm{\chi'}_\infty \p{p/\varepsilon_n}$. But $\chi_n$ is identically equal to $1$ over $A_n$, therefore
\[
\ind{E_n}\Pi_n\c{\abs{\inner{\nabla \chi_n | V}}} \leq \norm{a}_\infty \norm{\chi'}_\infty \frac{p}{\varepsilon_n}\ind{E_n}\Pi_n \c{A_n^c} \leq \norm{a}_\infty \norm{\chi'}_\infty \frac{p}{\varepsilon_n} n^{-K},
\]
which is less than $n^{-K/2}$ for large $K>0$ since $p \ll n/\ln n$ and therefore $p/\varepsilon_n \ll n/\ln n$. Because $\Pro_0\c{E_n} \to 1$ this proves that $\Pi_n \c{\abs{\inner{\nabla \chi_n | V}}} = o_p(1) = o_p\p{\sqrt{n}}$. Moreover by Proposition \ref{proposition:linear_functionals_histograms} we have $\hat \Phi_a = \hat \Psi_0$, so that $\sup_{\Supp\p{\chi_n}} \abs{\hat \Phi_a - \hat \Psi_0} = 0 = o_p\p{n^{-1/2}}$ where $\hat \Phi_a$ is the random functional $\hat \Phi_a = \Phi_a + n^{-1} \inner{\nabla \Phi_a | \nabla \ell_n}$. To conclude it remains to show that in both cases we have
\[
\sup_{\Supp\p{\chi_n}} \abs{\sigma^2 - \inner{V | \nabla \Phi_a} + \frac{1}{n} \inner{\nabla \Div_\pi V | V}} = \sup_{\Supp\p{\chi_n}} \abs{\sigma^2 - \abs{\nabla \Phi_a}^2 + \frac{1}{n} \inner{\nabla \Delta_\pi \Phi_a | \nabla \Phi_a}} = o_p(1).
\]
For this we first observe that because $\Supp\p{\chi_n} \subset A_n' := \b{f : \norm{f - \hat f}_\infty \leq 2\varepsilon_n}$ and
\[
\abs{\nabla \Phi_a}^2 = P_\omega \c{\p{\nabla \Phi_a}^2} = \int \p{\pi_p\c{a} - \int \pi_p\c{a} h_\omega}^2 h_\omega,
\]
we have
\begin{align*}
\ind{E_n} \sup_{\Supp\p{\chi_n}} \abs{\sigma^2 - \abs{\nabla \Phi_a}^2} & \leq \ind{E_n} \sup_{A_n'} \abs{\sigma^2 - \abs{\nabla \Phi_a}^2} \\
& = \ind{E_n} \sup_{f \in A_n'} \abs{\int \p{a - \int a f_0}^2 f_0 - \int \p{\pi_p\c{a} - \int \pi_p\c{a} f}^2 f},
\end{align*}
which is less than a multiple of $p^{-\p{\gamma \wedge s}} + \sqrt{\frac{p \ln n}{n}}$, since $\norm{a - \pi_p\c{a}}_\infty \lesssim p^{-\gamma} \norm{a}_{\CC^\gamma}$ and on the event $E_n$ the set $A_n'$ is contained in a set of the form $\b{f \in \HH_p : \norm{f - f_0}_\infty \leq C\varepsilon_n}, \quad \varepsilon_n = p^{-s} + \sqrt{\frac{p \ln n}{n}}$ for some $C > 0$. Since $\Pro_0\c{E_n} \to 1$ this proves that $\sup_{\Supp\p{\chi_n}} \abs{\sigma^2 - \abs{\nabla \Phi_a}^2} = O_p\p{p^{-\p{\gamma \wedge s}} + \sqrt{\frac{p \ln n}{n}}} = o_p(1)$. It is thus sufficient to show that
\[
\sup_{\Supp\p{\chi_n}} \abs{\inner{\nabla \Delta_\pi \Phi_a | \nabla \Phi_a}} = o_p(n)
\]
for both priors, and we do so by distinguishing between the two cases.
\begin{enumerate}
	\item For a prior of type \ref{prior:independent_haar_series_histograms}, using Proposition \ref{proposition:linear_functionals_histograms_2} and $L^2$ orthogonality we have $\partial_\mu \Delta_\pi \Phi_a = - \sigma_\mu^{-2} H''\p{\theta^\mu/\sigma_\mu} \inner{u_\mu | a}_2$, therefore using also $\partial^\mu \Phi_a = \inner{u_\mu | a}_2$,
	\begin{align*}
	\abs{\inner{\nabla \Delta_\pi \Phi_a | \nabla \Phi_a}} & = \abs{\p{\partial^\mu \Phi_a}\p{\partial_\mu \Delta_\pi \Phi_a}} = \abs{- \sum_{\mu = 1}^{p-1} \sigma_\mu^{-2} H''\p{\theta^\mu/\sigma_\mu} \inner{u_\mu | a}_2^2} \leq \norm{H''}_\infty \sum_{\mu = 1}^{p-1} \sigma_\mu^{-2} \inner{u_\mu | a}_2^2,
	\end{align*}
	which is $o_p\p{n}$ by assumption. We can therefore apply Theorem \ref{theorem:generic_bvm}, which implies that
	\[
	d_{BL}\p{\Pi_n \circ Z_n^{-1}, \NN\p{0,\sigma^2}} = O_p\p{p^{-\p{\gamma \wedge s}} + \sqrt{\frac{p \ln n}{n}} + \frac{1}{n} \sum_{\mu = 1}^{p-1} \sigma_\mu^{-2} \inner{u_\mu | a}_2^2 } = o_p(1)
	\]
	where $Z_n = \sqrt{n}\p{\Phi_a - \frac{1}{n} \Div_\pi V - \hat \Psi_0} = \sqrt{n}\p{\Phi_a - \frac{1}{n} \Delta_\pi \Phi_a - \hat \Psi_0}$, which proves the desired result since using Proposition \ref{proposition:linear_functionals_histograms_2} we have
	\[
	\Phi_a - \frac{1}{n} \Delta_\pi \Phi_a = \int \pi_p\c{a} f_\theta + \frac{1}{n} \sum_{\mu = 1}^{p-1} \sigma_\mu^{-1} H'\p{\theta^\mu/\sigma_\mu} \inner{u_\mu | a}_2.
	\]
	\item For a prior of type \ref{prior:dirichlet}, using Proposition \ref{proposition:linear_functionals_histograms_2} and $L^2$ orthogonality we have $\nabla \Delta_\pi \Phi_a = - \norm{m}_1 \p{a - \int ah_\omega}$, which implies that
	\begin{align*}
	\abs{\inner{\nabla \Delta_\pi \Phi_a | \nabla \Phi_a}} & = \abs{P_\omega \c{\p{\nabla \Delta_\pi \Phi_a} \p{\nabla \Phi_a}}} = \norm{m}_1 \abs{P_\omega \c{\p{\pi_p\c{a} - \int \pi_p\c{a} h_\omega}\p{a - \int a h_\omega}}} \\
	& \leq p \norm{m}_\infty \norm{\pi_p\c{a}}_\infty \norm{a}_\infty \leq p \norm{m}_\infty \norm{a}_\infty^2,
	\end{align*}
	which is $o(n)$ since $p \ll n/\ln n$ and $\norm{m}_\infty = O(1)$. We can therefore apply Theorem \ref{theorem:generic_bvm}, which implies that $d_{BL}\p{\Pi_n \circ Z_n^{-1}, \NN\p{0,\sigma^2}} = O_p\p{p^{-\p{\gamma \wedge s}} + \sqrt{\frac{p \ln n}{n}} + \frac{p}{n}} = o_p(1)$ where $Z_n = \sqrt{n}\p{\Phi_a - \frac{1}{n} \Div_\pi V - \hat \Psi_0} = \sqrt{n}\p{\Phi_a - \frac{1}{n} \Delta_\pi \Phi_a - \hat \Psi_0}$, which proves the desired result since using Proposition \ref{proposition:linear_functionals_histograms_2} we have
	\[
	\Phi_a - \frac{1}{n} \Delta_\pi \Phi_a = \int \pi_p\c{a} f_\theta - \frac{\norm{m}_1}{n} \inner{a | h_{\bar m} - h_\omega}_2.
	\]
\end{enumerate}
\end{proof}

\paragraph{Squared norm functional.} We now turn to the estimation of the quadratic functional $\Psi\p{f_0} = \int f_0^2$; for this we first remind the reader of the construction of an efficient estimator using histograms.

\begin{proposition}\label{proposition:efficiency_estimator_squared_norm_histograms}
	If $f_0 \in \CC^s, \quad \frac{1}{4} < s \leq 1, \quad n^{\frac{1}{4s}} \ll p \ll n$ the estimator $\hat \Psi_0 = \int \hat f^2 - \frac{p}{n}$ satisfies
	\[
	\hat \Psi_0 = \int \pi_p\c{f_0}^2 + 2\pi_p\c{f_0}\p{\hat f - \pi_p\c{f_0}} + o_p\p{n^{-1/2}} = \int f_0^2 + \p{\Pro_n - \Pro_0} \c{2f_0} + o_p\p{n^{-1/2}}
	\]
	and
	\[
	\sqrt{n}\p{\hat \Psi_0 - \int f_0^2} \xrightarrow[n \to \infty]{(d)} \NN\p{0,\sigma^2}, \quad \sigma^2 = 4\int\p{f_0 - \int f_0^2}^2 f_0.
	\]
\end{proposition}
\begin{proof}
This is a standard result but we present a proof in Section \ref{section:proofs_histograms} for completeness.
\end{proof}

The asymptotic variance $\sigma^2$ in Proposition \ref{proposition:efficiency_estimator_squared_norm_histograms} is optimal in a semiparametric sense, see e.g. \cite{laurentEfficientEstimationIntegral1996,laurentEstimationIntegralFunctionals1997,vaartAsymptoticStatistics1998}.

\begin{corollary}\label{corollary:squared_norm_histograms}
If $X^n = \p{X_i}_{i=1}^n \iid f_0$ for some density $f_0 \in \CC^s$ and $1 \ll p \ll n/\ln n$ then
\begin{enumerate}
	\item if $\Pi$ is a prior of type \ref{prior:independent_haar_series_histograms} with $H$ twice differentiable, $\norm{H''}_\infty <+\infty$ and $\sigma_- = \min_{0 \leq j \leq J} \sigma_j \gg n^{-1/2}$ then
	\[
	d_{BL}\p{\Pi_n \circ Z_n^{-1}, \quad \NN\p{0,\sigma^2}} = O_p\p{p^{-s} + \sqrt{\frac{p \ln n}{n}} + \frac{1}{n\sigma_-^2}} = o_p(1)
	\]
	\[
	\text{where }Z_n := \sqrt{n}\p{\p{1 + \frac{1}{n}} \int f^2 - \frac{2p}{n} + \frac{1}{n} \sum_{j = 0}^J \sigma_j^{-1} \sum_{k = 1}^{2^j} H'\p{\sigma_j^{-1} \inner{\psi_{jk} | \ln f}_2} \inner{\psi_{jk} | \hat f + f}_2 - \hat \Psi_0},
	\]
	and in particular the BvM phenomenon with centring $\hat \Psi_0$ holds for the functional
	\[
	f \mapsto \p{1 + \frac{1}{n}} \int f^2 - \frac{2p}{n} + \frac{1}{n} \sum_{j = 0}^J \sigma_j^{-1} \sum_{k = 1}^{2^j} H'\p{\sigma_j^{-1} \inner{\psi_{jk} | \ln f}_2} \inner{\psi_{jk} | \hat f + f}_2
	\]
	whenever in addition $s > 1/4$ and $n^{\frac{1}{4s}} \ll p \ll n/\ln n$.
	\item If $\Pi$ is a Dirichlet prior \ref{prior:dirichlet} with parameter $m \in \p{0,\infty}^p$ such that $\norm{m}_\infty = O(1)$ then
	\[
	d_{BL}\p{\Pi_n \circ Z_n^{-1}, \quad \NN\p{0,\sigma^2}} = O_p\p{p^{-s} + \sqrt{\frac{p \ln n}{n}}} = o_p(1)
	\]
	\[
	\text{where }Z_n := \sqrt{n}\p{\p{1 + \frac{1}{n}} \int f^2 - \frac{2p}{n} - \frac{\norm{m}_1}{n} \int \p{h_{\bar m} - f}\p{\hat f + f} - \hat \Psi_0},
	\]
	and in particular the BvM phenomenon with centring $\hat \Psi_0$ holds for the functional
	\[
	f \mapsto \p{1 + \frac{1}{n}} \int f^2 - \frac{2p}{n} - \frac{\norm{m}_1}{n} \int \p{h_{\bar m} - f}\p{\hat f + f}
	\]
	whenever in addition $s > 1/4$ and $n^{\frac{1}{4s}} \ll p \ll n/\ln n$.
\end{enumerate}
\end{corollary}

\begin{remark}
It can be checked that the $\p{1 + \frac{1}{n}}$ factors in the corrected functionals in Corollary \ref{corollary:squared_norm_histograms} can be omitted, as $\frac{1}{n} \int f^2$ and $\frac{p}{n^2}$ are uniformly $o_p\p{n^{-1/2}}$ on a posterior contraction set (for instance here a supremum norm ball of vanishing radius around $\hat f$, see Lemma \ref{lemma:supremum_norm_contraction_rates_histograms}).
\end{remark}

Like Corollaries \ref{corollary:squared_norm_functional_white_noise} \& \ref{corollary:linear_functionals_histograms}, Corollary \ref{corollary:squared_norm_histograms} illustrates clearly the effect of the divergence correction. For both priors, the correction term is the sum of a constant bias $2p/n$ with a non-constant, prior-dependent term, and allows us to relax the assumption $s>1/2$ that a direct application of Corollary \ref{corollary:bvm_without_bias} or the strategy in \cite{castilloBernsteinVonMisesTheorem2015} would require. Using conjugacy and Dirichlet moment calculations it can be checked that the divergence corrected functional has a posterior mean exactly equal to $\hat \Psi_0$ (as explained in Section \ref{section:general_results} in more generality), so that the correction step allows us to completely remove posterior biases. The analogous result for priors of type \ref{prior:independent_haar_series_histograms} confirms that this is not a consequence of conjugacy but of the divergence correction technique itself.

\begin{proof}{Proof of Corollary \ref{corollary:squared_norm_histograms}}\\
	The proof is based on an application of Theorem \ref{theorem:generic_bvm} to the functional $\Phi = \int f_\theta^2 - \frac{p}{n} = \int h_\omega^2 - \frac{p}{n}$ and the vector field
	\[
	V^\mu = \nabla^\mu \Phi + \frac{1}{2} \p{\nabla^{\mu \nu} \Phi - \frac{1}{2} T^{\mu \nu \rho} \partial_\rho \Phi} \Pro_n \c{\partial_\nu \ell}.
	\]
	Moreover we construct the cutoff functions $\chi_n : \Theta \to \c{0,1}$ as in Corollary \ref{corollary:linear_functionals_histograms}, so that in particular $\chi_n = 1$ on $A_n = \b{f \in \HH_p : \norm{f - \hat f}_\infty \leq \varepsilon_n}, \quad \Supp\p{\chi_n} \subset A_n' = \b{f \in \HH_p : \norm{f - \hat f}_\infty \leq 2\varepsilon_n}$ and $\abs{\nabla \chi_n} \leq \norm{\chi'}_\infty \p{p/\varepsilon_n}$ where $\varepsilon_n = M\sqrt{\p{p \ln n}/n}$. Using Proposition \ref{proposition:squared_norm_functional_histograms} we have $V = \hat f + f_\theta - P_\theta \c{\hat f + f_\theta} = \hat f + h_\omega - P_\omega \c{\hat f + h_\omega}$, therefore
	\begin{align*}
	\ind{E_n} \sup_{\Supp\p{\chi_n}} \abs{V} & = \ind{E_n} \sup_{\Supp\p{\chi_n}} P_\theta \c{V^2}^{1/2} \leq \ind{E_n} \sup_{\Supp\p{\chi_n}} P_\theta \c{\p{\hat f + f_\theta}^2}^{1/2} \\
	& \leq \ind{E_n} \sup_{\Supp\p{\chi_n}} P_\theta \c{\p{2\hat f - 2\pi_p\c{f_0} + f_\theta - \hat f}^2}^{1/2} + 2\norm{\pi_p\c{f_0}}_\infty \leq 2\norm{f_0}_\infty + 2\varepsilon_n \leq 4\norm{f_0}_\infty
	\end{align*}
	for large $n$. Therefore by Cauchy--Schwarz inequality we also have $\ind{E_n} \sup_{\Supp\p{\chi_n}} \abs{\inner{\nabla \chi_n | V}} \leq 4\norm{f_0}_\infty\norm{\chi'}_\infty \p{p/\varepsilon_n}$. But $\chi_n = 1$ on $A_n$ and $\ind{E_n} \Pi_n \c{A_n^c} \leq n^{-K}$ with $K$ fixed but that can be taken arbitrarily large if $M$ is itself accordingly chosen, thus $\ind{E_n} \Pi_n \c{\abs{\inner{\nabla \chi_n | V}}} \leq 4\norm{f_0}_\infty\norm{\chi'}_\infty \p{p/\varepsilon_n} \ind{E_n} \Pi_n \c{A_n^c} \leq 4\norm{f_0}_\infty\norm{\chi'}_\infty \p{p/\varepsilon_n} n^{-K}$, which is less than $n^{-K/2}$ for large enough $K$. Since $\Pro_0\c{E_n} \to 1$ this shows in particular that $\Pi_n\c{\abs{\inner{\nabla \chi_n | V}}} = O_p\p{n^{-K/2}} = o_p\p{\sqrt{n}}$. Furthermore by Proposition \ref{proposition:squared_norm_functional_histograms} we have $\hat \Phi = \int \hat f^2 - \frac{p}{n} = \hat \Psi_0$, so that $\sup_{\Supp\p{\chi_n}} \abs{\hat \Phi - \hat \Psi_0} = 0 = o_p(n^{-1/2})$. To conclude we thus need to bound
	\[
	\sup_{\Supp\p{\chi_n}} \abs{\sigma^2 - \inner{\nabla \Phi | V} + \frac{1}{n} \inner{\nabla \Div_\pi V | V}}
	\]
	for both priors. For this we first observe that, using the expressions for $V$ and $\nabla \Phi$ given in Proposition \ref{proposition:squared_norm_functional_histograms},
	\[
	\inner{\nabla \Phi | V} = P_\theta \c{\p{\nabla \Phi}V} = 2\int \p{f_\theta - \int f_\theta^2}\p{\hat f + f_\theta - \int \p{\hat f + f_\theta}f_\theta} f_\theta.
	\]
	and since $\Supp\p{\chi_n} \subset A_n'$, by definition of the sets $A_n',E_n$ and because $\norm{\pi_p\c{f_0} - f_0}_\infty \to 0$ (since $f_0 \in \CC^0$ and $p \gg 1$, we have
	\begin{align*}
	\ind{E_n} \sup_{\Supp\p{\chi_n}} \abs{\sigma^2 - \inner{\nabla \Phi | V}} & = \ind{E_n} \sup_{\Supp\p{\chi_n}} \abs{4\int \p{f_0 - \int f_0^2}^2 - 2\int \p{f_\theta - \int f_\theta^2}\p{\hat f + f_\theta - \int \p{\hat f + f_\theta}f_\theta} f_\theta}
	\end{align*}
	which (as in the proof of Corollary \ref{corollary:linear_functionals_histograms}) is less than a multiple of $p^{-s} + \sqrt{\frac{p \ln n}{n}}$, because $\Supp\p{\chi_n} \subset A_n'$ and on the event $E_n$ the set $A_n'$ is contained in another set of the form $\b{f \in \HH_p : \norm{f - f_0}_\infty \leq C\varepsilon_n}, \quad \varepsilon_n = p^{-s} + \sqrt{\frac{p \ln n}{n}}$ for some $C>0$. Because $\Pro_0 \c{E_n} \to 1$ this proves that $\sup_{\Supp\p{\chi_n}} \abs{\sigma^2 - \inner{\nabla \Phi | V}} = O_p\p{p^{-s} + \sqrt{\frac{p \ln n}{n}}} = o_p(1)$. To conclude it thus suffices to show that for both priors we have
	\[
	\sup_{\Supp\p{\chi_n}} \abs{\inner{\nabla \Div_\pi V | V}} = o_p(n),
	\]
	which we do by distinguishing between the two cases.
	\begin{enumerate}
		\item For a prior of type \ref{prior:independent_haar_series_histograms}, Proposition \ref{proposition:squared_norm_functional_histograms} yields
		\begin{align*}
		& \abs{\inner{\nabla \Div_\pi V | V}} \\
		& \leq 4\p{1 + \norm{H'}_\infty + \norm{H''}_\infty + \sigma_-^{-2}}\p{1 + \norm{f_\theta - \hat f}_\infty + \norm{\hat f - \pi_p\c{f_0}}_\infty + \norm{f_0}_\infty}^2,
		\end{align*}
		therefore, using $\Supp\p{\chi_n} \subset A_n'$ and the definition of the sets $A_n',E_n$,
		\begin{align*}
			& \ind{E_n} \sup_{\Supp\p{\chi_n}} \abs{\inner{\nabla \Div_\pi V | V}} \leq \ind{E_n} \sup_{A_n'} \abs{\inner{\nabla \Div_\pi V | V}} \\
			& \leq 4\p{1 + \norm{H'}_\infty + \norm{H''}_\infty + \sigma_-^{-2}} \ind{E_n} \sup_{A_n'} \p{1 + \norm{f_\theta - \hat f}_\infty + \norm{\hat f - \pi_p\c{f_0}}_\infty + \norm{f_0}_\infty}^2 \\
			& \leq 4\p{1 + \norm{H'}_\infty + \norm{H''}_\infty + \sigma_-^{-2}} \p{1 + \norm{f_0}_\infty + 4\varepsilon_n}^2 \\
			& \leq 4\p{1 + \norm{H'}_\infty + \norm{H''}_\infty + \sigma_-^{-2}} \p{2 + \norm{f_0}_\infty}^2
		\end{align*}
		for large $n$. Because $\Pro_0 \c{E_n} \to 1$ and $\sigma_- \gg n^{-1/2}$ by assumption, we obtain
		\[
		\ind{E_n} \sup_{\Supp\p{\chi_n}} \abs{\inner{\nabla \Div_\pi V | V}} = O_p\p{\sigma_-^{-2}} = o_p(n).
		\]
		We can therefore apply Theorem \ref{theorem:generic_bvm}, which implies that $d_{BL}\p{\Pi_n \circ Z_n^{-1}, \NN\p{0,\sigma^2}} = o_p(1)$ where $Z_n = \sqrt{n}\p{\Phi_a - \frac{1}{n} \Div_\pi V - \hat \Psi_0} $, which proves the desired result since using Proposition \ref{proposition:squared_norm_functional_histograms} we have
		\[
		\Phi_a - \frac{1}{n} \Div_\pi V = \p{1 + \frac{1}{n}} \int f_\theta^2 - \frac{2p}{n} + \frac{1}{n} \sum_{\mu = 1}^{p-1} \sigma_\mu^{-1} H'\p{\theta^\mu/\sigma_\mu} \p{\hat \eta_\mu + \eta_\mu}.
		\]
		\item For a prior of type \ref{prior:dirichlet}, Proposition \ref{proposition:squared_norm_functional_histograms} yields
		\[
		\abs{\inner{\nabla \Div_\pi V | V}} \leq 6p \p{1 + \norm{m}_\infty} \p{1 + \norm{h_\omega - \hat f}_\infty + \norm{\hat f - \pi_p\c{f_0}}_\infty + \norm{f_0}_\infty}^2
		\]
		therefore, using $\Supp\p{\chi_n} \subset A_n'$ and the definition of the sets $A_n',E_n$,
		\begin{align*}
		& \ind{E_n} \sup_{\Supp\p{\chi_n}} \abs{\inner{\nabla \Div_\pi V | V}} \leq \ind{E_n} \sup_{A_n'} \abs{\inner{\nabla \Div_\pi V | V}} \\
		& \leq 6p \p{1 + \norm{m}_\infty} \ind{E_n} \sup_{A_n'} \p{1 + \norm{h_\omega - \hat f}_\infty + \norm{\hat f - \pi_p\c{f_0}}_\infty + \norm{f_0}_\infty}^2 \\
		& \leq 6p \p{1 + \norm{m}_\infty} \p{1 + \norm{f_0}_\infty + 4\varepsilon_n}^2 \leq 6p \p{1 + \norm{m}_\infty} \p{2 + \norm{f_0}_\infty}^2
		\end{align*}
		for large $n$. Because $\norm{m}_\infty = O(1), \quad \Pro_0 \c{E_n} \to 1$ and $p \ll n/\ln n$ in particular this implies
		\[
		\sup_{\Supp\p{\chi_n}} \abs{\inner{\nabla \Div_\pi V | V}} = O_p\p{p} = o_p(n).
		\]
		We can therefore apply Theorem \ref{theorem:generic_bvm}, which implies that $d_{BL}\p{\Pi_n \circ Z_n^{-1}, \NN\p{0,\sigma^2}} = o_p(1)$ where $Z_n = \sqrt{n}\p{\Phi_a - \frac{1}{n} \Div_\pi V - \hat \Psi_0} $, which proves the desired result since using Proposition \ref{proposition:squared_norm_functional_histograms} we have
		\[
		\Phi_a - \frac{1}{n} \Div_\pi V = \p{1 + \frac{1}{n}} \int h_\omega^2 - \frac{2p}{n} - \frac{\norm{m}_1}{n} \int \p{h_{\bar m} - h_\omega}\p{\hat f + h_\omega}.
		\]
	\end{enumerate}
\end{proof}

\paragraph{General integral functionals.} We conclude this section by investigating the case of integral functionals of the form $\int T\p{x,f_0(x)}dx$ for some function $T : \c{0,1} \times \RR \to \RR$, limiting ourselves to the ``no-bias'' case where $s > 1/2$ and $p \ll \p{n/\ln n}^{1/2}$, which allows us to recover the results of \cite{castilloBernsteinVonMisesTheorem2015} by an application of Corollary \ref{corollary:bvm_without_bias}, which yields in addition non-asymptotic quantitative estimates on the underlying bounded Lipschitz distances. In that context an efficient (oracle) estimator of $\int T\p{x,f_0(x)}dx$ is given by
\[
\hat \Psi_0 = \int T\p{x,f_0(x)}dx + \p{\Pro_n - \Pro_0}\c{\frac{\partial T}{\partial f}\p{\cdot,f_0(\cdot)}}
\]
and the efficient asymptotic variance is
\[
\sigma^2 = \int \p{\frac{\partial T}{\partial f}(x,f_0(x)) - \int \frac{\partial T}{\partial f}(x,f_0(x)) f_0(x)dx}^2 f_0(x)dx,
\]
see e.g. \cite{laurentEfficientEstimationIntegral1996,laurentEstimationIntegralFunctionals1997}. We make the following regularity assumption on $T$ :

\begin{assumption}\label{assumption:T}
$T : \c{0,1} \times \RR \to \RR$ is continuous, twice differentiable in its second argument and for some fixed $C, \eta > 0$ and $\gamma > 1/2$ we have
\[
\abs{\frac{\partial T}{\partial f}(x,f) - \frac{\partial T}{\partial f}\p{y,f}} \vee \abs{T(x,f)-T(y,f)} \leq C\abs{x-y}^\gamma, \quad \abs{T\p{x,f}} \vee \abs{\frac{\partial T}{\partial f}\p{x,f}} \vee \abs{\frac{\partial^2 T}{\partial f^2}\p{x,f}} \leq C,
\]
for any $x,y \in \c{0,1}$ and $f \in \p{\inf_{\c{0,1}} f_0 - \eta, \sup_{\c{0,1}} f_0 + \eta}$.
\end{assumption}

\begin{corollary}\label{corollary:integral_functionals_histograms}
	Assume that $s,\gamma > 1/2$, that $T$ satisfies Assumption (\ref{assumption:T}) and $p = \floor{\frac{\sqrt{n}}{\ln^2 n}}$. Then with $Z_n = \sqrt{n}\p{\int T\p{x,f(x)}dx - \hat \Psi_0}$:
	\begin{enumerate}
		\item if $\Pi$ is a prior of type \ref{prior:independent_haar_series_histograms} with $\norm{H'}_\infty < +\infty$ and $\sigma_- := \min_{0 \leq j \leq J} \sigma_j \gg \frac{1}{\sqrt{p} \ln^2 n}$ we have
		\begin{align*}
		d_{BL}\p{\Pi_n \circ Z_n^{-1}, \NN\p{0,\sigma^2}} & = O_p\p{p^{- s \wedge \gamma} \sqrt{p \ln n} + \frac{p \ln n}{\sqrt{n}} + \sqrt{n}p^{-2s} + \sqrt{n}p^{-(s+\gamma)} + \sigma_-^{-1} \sqrt{\frac{p}{n}}} \\
		& = O_p\p{\frac{1}{\ln n}} = o_p(1),
		\end{align*}
		and in particular the BvM phenomenon with centring $\hat \Psi_0$ holds for the functional $f \mapsto \int T\p{x,f(x)}dx$,
		\item if $\Pi$ is a prior of type \ref{prior:dirichlet} with parameter $m \in \p{0,\infty}^p$ such that $\norm{m}_\infty = O(1)$ we have
		\begin{align*}
			d_{BL}\p{\Pi_n \circ Z_n^{-1}, \NN\p{0,\sigma^2}} & = O_p\p{p^{- s \wedge \gamma} \sqrt{p \ln n} + \frac{p \ln n}{\sqrt{n}} + \sqrt{n}p^{-2s} + \sqrt{n}p^{-(s+\gamma)}} \\
			& = O_p\p{\frac{1}{\ln n}} = o_p(1),
		\end{align*}
		and in particular the BvM phenomenon with centring $\hat \Psi_0$ holds for the functional $f \mapsto \int T\p{x,f(x)}dx$.
	\end{enumerate}
\end{corollary}

The condition $\sigma_- \gg \frac{1}{\sqrt{p} \ln^2 n}$ is satisfied for instance with $\sigma_j = 2^{-j/2}$ and the condition $\norm{m}_\infty = O(1)$ essentially coincides with the one considered in \cite{castilloBernsteinVonMisesTheorem2015} in the case of Dirichlet priors. Although not explicitly worked out in \cite{castilloBernsteinVonMisesTheorem2015}, non-quantitative BvM results for priors of type \ref{prior:independent_haar_series_histograms} can be derived in this setting using the Laplace transform approach \cite{castilloBernsteinVonMisesTheorem2015}, but an application of Corollary \ref{corollary:bvm_without_bias} instead allows precisely to obtain these non-asymptotic and quantitative estimates.

Corollary \ref{corollary:integral_functionals_histograms} has a somewhat different purpose from Corollaries \ref{corollary:squared_norm_functional_white_noise} - \ref{corollary:squared_norm_histograms} since we restrict our case to the classical no-bias regime $s \wedge \gamma > 1/2$ via an application of Corollary \ref{corollary:bvm_without_bias} i.e. no functional correction. Compared to the consequences of the Laplace transform approach the interest of the method is therefore quantitative: Corollary \ref{corollary:integral_functionals_histograms} directly implies for both priors a high-probability bound on the bounded Lipschitz distance of the form
\[
d_{BL}\p{\Pi_n \circ Z_n^{-1}, \NN\p{0,\sigma^2}} = O_p\p{\frac{1}{\ln n}},
\]
which is not a direct consequence of the posterior Laplace transform approach. Extending these results to the regime $s>1/4$ (as done in \cite{laurentEfficientEstimationIntegral1996,laurentEstimationIntegralFunctionals1997} using frequentist techniques) would require to design a more refined choice of vector field $V$ and functional $\Phi$ such that $\abs{\hat \Phi - \hat \Psi_0} = o_p\p{n^{-1/2}}$ uniformly on a posterior contraction set and then to consider the divergence corrected functional $\Phi - n^{-1} \Div_\pi V$; this is left for future work.

\begin{proof}
The proof follows as an application of Corollary \ref{corollary:bvm_without_bias} to $\Phi = \int T\p{x,f(x)}dx$ and $V = \nabla \Phi$. For any $K>0$ Lemma \ref{lemma:supremum_norm_contraction_rates_histograms} yields the existence of $M>0$ such that $\Pro_0\c{E_n^c} \leq n^{-K}$ and $\ind{E_n} \Pi_n \c{A_n^c} \leq n^{-K}$ where
\[
E_n = \b{\norm{\hat f - \pi_p\c{f_0}}_\infty \leq \varepsilon_n'}, \quad A_n = \b{f \in \HH_p : \norm{f - \hat f}_\infty \leq \varepsilon_n'}, \quad \varepsilon_n' := M\sqrt{\frac{p \ln n}{n}}.
\]
The cutoff functions $\chi_n$ are then constructed as in the proof of Corollary \ref{corollary:linear_functionals_histograms}; they satisfy $0 \leq \chi_n \leq 1, \quad \chi_n = 1$ on $A_n, \quad \Supp\p{\chi_n} \subset A_n'$ where
\[
A_n' := \b{f \in \HH_p : \norm{f - \hat f}_\infty \leq 2\varepsilon_n'}
\]
and $\abs{\nabla \chi_n}^2 \leq \p{\frac{p}{\varepsilon_n'}}^2 \norm{\chi'}_\infty^2$. Using Proposition \ref{proposition:nabla_phi_integral_functionals_histograms} we have
\[
\nabla \Phi_n = \pi_p \c{ \frac{\partial T}{\partial f}\p{\cdot,f(\cdot)} } - \int \pi_p \c{ \frac{\partial T}{\partial f}\p{\cdot,f(\cdot)} }(x)f(x)dx,
\]
which using assumption \ref{assumption:T} ($T$ is $\CC^2$ in its second argument) together with $f_0 \in \CC^s$ (which implies $\norm{f_0 - \pi_p\c{f_0}}_\infty \to 0$) and the definition of $\sigma^2$ can be seen to imply that
\[
\sup_{f \in A_n'} \abs{\sigma^2 - \abs{\nabla \Phi}^2} \leq c\varepsilon_n, \quad \varepsilon_n = p^{-s} + \sqrt{\frac{p \ln n}{n}}
\]
for some $c>0$. Because $\Supp\p{\chi_n} \subset A_n'$ this shows that condition $2$ of Corollary \ref{corollary:bvm_without_bias} is satisfied. Since $\chi_n = 1$ on $A_n$ and $\Supp\p{\chi_n} \subset A_n'$ we also have $\Supp\p{\nabla \chi_n} \subset A_n^c \cap A_n'$, therefore using Cauchy--Schwarz inequality we obtain
\[
\Pi_n\c{\abs{\inner{\nabla \chi_n|\nabla \Phi}}} \leq \Pi_n \c{\abs{\nabla \chi_n} \abs{\nabla \Phi} \ind{A_n^c \cap A_n'}} \leq \p{\frac{p}{\varepsilon_n'}} \norm{\chi'}_\infty \p{\sigma^2 + c\varepsilon_n}\Pi_n \c{A_n^c},
\]
which implies $\ind{E_n}\Pi_n\c{\abs{\inner{\nabla \chi_n|\nabla \Phi}}} \lesssim \p{p/\varepsilon_n'}n^{-K}$. Taking $K>0$ large enough then yields $\ind{E_n}\Pi_n\c{\abs{\inner{\nabla \chi_n|\nabla \Phi}}} \leq \p{p/\varepsilon_n'}n^{-K/2}$, which shows that $\Pi_n\c{\abs{\inner{\nabla \chi_n|\nabla \Phi}}} = o_p\p{\sqrt{n}}$ since $\Pro_0\c{E_n} \to 1$. To conclude we need to show that condition $1$ of Corollary \ref{corollary:bvm_without_bias} holds. Since $\Supp\p{\chi_n} \subset A_n'$ and $\p{\ln n}^{\frac{1}{2(s \wedge \gamma) - 1}} \ll p \ll \p{n/\ln^2 n}^{1/2}$, Lemma \ref{lemma:bound_hat_Phi_integral_functionals_histograms} yields
\[
\sup_{f \in \Supp\p{\chi_n}} \abs{\hat \Phi - \hat \Psi_0} = O_p\p{p^{-2s} + p^{-(s + \gamma)} + p^{-s \wedge \gamma} \sqrt{\frac{p \ln n}{n}} + \frac{p \ln n}{n}},
\]
in particular $\sup_{f \in \Supp\p{\chi_n}} \abs{\hat \Phi - \hat \Psi_0} = o_p\p{n^{-1/2}}$ using the choice of $p$ and the fact that $s \wedge \gamma > 1/2$. It remains to bound $\sup_{\Supp\p{\chi_n}} \abs{\Div_\pi V} = \sup_{\Supp\p{\chi_n}} \abs{\Delta_\pi \Phi}$ for each prior, and we do so by distinguishing between the two cases.
\begin{enumerate}
	\item For a prior of type \ref{prior:independent_haar_series_histograms}, using Proposition \ref{proposition:delta_pi_Phi_integral_functionals_histograms} we have
	\[
	\Delta_\pi \Phi = \int \frac{\partial^2 T}{\partial f^2}\p{x,f(x)} f(x)\p{p-f(x)} dx - \int \frac{\partial T}{\partial f}\p{x,f(x)}\b{\sum_{\mu = 1}^{p-1} \sigma_\mu^{-1} H'\p{\theta^\mu/\sigma_\mu}u_\mu(x)}dx.
	\]
	Moreover by localisation properties of wavelets (see e.g. \cite{gineMathematicalFoundationsInfiniteDimensional2015}) we have
	\[
	\norm{\sum_{\mu = 1}^{p-1} \sigma_\mu^{-1} H'\p{\theta^\mu/\sigma_\mu}u_\mu}_\infty \lesssim \sum_{j = 0}^J \sigma_j^{-1} 2^{j/2} \norm{H'}_\infty,
	\]
	and since on the event $E_n$ we have $\Supp\p{\chi_n} \subset A_n'$ and $A_n'$ is itself contained in a uniform neighbourhood of $f_0$ with vanishing radius, using Assumption \ref{assumption:T} we find
	\[
	\ind{E_n} \sup_{\Supp\p{\chi_n}} \abs{\Delta_\pi \Phi} \leq c\p{p + \sum_{j = 0}^J 2^{j/2} \sigma_j^{-1}} \leq c\p{p + \sigma_-^{-1} \sqrt{p}}
	\]
	for some fixed constant $c>0$. Because $\Pro_0\c{E_n} \to 1$ and by assumption $\sigma_- \gg \frac{1}{\sqrt{p} \ln^2 p} \asymp \sqrt{\frac{p}{n}}$ this implies that $\sup_{\Supp\p{\chi_n}} \abs{\Delta_\pi \Phi} = O_p\p{\sigma_-^{-1}\sqrt{p}} = o_p\p{\sqrt{n}}$, since $p \ll \sqrt{n}$.
	\item For a prior of type \ref{prior:dirichlet}, using Proposition \ref{proposition:delta_pi_Phi_integral_functionals_histograms} we have
	\[
	\Delta_\pi \Phi = \int f(x)\p{p - f(x)} \frac{\partial^2 T}{\partial f^2}\p{x,f(x)}dx + \norm{m}_1 \int \p{h_{\bar m}(x) - f(x)} \frac{\partial T}{\partial f}\p{x,f(x)}dx,
	\]
	where $\bar m = m/\norm{m}_1$. As in the first case since on the event $E_n$ we have $\Supp\p{\chi_n} \subset A_n'$ and $A_n'$ is itself contained in a uniform neighbourhood of $f_0$ with vanishing radius, using Assumption \ref{assumption:T} together with $\norm{m}_1 \leq p \norm{m}_\infty = O(p)$ we find $\ind{E_n} \sup_{\Supp\p{\chi_n}} \abs{\Delta_\pi \Phi} \leq cp$ for a fixed constant $c>0$, thus showing that $\sup_{\Supp\p{\chi_n}} \abs{\Delta_\pi \Phi} = O_p\p{p} = O_p\p{\sqrt{n}/\ln n}$ since $\Pro_0\c{E_n} \to 1$.
\end{enumerate}
\end{proof}

\section{Simulation study}\label{section:simulation_study}

In this section we illustrate the effect of the divergence correction in the white-noise and density models. We limit ourselves to Gaussian and Dirichlet priors, so that all posterior samples can be generated as exact independent draws from the posterior distributions by conjugacy. Equal-tailed $95\%$ credible intervals are computed from $2\,000$ posterior draws, and their coverage is estimated over $500$ independent datasets. The posterior densities in Figure \ref{figure:simulation_posterior_densities} correspond to representative datasets and are computed from $20\,000$ posterior draws.

\paragraph{White noise.} We take $N=10^4$, $p=\floor{N/\ln N}=1\,085$ and Gaussian product priors $\theta^\mu \sim \NN(0,\sigma_\mu^2)$. We compare the posterior distributions of the uncorrected functional $\norm{\theta}_2^2$, the constant-shift correction (as introduced in Section \ref{section:proofs_white_noise})
\[
\norm{\theta}_2^2-\frac{2p}{N}+\frac{1}{2\sqrt N}\sum_{\mu=1}^p A_\mu,
\qquad
A_\mu=\frac{2(3N\sigma_\mu^2+2)}{\sqrt N(N\sigma_\mu^2+1)^2}
\]
and our divergence-corrected functional
\[
\sum_{\mu=1}^p\left\{\theta^{\mu2}+N^{-1}\sigma_\mu^{-2}\theta^\mu(Y^\mu+\theta^\mu)\right\}-\frac{2p}{N}.
\]
For each dataset we sample directly from
$\theta^\mu\mid Y^\mu\sim\NN\bigl(NY^\mu/(N+\sigma_\mu^{-2}),(N+\sigma_\mu^{-2})^{-1}\bigr)$.

In the first experiment, $\sigma_\mu^2=p/N^{3/2}$, so that $\sigma_\mu\simeq0.0329$, and the truth is $\theta_0^\mu=0.7\sigma_\mu\simeq0.0231$ for $1\leq \mu\leq p$ and zero otherwise. Thus $\norm{\theta_0}_2^2\simeq0.577$ and
\[
\sum_{\mu=1}^p\sigma_\mu^{-2}=10^6=N^{3/2},
\qquad
\sum_{\mu=1}^p\sigma_\mu^{-2}\theta_0^{\mu 2}=0.49p\simeq531.7=5.317\sqrt N.
\]
More generally, along the sequence $p=\floor{N/\ln N}$ the ratios of these sums to $N^{3/2}$ and $N^{1/2}$ are respectively $1$ and asymptotic to $0.49\sqrt N/\ln N\to\infty$. Thus neither term in Condition (\ref{equation:correct_condition_CR15}) is negligible.

In the second experiment, we instead take
\[
\sigma_1^2=1,\qquad \theta_0^1=1,\qquad
\sigma_\mu^2=\frac{p-1}{N^{3/2}-1},\quad \theta_0^\mu=0,\quad 2\leq\mu\leq p.
\]
Then $\norm{\theta_0}_2^2=1$ and, exactly,
\[
\sum_{\mu=1}^p\sigma_\mu^{-2}=N^{3/2},
\qquad
\sum_{\mu=1}^p\sigma_\mu^{-2}\theta_0^{\mu2}=1=o(\sqrt N).
\]
Thus the total prior precision remains at its critical scale, but the signal-dependent term is negligible. The constant-shift and divergence corrections are consequently expected to give similar results.

\paragraph{Density model.} We use the same experiment for the linear and quadratic functionals. We generate $n=1\,000$ observations on $[0,1]$ from
\[
f_0(x)=0.2+1.6x
\]
and use histograms with $p=64$ equal bins and the uniform Dirichlet prior $\omega\sim\Dir(1,\ldots,1)$ on their probabilities. Thus $\norm{m}_1=p$ and $h_{\bar m}=1$. The chosen linear functional is $\Psi_L(f)=\int_0^1xf(x)dx$, with true value $\Psi_L(f_0)=19/30$. We compare it with the divergence corrected functional
\[
\widetilde\Psi_L(f)=\int_0^1xf(x)dx-\frac{p}{n}\int_0^1x\p{1-f(x)}dx.
\]
For the quadratic functional $\Psi_Q(f)=\int_0^1f^2$, whose true value is $\Psi_Q(f_0)=91/75$, the divergence corrected version is
\[
\widetilde\Psi_Q(f)=\left(1+\frac1n\right)\int_0^1f^2-\frac{2p}{n}
-\frac{p}{n}\int_0^1\p{1-f}(\hat f+f).
\]
Writing $N_\mu$ for the observations in bin $I_\mu$, posterior draws are generated exactly from $\omega\mid X^n\sim\Dir(1+N_1,\ldots,1+N_p)$.

\begin{table}[ht]
\centering
\begin{tabular}{llcc}
\hline
Experiment & Posterior functional & Coverage & Monte Carlo s.e. \\
\hline
White-noise quadratic I & Uncorrected & $0.000$ & $0.000$ \\
 & Constant shift & $0.000$ & $0.000$ \\
 & Divergence corrected & $0.970$ & $0.008$ \\
White-noise quadratic II & Uncorrected & $0.000$ & $0.000$ \\
 & Constant shift & $0.952$ & $0.010$ \\
 & Divergence corrected & $0.954$ & $0.009$ \\
Density linear & Uncorrected & $0.846$ & $0.016$ \\
 & Divergence corrected & $0.966$ & $0.008$ \\
Density quadratic & Uncorrected & $0.054$ & $0.010$ \\
 & Divergence corrected & $0.984$ & $0.006$ \\
\hline
\end{tabular}
\caption{Estimated frequentist coverage of equal-tailed $95\%$ posterior credible intervals over $500$ independent datasets.}
\label{table:simulation_coverage}
\end{table}

Table \ref{table:simulation_coverage} and Figure \ref{figure:simulation_posterior_densities} display the same pattern. In the first white-noise experiment, the constant shift does not remove the non-negligible signal-dependent prior bias, while the divergence correction restores approximately nominal coverage. In the second experiment, where this signal-dependent term is negligible, the constant-shift and divergence corrections have nearly identical coverage. The divergence correction also substantially improves coverage for both density functionals.

The datasets displayed in Figure \ref{figure:simulation_posterior_densities} are selected by a fixed rule rather than by visual inspection. For every repetition $r$ and every posterior functional $j$ being compared, let $e_{rj}$ be the posterior mean of that functional minus its true value. Across the $500$ repetitions, we compute separately for each $j$ the median $m_j$ of these errors and their median absolute deviation $s_j$. We then display the dataset minimising
\[
\sum_j\left(\frac{e_{rj}-m_j}{s_j}\right)^2.
\]
The selected dataset therefore has a vector of posterior-mean errors close to the typical error vector across the simulations, simultaneously for all methods, with the division by $s_j$ preventing a functional with a larger error scale from dominating the selection. In each white-noise experiment the vector contains the errors of the three displayed posterior functionals. In the density experiment it contains the uncorrected and divergence-corrected errors for both the linear and quadratic functionals; consequently, one dataset is selected jointly and used in both density panels.

\begin{center}
\centering
\captionsetup{type=figure,hypcap=false}
\includegraphics[width=0.9\textwidth]{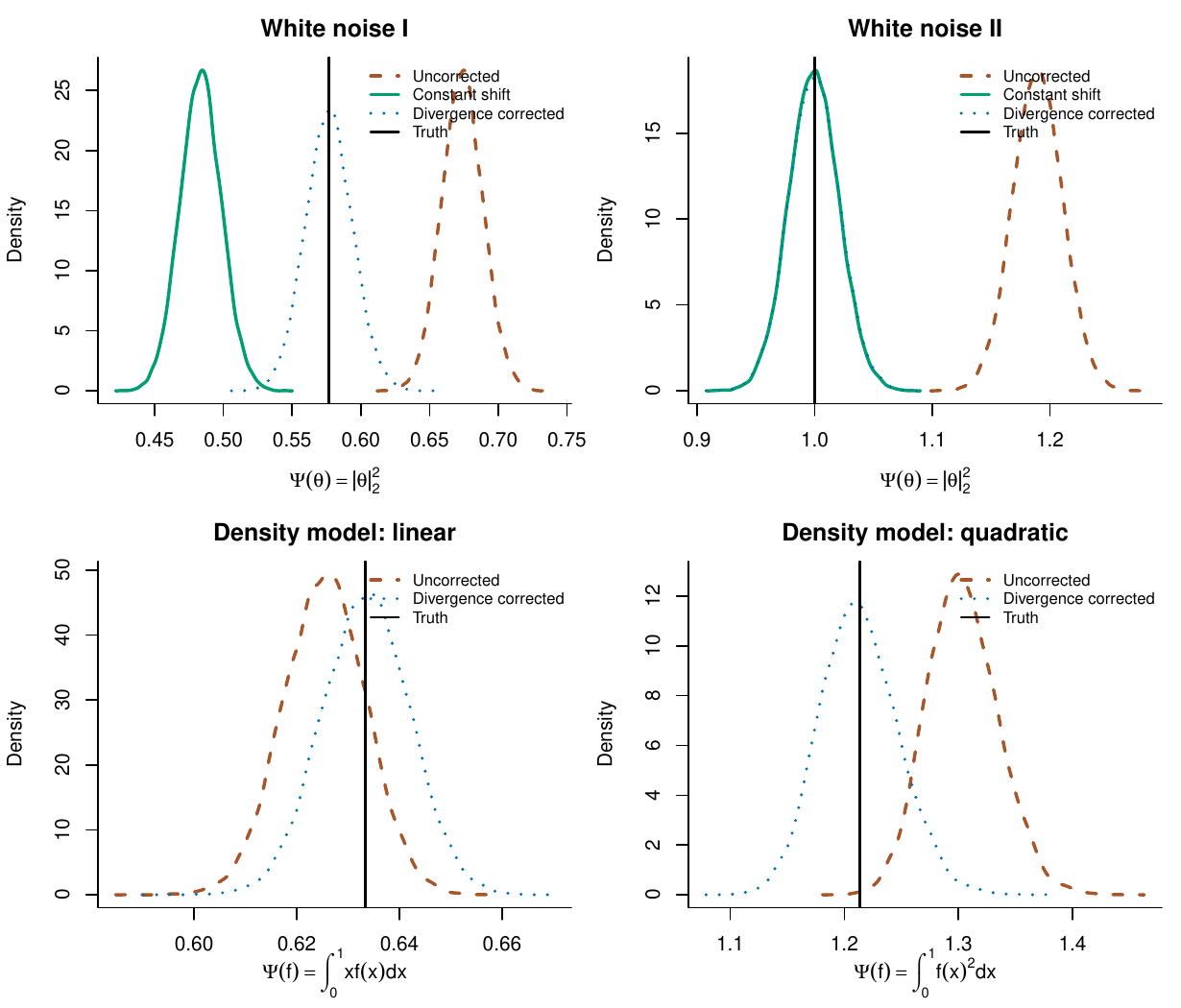}
\caption{Posterior densities for representative datasets in the two white-noise experiments and the density experiment. The vertical black line is the true functional value.}
\label{figure:simulation_posterior_densities}
\end{center}

\section{AI disclosure}

Although the author takes full responsibility for all mathematical claims in this paper, GPT-5.6 and GPT-6 were used to assist with writing, language polishing, and proofreading, as well as with the design and implementation of the simulation study in Section \ref{section:simulation_study}.

\section{Acknowledgments \& funding}

The author is grateful to Judith Rousseau and Isma\"el Castillo for helpful discussions. This work was supported by the Engineering
and Physical Sciences Research Council [Grant Ref: EP/Y028732/1].

\printbibliography

\section{Appendix}

\subsection{Lemma \ref{lemma:laplace_transform}}\label{section:proof_lemma_laplace_transform}

\begin{lemma}\label{lemma:laplace_transform}
Let $n \geq 1$ be fixed, $X^n$ be an observation, $\quad \Phi : \Theta \to \RR$ be a $\CC^1$ functional, $V$ a $\CC^1$ vector field, and $\hat \Psi_0$ an estimator (i.e. $X^n-$measurable real valued random variable). Assume that there exists a sequence $\p{\chi_l}_{l \geq 1}$ of $\CC^1$ compactly supported functions with values in $\c{0,1}$ such that, for any $t > 0$ and $l \geq 1$,
\[
e^{t\abs{\Phi}} \b{1 + \abs{\Div_\pi V} + \abs{\nabla \Phi} + \abs{V} + \abs{\inner{V|\nabla \Phi}} + \abs{\hat \Phi - \hat \Psi_0} + \abs{\inner{\nabla \chi_l | V}}} \in L^1\p{\Pi_n}.
\]
and that
\[
\chi_l \xrightarrow[l \to \infty]{} 1 \text{ pointwise and } \Pi_n \c{e^{t\sqrt{n}\p{\Phi - \hat \Psi_0}} \inner{\nabla \chi_l | V}} \to 0, \text{ for any } t \in \RR.
\]
Then the function $t \mapsto \ln \Pi_n \c{\exp\p{t \sqrt{n}\p{\Phi - \hat \Psi_0}}}$ is differentiable on $\RR$ and we have
\begin{align*}
	\frac{d}{dt} \ln \Pi_n \c{e^{t\sqrt{n}\p{\Phi - \hat \Psi_0}}} & = \Pi_n \c{\sqrt{n}\p{\Phi - \hat \Psi_0} \frac{e^{t\sqrt{n}\p{\Phi - \hat \Psi_0}}}{\Pi_n \c{e^{t\sqrt{n}\p{\Phi - \hat \Psi_0}}}}} \\
	& = \Pi_n \c{\p{t\inner{\nabla \Phi | V} + \sqrt{n} \p{\hat \Phi - \hat \Psi_0 + \frac{1}{n} \Div_\pi V}} \frac{e^{t\sqrt{n}\p{\Phi - \hat \Psi_0}}}{\Pi_n\c{e^{t\sqrt{n}\p{\Phi - \hat \Psi_0}}}}}, \quad t \in \RR,
\end{align*}
where $\hat \Phi = \Phi + \frac{1}{n} \inner{V|\nabla \ell_n}$.
\end{lemma}

\begin{remark}
	In Lemma \ref{lemma:laplace_transform} the sample size $n$ and the observation $X^n$ are fixed; the sequence $\p{\chi_l}_{l \geq 1}$ is only introduced to justify differentiation under the expectation sign and application of the divergence theorem. In particular any potential domination functional that is considered to show the integrability condition can depend on $n$ and $p$ in an arbitrarily bad way. When applying Lemma \ref{lemma:laplace_transform} in our proofs of posterior contraction, we defer all formal justification to Section \ref{section:justifications_lemma_laplace_transform} since these do not add much to the statistical content of this paper. These technical justifications may be skipped by most readers.
\end{remark}

\begin{proof}
In the rest of this proof we set $\tau = \sqrt{n}\p{\Phi - \hat \Psi_0}$. We have, for any $t_0 \geq 1$,
\[
\abs{\frac{\partial}{\partial t} e^{t \tau}} \leq \abs{\tau} e^{t \tau} \leq e^{2t_0 \abs{\tau}} \leq e^{2t_0 \sqrt{n} \abs{\hat \Psi_0}} e^{2t_0 \sqrt{n} \abs{\Phi}}, \quad \abs{t} \leq t_0,
\]
and since $e^{2t_0 \sqrt{n} \abs{\Phi}} \in L^1\p{\Pi_n}$ by assumption, we obtain that $e^{2t_0 \abs{\tau}} \in L^1\p{\Pi_n}$. Therefore $t \mapsto \ln \Pi_n\c{e^{t\tau}}$ is indeed differentiable over $\bigcup_{t_0 \geq 1} \p{-t_0,t_0} = \RR$ and satisfies
\[
\frac{d}{dt} \ln \Pi_n \c{e^{t\tau}} = \Pi_n \c{\tau \frac{e^{t\tau}}{\Pi_n \c{e^{t\tau}}}}.
\]
Since $\inner{V|\nabla \ell_n} = \Div_{\pi_n} V - \Div_\pi V$ we have
\begin{align*}
	\Pi_n \c{\tau e^{t\tau}} & = \Pi_n \c{\sqrt{n}\p{\Phi - \hat \Psi_0} e^{t\tau}} \\
	& = \Pi_n \c{\sqrt{n}\p{\hat \Phi - \hat \Psi_0 - \frac{1}{n}\inner{V|\nabla \ell_n}}e^{t\tau}} \\ \\
	& = \Pi_n \c{\p{\sqrt{n}\p{\hat \Phi + \frac{1}{n} \Div_\pi V - \hat \Psi_0} - \frac{1}{\sqrt{n}} \Div_{\pi_n} V}e^{t\tau}},
\end{align*}
but
\[
\Div_{\pi_n} \p{e^{t\tau} V} = e^{t\tau} \Div_{\pi_n} V + \inner{V | \nabla e^{t\tau}} = e^{t\tau} \p{ \Div_{\pi_n} V + t\sqrt{n} \inner{V|\nabla \Phi}},
\]
therefore
\[
\Pi_n \c{\tau e^{t\tau}} = \Pi_n\c{\p{\sqrt{n}\p{\hat \Phi + \frac{1}{n} \Div_\pi V - \hat \Psi_0} + t \inner{\nabla \Phi | V} } e^{t\tau} - \frac{1}{\sqrt{n}} \Div_{\pi_n}\p{e^{t\tau} V}}.
\]
To conclude it thus suffices to show that
\[
\p{\sqrt{n}\p{\hat \Phi + \frac{1}{n} \Div_\pi V - \hat \Psi_0} + t \inner{\nabla \Phi | V} } e^{t\tau}, \quad \Div_{\pi_n}\p{e^{t\tau} V} \in L^1\p{\Pi_n}
\]
and that $\Pi_n \c{\Div_{\pi_n} \p{e^{t\tau} V}} = 0$, as in that case
\begin{align*}
	\Pi_n\c{\tau e^{t\tau}} & = \Pi_n\c{\p{\sqrt{n}\p{\hat \Phi + \frac{1}{n} \Div_\pi V - \hat \Psi_0} + t \inner{\nabla \Phi | V} } e^{t\tau}} - \Pi_n\c{\frac{1}{\sqrt{n}} \Div_{\pi_n}\p{e^{t\tau} V}} \\
	& = \Pi_n\c{\p{\sqrt{n}\p{\hat \Phi + \frac{1}{n} \Div_\pi V - \hat \Psi_0} + t \inner{\nabla \Phi | V} } e^{t\tau}}.
\end{align*}
First by the triangle inequality $\abs{\tau} \leq \sqrt{n}\p{\abs{\Phi} + \abs{\hat \Psi_0}}$, therefore
\begin{align*}
	\abs{\p{\sqrt{n}\p{\hat \Phi + \frac{1}{n} \Div_\pi V - \hat \Psi_0} + t \inner{\nabla \Phi | V} } e^{t\tau}} & \leq \p{\sqrt{n} + \abs{t}} e^{\abs{t} \sqrt{n} \abs{\hat \Psi_0}} e^{\abs{t}\sqrt{n} \abs{\Phi}}  \b{\abs{\hat \Phi - \hat \Psi_0} + \abs{\Div_\pi V} + \abs{\inner{V|\nabla \Phi}}},
\end{align*}
and the right-hand side of the last display is in $L^1\p{\Pi_n}$ by assumption. Moreover, using the computations above,
\begin{align*}
	\abs{\Div_{\pi_n} \p{e^{t\tau} V}} & = \abs{\Div_{\pi_n} V + t\sqrt{n} \inner{V|\nabla \Phi}} e^{t\tau} = \abs{\Div_\pi V + \inner{V|\nabla \ell_n} + t\sqrt{n} \inner{V|\nabla \Phi}} e^{t\tau} \\
	& = \abs{\Div_\pi V + n\p{\hat \Phi - \Phi} + t\sqrt{n} \inner{V|\nabla \Phi}} e^{t\tau} \\
	& \leq e^{\abs{t}\sqrt{n}\abs{\hat \Psi_0}} \b{\abs{\Div_\pi V} + n \abs{\hat \Phi - \hat \Psi_0} + n \abs{\hat \Psi_0} + n \abs{\Phi} + \abs{t}\sqrt{n} \abs{\inner{V|\nabla \Phi}}} e^{\abs{t}\sqrt{n}\abs{\Phi}} \\
	& \leq e^{\abs{t}\sqrt{n}\abs{\hat \Psi_0}} \b{1 + n + n \abs{\hat \Psi_0} + \abs{t} \sqrt{n}} \b{\abs{\Div_\pi V} + \abs{\hat \Phi - \hat \Psi_0} + \abs{\inner{V|\nabla \Phi}}} e^{\abs{t}\sqrt{n}\abs{\Phi}} \\
	& + n e^{\abs{t}\sqrt{n}\abs{\hat \Psi_0}} e^{\p{n + \abs{t}\sqrt{n}} \abs{\Phi}},
\end{align*}
and the right-hand side of the last display is again in $L^1\p{\Pi_n}$ by assumption. It remains to show that $\Pi_n \c{\Div_{\pi_n} \p{e^{t\tau} V}} = 0$. For this we use the cutoff functions $\chi_l$: for any $l$ since $\chi_l$ is compactly supported the divergence theorem implies that $\Pi_n \c{\Div_{\pi_n}\p{\chi_l e^{t\tau} V}} = 0$. But $\Div_{\pi_n}\p{\chi_l e^{t\tau} V} = \chi_l \Div_{\pi_n}\p{e^{t\tau} V} + e^{t\tau} \inner{\nabla \chi_l | V}$, and since by assumption $e^{t\tau} \inner{\nabla \chi_l | V} \in L^1\p{\Pi_n}$ and we have shown above that $\Div_{\pi_n}\p{e^{t\tau} V} \in L^1\p{\Pi_n}$ we obtain
\[
\Pi_n \c{\chi_l \Div_{\pi_n} \p{e^{t\tau} V}} = - \Pi_n \c{e^{t\tau} \inner{\nabla \chi_l | V}}.
\]
The right-hand side of the last display converges to $0$ by assumption and the left-hand side converges to $\Pi_n \c{\Div_{\pi_n}\p{e^{t\tau} V}}$ by the dominated convergence theorem, which implies the desired result.
\end{proof}

\subsection{Proofs of Section \ref{section:white_noise}}\label{section:proofs_white_noise}

\begin{proposition}\label{proposition:linear_functionals_white_noise}
For a linear functional $\Phi = \inner{a|\theta}_2, \quad a \in \Theta$, we have
\[
\nabla^\mu \Phi = N^{-1} a^\mu, \quad \abs{\nabla \Phi}^2 = N^{-1} \norm{a}_2^2, \quad \Phi + \inner{\nabla \Phi | \nabla \ell} = \inner{a|\pi_p\c{Y}}_2.
\]
\end{proposition}
\begin{proof}
We have $\Phi = \inner{a|\theta}_2 = \sum_{\mu = 1}^p a^\mu \theta^\mu$, therefore $\partial_\mu \Phi = a^\mu$ which implies that $\nabla^\mu \Phi = N^{-1} \delta^{\mu \nu} \partial_\nu \Phi = N^{-1} a^\mu$, which in turn implies $\abs{\nabla \Phi}^2 = N^{-2} a^\mu a^\nu g_{\mu \nu} = N^{-1} \norm{a}_2^2$. Finally
\[
\Phi + \inner{\nabla \Phi | \nabla \ell} = \Phi + g_{\mu \nu}\p{\nabla^\nu \Phi}\p{\nabla^\nu \ell} = \sum_{\mu = 1}^p a^\mu \theta^\mu + N \times N^{-1} a^\mu \times \p{Y^\mu - \theta^\mu} = \inner{a | \pi_p\c{Y}}_2.
\]
\end{proof}

\begin{proposition}\label{proposition:linear_functionals2_white_noise}
Let $a \in \RR^p$ and $\Phi$ be the functional $\Phi = \inner{a|\theta}_2$. Then:
\begin{enumerate}
	\item for a prior of type \ref{prior:independent_series_white_noise} we have
	\[
	\Delta_\pi \Phi = - N^{-1} \sum_{\mu = 1}^p \sigma_\mu^{-1} H'\p{\theta^\mu/\sigma_\mu} a^\mu, \quad \nabla_\mu \Delta_\pi \Phi = - N^{-1} a^\mu \sigma_\mu^{-2} H''\p{\theta^\mu/\sigma_\mu},
	\]
	\item for a prior of type \ref{prior:gaussian} we have
	\[
	\Delta_\pi \Phi = - N^{-1}\sum_{\mu = 1}^p \sigma_\mu^{-2} \theta^\mu a^\mu, \quad \nabla_\mu \Delta_\pi \Phi = - N^{-1} a^\mu \sigma_\mu^{-2}.
	\]
\end{enumerate}
\end{proposition}
\begin{proof}
For a prior of type \ref{prior:independent_series_white_noise}, using Proposition \ref{proposition:linear_functionals_white_noise} we have
\begin{align*}
\Delta_\pi \Phi & = \Delta \Phi + \inner{\nabla \Phi | \nabla \ln \pi} = \partial_\mu \nabla^\mu \Phi + \p{\partial_\mu \ln \Pi}\p{\nabla^\mu \Phi} \\
& = \partial_\mu N^{-1} a^\mu -\sigma_\mu^{-1} H'\p{\theta^\mu/\sigma_\mu} \times N^{-1} a^\mu \\
& = - N^{-1} \sum_{\mu = 1}^p \sigma_\mu^{-1} H'\p{\theta^\mu/\sigma_\mu} a^\mu.
\end{align*}
In turn this also implies
\[
\partial_\mu \Delta_\pi \Phi = - N^{-1} a^\mu \sigma_\mu^{-2} H''\p{\theta^\mu/\sigma_\mu}.
\]
The results for a prior of type \ref{prior:gaussian} follow by taking $H\p{x} = x^2/2$ in the previous computations (even if this choice of $H$ does not define a bona fide prior of type \ref{prior:independent_series_white_noise}, the computations are similar).
\end{proof}

\begin{proposition}\label{proposition:gaussian_concentration_white_noise}
Under model \ref{equation:gaussian_sequence_model} and if $p \ll N/\ln N$, for any $K>0$ there exists $M>0$ such that $\Pro_0 \c{E_N^c} \leq N^{-K}$ where
\[
E_N = \b{\norm{\pi_p\c{Y} - \pi_p\c{\theta_0}}_\infty \leq M \sqrt{\frac{\ln N}{N}}} \leq N^{-K}.
\]
\end{proposition}
\begin{proof}
This is a standard sub-Gaussian concentration result : because $\pi_p\c{Y} - \theta_0 \sim \NN\p{N^{-1} I_p}$, for any $x>0$ by the union bound we have
\[
\Pro_0 \c{\norm{\pi_p\c{Y} - \pi_p\c{\theta_0}}_\infty > x} \leq p\Pro_0 \c{\abs{Z} > \sqrt{N}x},
\]
where $Z$ is a standard Gaussian random variable. Because the Gaussian distribution is sub-Gaussian, this is further bounded by a multiple of $p \exp\p{-cNx^2}$ for some constant $c>0$ \cite{vershyninHighDimensionalProbabilityIntroduction2018}, hence because $p \ll N/\ln N$ (this is actually way stronger than what is needed but is satisfied in our case), given $K>0$ taking $x$ a large enough multiple of $\p{N/\ln N}^{-1/2}$ yields the desired bound.
\end{proof}

\begin{lemma}\label{lemma:contraction_rates_white_noise}
	Under model \ref{equation:gaussian_sequence_model} :
	\begin{enumerate}
		\item if $\Pi$ is a prior of type \ref{prior:independent_series_white_noise} such that $\min_{1 \leq \mu \leq p} \sigma_\mu \geq \frac{c}{\sqrt{N \ln N}}$ for some fixed constant $c>0$ and $1 \ll p \ll N/\ln N$ then for any $K>0$ there exists $M>0$ such that $\Pi_1\c{A_N^c} \leq N^{-K}$ and $\Pro_0\c{E_N^c} \leq N^{-K}$ where
		\[
		A_N := \b{\theta : \norm{\theta - \pi_p\c{Y}}_\infty \leq M \sqrt{\frac{\ln N}{N}}}, \quad E_N = \b{\norm{\pi_p\c{Y} - \pi_p\c{\theta_0}}_\infty \leq M \sqrt{\frac{\ln N}{N}}}.
		\]
		Moreover if $\theta_0 \in H^s$ for some $s>0$ then on the event $E_N$ we have
		\[
		A_N \subset \b{\theta : \norm{\theta - \pi_p\c{\theta_0}}_2 \leq 2M \sqrt{\frac{p \ln N}{N}}} \subset \b{\theta : \norm{\pi_p^*\c{\theta} - \theta_0}_2 \leq C\varepsilon_N},
		\]
		\[
		C := 2M + \norm{\theta_0}_{H^s}, \quad \varepsilon_N := p^{-s} + \sqrt{\frac{p \ln N}{N}}.
		\]
		\item if $\Pi$ is a prior of type \ref{prior:gaussian} such that $\sigma_\mu \geq c\p{N^{-1/2} \vee \abs{\theta_0^\mu}}$ for some fixed $c>0$ and $p \ll N/\ln N$, then for any $K>0$ there exists $M>0$ such that $\Pi_1\c{A_N^c} \leq N^{-K}$ and $\Pro_0 \c{E_N^c} \leq N^{-K}$ where $E_N$ is defined as above and
		\[
		A_N := \c{\theta : \norm{R^{-1}\c{\theta - R^2 \pi_p\c{Y}}}_\infty \leq M \sqrt{\frac{\ln N}{N}}}, \quad R := \Diag\p{\rho}, \quad \rho_\mu := \sqrt{\frac{N\sigma_\mu^2}{N\sigma_\mu^2 + 1}}, \quad 1 \leq \mu \leq p.
		\]
		Moreover if $\theta_0 \in H^s$ for some $s>0$ then with $C = M\sqrt{3} + \norm{\theta_0}_{H^s}$, for large $N$ on the event $E_N$ we have
		\[
		A_N \subset \b{\theta : \norm{\theta - \pi_p\c{\theta_0}}_2 \leq M\sqrt{3\frac{p \ln N}{N}}} \subset \b{\theta : \norm{\pi_p^*\c{\theta} - \theta_0}_2 \leq C\varepsilon_N}, \quad \varepsilon_N := p^{-s} + \sqrt{\frac{p \ln N}{N}}.
		\]
	\end{enumerate}
\end{lemma}
\begin{proof}
	First let us mention that the proof for the Gaussian prior can be obtained directly using conjugacy, but we present here an argument that relies on an application of Lemma \ref{lemma:laplace_transform} in each case for consistency. First, the bound $\Pro_0 \c{E_N^c} \leq N^{-K}$ follows by Proposition \ref{proposition:gaussian_concentration_white_noise}. Let $a \in \RR^p$ and $\Phi$ be the functional $\Phi_a = \inner{a|\theta}_2, \quad V = \nabla \Phi_a$ and $\hat \Psi_{a0} = \inner{a|\pi_p\c{Y}}_2$. We show in Section \ref{section:justifications_lemma_laplace_transform} that Lemma \ref{lemma:laplace_transform} applies in the present context. As a result and since $n = 1$ here,
	\begin{align*}
		\frac{d}{dt} \ln \Pi_n \c{e^{t\sqrt{n}\p{\Phi_a - \hat \Psi_{a0}}}} & = \frac{d}{dt} \ln \Pi_1 \c{e^{t\p{\Phi_a - \hat \Psi_{a0}}}} \\
		& = \Pi_1 \c{\p{\Phi_a - \hat \Psi_{a0}} \frac{e^{t\p{\Phi_a - \hat \Psi_{a0}}}}{\Pi_1 \c{e^{t\p{\Phi_a - \hat \Psi_{a0}}}}}} \\
		& = \Pi_1 \c{\p{t\inner{\nabla \Phi_a | V} + \sqrt{n} \p{\hat \Phi_a - \hat \Psi_{a0} + \frac{1}{n} \Div_\pi V}} \frac{e^{t\sqrt{n}\p{\Phi_a - \hat \Psi_{a0}}}}{\Pi_1\c{e^{t\sqrt{n}\p{\Phi_a - \hat \Psi_{a0}}}}}} \\
		& = \Pi_1 \c{\p{t\abs{\nabla \Phi_a}^2 + \frac{1}{\sqrt{n}} \Delta_\pi \Phi_a} \frac{e^{t\sqrt{n}\p{\Phi_a - \hat \Psi_{a0}}}}{\Pi_1\c{e^{t\sqrt{n}\p{\Phi_a - \hat \Psi_{a0}}}}}} \\
		& = \Pi_1 \c{\p{t\abs{\nabla \Phi_a}^2 + \Delta_\pi \Phi_a} \frac{e^{t\p{\Phi_a - \hat \Psi_{a0}}}}{\Pi_1\c{e^{t\p{\Phi_a - \hat \Psi_{a0}}}}}} , \quad t \in \RR.
	\end{align*}
	We finish by distinguishing between the two priors :
	\begin{enumerate}
		\item for a prior of type \ref{prior:independent_series_white_noise} using Proposition \ref{proposition:linear_functionals_white_noise} we have
		\[
		\abs{\nabla \Phi_a}^2 = N^{-1} \norm{a}_2^2
		\]
		and, using Proposition \ref{proposition:linear_functionals2_white_noise}, with $\sigma_- := \min_{1 \leq \mu \leq p} \sigma_\mu$,
		\[
		\Delta_\pi \Phi_a = - N^{-1} \sum_{\mu = 1}^p \sigma_\mu^{-1} H'\p{\theta^\mu/\sigma_\mu} a^\mu \leq N^{-1} \norm{H'}_\infty \sigma_-^{-1} \norm{a}_1,
		\]
		therefore
		\[
		\frac{d}{dt} \ln \Pi_1 \c{e^{t\p{\Phi_a - \hat \Psi_{a0}}}} \leq tN^{-1} \norm{a}_2^2 + N^{-1} \norm{H'}_\infty \sigma_-^{-1} \norm{a}_1.
		\]
		Integrating this inequality between $0$ and $t > 0$ then yields
		\[
		\ln \Pi_1 \c{e^{t\p{\Phi_a - \hat \Psi_{a0}}}} \leq \frac{t^2}{2} N^{-1} \norm{a}_2^2 + tN^{-1} \norm{H'}_\infty \sigma_-^{-1} \norm{a}_1,
		\]
		and therefore, by Markov' inequality, for any $\mu \in \b{1,\ldots,p}$, taking $a$ equal to $\pm 1$ times the $\mu$-th canonical basis vector,
		\begin{align*}
			\Pi_1 \c{\abs{\theta^\mu - Y^\mu} > x} & = \Pi_1 \c{\abs{\Phi_a - \hat \Psi_{a0}} > x} \\
			& \leq \Pi_1 \c{\Phi_a - \hat \Psi_{a0} > x} +  \Pi_1\c{-\p{\Phi_a - \hat \Psi_{a0}} > x} \\
			& \leq 2\exp\p{- tx + \frac{t^2}{2N} + \frac{t\norm{H'}_\infty}{N\sigma_-}}.
		\end{align*}
		By a union bound this implies
		\[
		\Pi_1 \c{\norm{\theta - \pi_p\c{Y}}_\infty > x} \leq 2p\exp\p{- tx + \frac{t^2}{2N} + \frac{t\norm{H'}_\infty}{N\sigma_-}},
		\]
		and setting $t = Nx$ yields
		\[
		\Pi_1 \c{\norm{\theta - \pi_p\c{Y}}_\infty > x} \leq 2p \exp\p{-Nx^2/2 + \frac{x\norm{H'}_\infty}{\sigma_-}}
		\]
		which (because $p \ll N/\ln N$) for any $K>0$ can be made less than $N^{-K}$ for $x$ a large enough multiple of $\sqrt{\frac{\ln N}{N}}$, using the assumption on $\sigma_-$. This proves the first statement. For the second one we simply use the inequalities
		\begin{align*}
			\norm{\theta - \pi_p\c{\theta_0}}_2 & \leq \norm{\theta - \pi_p\c{Y}}_2 + \norm{\pi_p\c{Y} - \pi_p\c{\theta_0}}_2 \\
			& \leq \sqrt{p} \norm{\theta - \pi_p\c{Y}}_\infty + \sqrt{p} \norm{\pi_p\c{Y} - \pi_p\c{\theta_0}}_\infty,
		\end{align*}
		and $\norm{\pi_p^*\c{\theta} - \theta_0}_2 \leq \norm{\theta - \pi_p\c{\theta_0}}_2 + p^{-s}\norm{\theta_0}_{H^s}$.
		\item For a Gaussian prior \ref{prior:gaussian} Proposition \ref{proposition:linear_functionals2_white_noise} yields
		\[
		\Delta_\pi \Phi_a = - N^{-1} \sum_{\mu = 1}^p \sigma_\mu^{-2} \theta^\mu a^\mu,
		\]
		so that with $a$ equal to the $\mu$-th canonical basis vector of $\RR^p$ we obtain
		\[
		\Delta_\mu \Phi_a = - N^{-1} \sigma_\mu^{-2} \theta^\mu = - N^{-1} \sigma_\mu^{-2} \Phi_a.
		\]
		As a result
		\begin{align*}
			\frac{d}{dt} \ln \Pi_1 \c{e^{t\p{\Phi_a - \hat \Psi_{a0}}}} & = \Pi_1 \c{\p{\Phi_a - \hat \Psi_{a0}} \frac{e^{t\p{\Phi_a - \hat \Psi_{a0}}}}{\Pi_1\c{e^{t\p{\Phi_a - \hat \Psi_{a0}}}}}} \\
			& = \Pi_1 \c{\p{t\abs{\nabla \Phi_a}^2 + \Delta_\pi \Phi_a} \frac{e^{t\p{\Phi_a - \hat \Psi_{a0}}}}{\Pi_1\c{e^{t\p{\Phi_a - \hat \Psi_{a0}}}}}} \\
			& = t\Pi_1 \c{\abs{\nabla \Phi_a}^2 \frac{e^{t\p{\Phi_a - \hat \Psi_{a0}}}}{\Pi_1\c{e^{t\p{\Phi_a - \hat \Psi_{a0}}}}}} - N^{-1} \sigma_\mu^{-2} \Pi_1 \c{ \Phi_a \frac{e^{t\p{\Phi_a - \hat \Psi_{a0}}}}{\Pi_1\c{e^{t\p{\Phi_a - \hat \Psi_{a0}}}}}} \\
			& = t\Pi_1 \c{\abs{\nabla \Phi_a}^2 \frac{e^{t\p{\Phi_a - \hat \Psi_{a0}}}}{\Pi_1\c{e^{t\p{\Phi_a - \hat \Psi_{a0}}}}}} \\
			& - N^{-1} \sigma_\mu^{-2} \Pi_1 \c{ \p{\Phi_a - \hat \Psi_{a0}} \frac{e^{t\p{\Phi_a - \hat \Psi_{a0}}}}{\Pi_1\c{e^{t\p{\Phi_a - \hat \Psi_{a0}}}}}} - N^{-1} \sigma_\mu^{-2} \hat \Psi_{a0} \\
			& = t\Pi_1 \c{\abs{\nabla \Phi_a}^2 \frac{e^{t\p{\Phi_a - \hat \Psi_{a0}}}}{\Pi_1\c{e^{t\p{\Phi_a - \hat \Psi_{a0}}}}}} - N^{-1} \sigma_\mu^{-2} \frac{d}{dt} \ln \Pi_1 \c{e^{t\p{\Phi_a - \hat \Psi_{a0}}}} - N^{-1} \sigma_\mu^{-2} \hat \Psi_{a0},
		\end{align*}
		which implies that
		\[
		\frac{d}{dt} \ln \Pi_1 \c{e^{t\p{\Phi_a - \hat \Psi_{a0}}}} = \p{1 + \frac{1}{N\sigma_\mu^2}}^{-1} \c{t\Pi_1 \c{\abs{\nabla \Phi_a}^2 \frac{e^{t\p{\Phi_a - \hat \Psi_{a0}}}}{\Pi_1\c{e^{t\p{\Phi_a - \hat \Psi_{a0}}}}}} - N^{-1} \sigma_\mu^{-2} \hat \Psi_{a0}},
		\]
		i.e., using Proposition \ref{proposition:linear_functionals_white_noise},
		\[
		\frac{d}{dt} \ln \Pi_1 \c{e^{t\p{\inner{\theta|a}_2 - \inner{\pi_p\c{Y}|a}_2}}} = \frac{N\sigma_\mu^2}{N \sigma_\mu^2 + 1} \frac{t}{N} \norm{a}_2^2 - \frac{1}{N \sigma_\mu^2 + 1} \inner{\pi_p\c{Y}|a}_2,
		\]
		and taking $a$ equal to the $\mu$-th canonical basis vector we obtain
		\[
		\frac{d}{dt} \ln \Pi_1 \c{e^{t\p{\theta^\mu - Y^\mu}}} = \frac{N\sigma_\mu^2}{N \sigma_\mu^2 + 1} \frac{t}{N} - \frac{1}{N \sigma_\mu^2 + 1} Y^\mu,
		\]
		i.e.
		\[
		\ln \Pi_1 \c{e^{t\p{\theta^\mu - \rho_\mu^2 Y^\mu}}} = \rho_\mu^2 \frac{t^2}{2N}, \quad \rho_\mu^2 := \frac{N\sigma_\mu^2}{N \sigma_\mu^2 + 1}.
		\]
		Obviously this could have been derived directly since by conjugacy $\theta | Y \sim \bigotimes_{\mu = 1}^p \NN\p{\rho_\mu^2 Y^\mu, \frac{\rho_\mu^2}{N}}$. Reasoning analogously for $-\Phi_a$, by Markov' inequality and a union bound we obtain
		\[
		\Pi_1\c{\exists \mu : \abs{\theta^\mu - \rho_\mu^2 Y^\mu} > \rho_\mu x} \leq 2p \exp\p{-t\rho_\mu x + \frac{\rho_\mu^2 t^2}{2N}} \leq 2N \exp\p{-t\rho_\mu x + \frac{\rho_\mu^2 t^2}{2N}},
		\]
		which can be made less than $N^{-K}$ via $t = Nx/\rho_\mu$ and $x$ a large enough multiple of $\sqrt{\frac{\ln N}{N}}$, i.e. for some $M>0$
		\[
		\Pi_1\c{\exists \mu : \abs{\theta^\mu - \rho_\mu^2 Y^\mu} > M \rho_\mu \sqrt{\frac{\ln N}{N}}} \leq N^{-K},
		\]
		which proves the first result. To conclude we bound $\norm{\pi_p^*\c{\theta} - \theta_0}_2$ as follows : for any $1 \leq q \leq p$ decompose
		\[
		\norm{\pi_p^*\c{\theta} - \theta_0}_2 \leq \underbrace{\norm{\theta - R^2 \pi_p\c{Y}}_2}_{=: I} + \underbrace{\norm{R^2 \pi_p\c{Y} - R^2 \pi_p\c{\theta_0}}_2}_{=: II} + \underbrace{\norm{R^2 \pi_p\c{\theta_0} - \pi_p\c{\theta_0}}_2}_{=III} + \underbrace{\norm{\theta_0 - \pi_p\c{\theta_0}}_2}_{=: IV},
		\]
		and $\norm{\theta - \pi_p\c{\theta_0}}_2 \leq I + II + III$. First, $IV$ is bounded above by $\norm{\theta_0}_{H^s} p^{-s}$. Moreover by definition of $A_N$
		\[
		\sup_{\theta \in A_N} \norm{I}_2^2 \leq M^2 \sum_{\mu = 1}^p \rho_\mu^2 \frac{\ln N}{N} \leq M^2 \p{\ln N} \sum_{\mu = 1}^p N^{-1} \wedge \sigma_\mu^2 = M^2 \frac{p \ln N}{N},
		\]
		and similarly
		\[
		\ind{E_N} \norm{II}_2^2 \leq M^2 \frac{\ln N}{N} \sum_{\mu = 1}^p \rho_\mu^2 \leq M^2 \p{\ln N} \sum_{\mu = 1}^p N^{-1} \wedge \sigma_\mu^2 \leq M^2 \frac{p \ln N}{N}
		\]
		by definition of the event $E_N$. Finally
		\[
		\norm{III}_2^2 = \sum_{\mu = 1}^p \frac{\theta_0^{\mu 2}}{\p{N\sigma_\mu^2 + 1}^2} \leq c^{-2} \sum_{\mu = 1}^p \frac{\sigma_\mu^2}{\p{N\sigma_\mu^2 + 1}^2} \leq \frac{c^{-2}p}{N},
		\]
		therefore for large $N$ and on the event $E_N$ we have
		\[
		A_N \subset \b{\theta : \norm{\theta - \pi_p\c{\theta_0}}_2 \leq M\sqrt{\frac{3p\ln N}{N}}} \subset \b{\theta : \norm{\pi_p^*\c{\theta} - \theta_0}_2 \leq \p{M\sqrt{3} + \norm{\theta_0}_{H^s}} \varepsilon_N}.
		\]
	\end{enumerate}
\end{proof}

\begin{proposition}\label{proposition:quadratic_functional_white_noise}
With $\Phi = \norm{\theta}_2^2$ and $V$ the random vector field $V = \nabla \Phi + \frac{1}{2} \nabla_{\mu \nu}^2 \Phi. \nabla \ell_1$ we have
\begin{enumerate}
	\item $V_\mu = Y^\mu + \theta^\mu$,
	\item for a prior of type \ref{prior:independent_series_white_noise}
	\[
	\Div_\pi V = \frac{p}{N} - N^{-1} \sum_{\mu = 1}^p \sigma_\mu^{-1} \p{Y^\mu + \theta^\mu} H'\p{\theta^\mu/\sigma_\mu},
	\]
	\[
	\quad \nabla_\mu \Div_\pi V = - N^{-1} \sigma_\mu^{-1} \b{(Y^\mu + \theta^\mu)\sigma_\mu^{-1} H''\p{\theta^\mu/\sigma_\mu} + H'\p{\theta^\mu/\sigma_\mu}},
	\]
	\item for a prior of type \ref{prior:gaussian}
	\[
	\Div_\pi V = \frac{p}{N} - N^{-1} \sum_{\mu = 1}^p \sigma_\mu^{-2} \p{Y^\mu + \theta^\mu}\theta^\mu,
	\]
	\[
	\quad \nabla_\mu \Div_\pi V = - N^{-1} \sigma_\mu^{-2} \p{Y^\mu + 2\theta_\mu}.
	\]
\end{enumerate}
\end{proposition}
\begin{proof}
\begin{enumerate}
	\item Since $\partial_\mu \Phi = 2\theta_\mu$ we have
	\[
	V_\mu = \nabla_\mu \Phi + \frac{1}{2} \nabla_{\mu \nu}^2 \Phi. \nabla^\nu \ell_1 = \partial_\mu \Phi + \frac{1}{2} \p{\partial_{\mu \nu}^2 \Phi} \p{\partial^\nu \ell_1} = 2\theta^\mu + \delta_{\mu \nu} \p{Y^\nu - \theta^\nu} = Y^\mu + \theta^\mu.
	\]
	\item By definition
	\[
	\Div_\pi V = \Div V + \inner{V | \nabla \ln \pi} = \partial_\mu V^\mu + \b{\Gamma_{~\mu \nu}^\mu + \partial_\mu \ln \pi}V^\nu = \partial_\mu V^\mu + \p{\partial_\mu \ln \Pi} V^\mu,
	\]
	but using $g^{\mu \nu} = N^{-1} \delta^{\mu \nu}$ we have $V^\mu = g^{\mu \nu} V_\nu = N^{-1}\p{Y^\mu + \theta^\mu}$, which implies that
	\[
	\Div_\pi V = \frac{p}{N} - N^{-1} \sum_{\mu = 1}^p \sigma_\mu^{-1} \p{Y^\mu + \theta^\mu} H'\p{\theta^\mu/\sigma_\mu}.
	\]
	and
	\[
	\partial_\mu \Div_\pi V = - N^{-1} \sigma_\mu^{-1} \b{(Y^\mu + \theta^\mu)\sigma_\mu^{-1} H''\p{\theta^\mu/\sigma_\mu} + H'\p{\theta^\mu/\sigma_\mu}}.
	\]
	\item The result follows by setting $H\p{x} = x^2/2$ in the last item (even if a Gaussian prior \ref{prior:gaussian} is not a bona fide prior of type \ref{prior:independent_series_white_noise} the computations are similar).
\end{enumerate}
\end{proof}

\paragraph{On the proof of \cite{castilloBernsteinVonMisesTheorem2015} Theorem $(3.1)$ in the case of Gaussian priors.} There is an error in the proof: on page $5$ of the supplementary material, the integral should be (in the notations used there) over $\psi_n\p{B_n}$ not $\psi^{-1}\p{B_n}$, and the argument of the prior density $\varphi$ after change of variable should be $\frac{g_k + \delta_k}{c_t \sigma_k}$, not $\frac{c_t^{-1} g_k - \delta_k}{\sigma_k}$. This propagates to page $6$, so that the correct expression for $B_k\p{\epsilon_k}$ actually is
\begin{equation}\label{equation:correct_B_k}
	B_k\p{\epsilon_k} = \frac{\int e^{-\frac{u^2}{2} + \epsilon_k u} \varphi\p{\frac{f_{0k} + u/\sqrt{n} + \delta_k}{c_t \sigma_k}} du}{\int e^{-\frac{v^2}{2} + \epsilon_k v} \varphi\p{\frac{f_{0k} + v/\sqrt{n}}{\sigma_k}}dv},
\end{equation}
not
\begin{equation}\label{equation:incorrect_B_k}
	B_k\p{\epsilon_k} = \frac{\int e^{-\frac{u^2}{2} + \epsilon_k u} \varphi\p{\frac{c_t^{-1}\p{f_{0k} + u/\sqrt{n}} - \delta_k}{\sigma_k}} du}{\int e^{-\frac{v^2}{2} + \epsilon_k v} \varphi\p{\frac{f_{0k} + v/\sqrt{n}}{\sigma_k}}dv},
\end{equation}
and to conclude the proof one has to show that
\[
\prod_{k = 1}^{K_n} B_k\p{\epsilon_k} = 1 + o_p(1), \text{ equivalently } \abs{\sum_{k = 1}^{K_n} \ln B_k\p{\epsilon_k}} = o_p(1).
\]
To show this in the Gaussian case where $\varphi\p{x} = \p{2\pi}^{-1/2}\exp\p{-x^2/2}$ we carefully complete the squares in the integrals defining $B_k\p{\epsilon_k}$: using
\[
\int \exp\p{-Ax^2 + Bx} dx = \exp\p{B^2/4A}\sqrt{\frac{\pi}{A}}
\]
we have
\begin{align*}
\int e^{-\frac{v^2}{2} + \epsilon_k v} \varphi\p{\frac{f_{0k} + v/\sqrt{n}}{\sigma_k}}dv & = \p{2\pi}^{-1/2} \int \exp\p{-\frac{v^2}{2} + \epsilon_k v - \frac{1}{2}\p{\frac{f_{0k} + v/\sqrt{n}}{\sigma_k}}^2} dv \\
& = \p{2\pi}^{-1/2} \int \exp\p{-\frac{v^2}{2} + \epsilon_k v - \frac{1}{2\sigma_k^2} \b{f_{0k}^2 + \frac{2f_{0k} v}{\sqrt{n}} + \frac{v^2}{n}}} dv \\
& = \p{2\pi}^{-1/2} \exp\p{- \frac{f_{0k}^2}{2\sigma_k^2}} \int \exp\p{-\frac{1}{2}\b{1 + \frac{1}{n\sigma_k^2} } v^2 + \p{\epsilon_k - \frac{f_{0k}}{\sigma_k^2 \sqrt{n}}}v} dv \\
& = \p{2\pi}^{-1/2} \exp\p{- \frac{f_{0k}^2}{2\sigma_k^2}} \exp\p{\frac{\p{\epsilon_k - \frac{f_{0k}}{\sigma_k^2 \sqrt{n}}}^2}{4\frac{1}{2} \b{1 + \frac{1}{n\sigma_k^2}}}} \p{\frac{\pi}{\frac{1}{2} \b{1 + \frac{1}{n\sigma_k^2}}}}^{1/2},
\end{align*}
and (recognising the same expression as above with $\sigma_k$ replaced by $c_t \sigma_k$ and $f_{0k}$ replaced by $f_{0k} + \delta_k$)
\begin{align*}
\int e^{-\frac{u^2}{2} + \epsilon_k u} \varphi\p{\frac{f_{0k} + u/\sqrt{n} + \delta_k}{c_t \sigma_k}} du & = \p{2\pi}^{-1/2} \int \exp\p{-\frac{u^2}{2} + \epsilon_k u - \frac{1}{2} \p{\frac{f_{0k} + u/\sqrt{n} + \delta_k}{c_t \sigma_k}}^2} du \\
& = \p{2\pi}^{-1/2} \exp\p{- \frac{\p{f_{0k} + \delta_k}^2}{2c_t^2\sigma_k^2}} \exp\p{\frac{\p{\epsilon_k - \frac{\p{f_{0k} + \delta_k}}{c_t^2\sigma_k^2 \sqrt{n}}}^2}{4\frac{1}{2} \b{1 + \frac{1}{nc_t^2\sigma_k^2}}}} \p{\frac{\pi}{\frac{1}{2} \b{1 + \frac{1}{nc_t^2\sigma_k^2}}}}^{1/2}.
\end{align*}
Thus the ratio $B_k\p{\epsilon_k}$ can be expressed as
\begin{align*}
B_k\p{\epsilon_k} & = \exp\p{\frac{f_{0k}^2}{2\sigma_k^2} - \frac{\p{f_{0k} + \delta_k}^2}{2c_t^2\sigma_k^2}} \exp\p{\frac{\p{\epsilon_k - \frac{\p{f_{0k} + \delta_k}}{c_t^2\sigma_k^2 \sqrt{n}}}^2}{2 \b{1 + \frac{1}{nc_t^2\sigma_k^2}}} - \frac{\p{\epsilon_k - \frac{f_{0k}}{\sigma_k^2 \sqrt{n}}}^2}{2 \b{1 + \frac{1}{n\sigma_k^2}}}} \p{\frac{1 + \frac{1}{n\sigma_k^2}}{1 + \frac{1}{nc_t^2 \sigma_k^2}}}^{1/2} \\
& = c_t \p{\frac{\sigma_k^2 + n^{-1}}{c_t^2 \sigma_k^2 + n^{-1}}}^{1/2} \exp\p{\frac{f_{0k}^2}{2\sigma_k^2} - \frac{\p{f_{0k} + \frac{t}{\sqrt{n}} Y_k}^2}{2c_t^2\sigma_k^2}} \exp\p{\frac{\p{\epsilon_k - \frac{\p{f_{0k} + \frac{t}{\sqrt{n}} Y_k}}{c_t^2\sigma_k^2 \sqrt{n}}}^2}{2 \b{1 + \frac{1}{nc_t^2\sigma_k^2}}} - \frac{\p{\epsilon_k - \frac{f_{0k}}{\sigma_k^2 \sqrt{n}}}^2}{2 \b{1 + \frac{1}{n\sigma_k^2}}}}
\end{align*}
where we have used $Y_k = f_{0k} + n^{-1/2} \epsilon_k \iff \delta_k = \frac{t}{\sqrt{n}} Y_k$. But
\begin{align*}
- \frac{f_{0k}^2}{2\sigma_k^2} + \frac{\p{\epsilon_k - \frac{f_{0k}}{\sigma_k^2 \sqrt{n}}}^2}{2 \b{1 + \frac{1}{n\sigma_k^2}}} & = - \frac{\p{Y_k - n^{-1/2}\epsilon_k}^2}{2\sigma_k^2} + \frac{\p{\epsilon_k - \frac{Y_k - n^{-1/2} \epsilon_k}{\sigma_k^2 \sqrt{n}}}^2}{2 \b{1 + \frac{1}{n\sigma_k^2}}} \\
& = - \frac{\p{Y_k - n^{-1/2}\epsilon_k}^2}{2\sigma_k^2} + \frac{\p{\p{1 + \frac{1}{n\sigma_k^2}} \epsilon_k - \frac{Y_k}{\sigma_k^2 \sqrt{n}}}^2}{2 \b{1 + \frac{1}{n\sigma_k^2}}} \\
& = - \frac{Y_k^2}{2\sigma_k^2} + \frac{Y_k \epsilon_k}{\sigma_k^2\sqrt{n}} - \frac{\epsilon_k^2}{2n\sigma_k^2} + \frac{1}{2}\p{1 + \frac{1}{n\sigma_k^2}} \epsilon_k^2 - \frac{\epsilon_kY_k}{\sigma_k^2 \sqrt{n}} + \frac{Y_k^2/\p{n\sigma_k^4}}{2\b{1 + \frac{1}{n\sigma_k^2}}} \\
& = \frac{\epsilon_k^2}{2} - \frac{1}{2} \frac{Y_k^2}{\sigma_k^2 + n^{-1}},
\end{align*}
and similarly, exchanging the roles of $\sigma_k$ and $c_t \sigma_k$ and the ones of $f_{0k}$ and $f_{0k} + \frac{t}{\sqrt{n}} Y_k$,
\[
- \frac{\p{f_{0k} + \frac{t}{\sqrt{n}} Y_k}^2}{2c_t^2\sigma_k^2} + \frac{\p{\epsilon_k - \frac{\p{f_{0k} + \frac{t}{\sqrt{n}} Y_k}}{c_t^2\sigma_k^2 \sqrt{n}}}^2}{2 \b{1 + \frac{1}{nc_t^2\sigma_k^2}}} = \frac{\epsilon_k^2}{2} - \frac{1}{2} \frac{\p{Y_k + \frac{t}{\sqrt{n}} Y_k}^2}{c_t^2 \sigma_k^2 + n^{-1}} = \frac{\epsilon_k^2}{2} - \frac{1}{2} \frac{\p{1 + \frac{t}{\sqrt{n}}}^2 Y_k^2}{c_t^2 \sigma_k^2 + n^{-1}}.
\]
Therefore
\begin{align*}
B_k\p{\epsilon_k} & = \p{\frac{c_t^2\p{\sigma_k^2 + n^{-1}}}{c_t^2 \sigma_k^2 + n^{-1}}}^{1/2} \exp\p{-\frac{\epsilon_k^2}{2} + \frac{1}{2} \frac{Y_k^2}{\sigma_k^2 + n^{-1}} + \frac{\epsilon_k^2}{2} - \frac{1}{2} \frac{\p{1 + \frac{t}{\sqrt{n}}}^2 Y_k^2}{c_t^2 \sigma_k^2 + n^{-1}}} \\
& = \p{\frac{n\sigma_k^2 + 1}{n\sigma_k^2 + c_t^{-2}}}^{1/2} \exp\p{\frac{1}{2} Y_k^2 \b{ \frac{1}{\sigma_k^2 + n^{-1}} - \frac{\p{1 + \frac{t}{\sqrt{n}}}^2}{c_t^2 \sigma_k^2 + n^{-1}}}},
\end{align*}
which implies
\[
\ln B_k\p{\epsilon_k} = -\frac{1}{2} \ln \p{\frac{n\sigma_k^2 + c_t^{-2}}{n\sigma_k^2 + 1}} - \frac{1}{2} Y_k^2 \b{\frac{\p{1 + \frac{t}{\sqrt{n}}}^2}{c_t^2 \sigma_k^2 + n^{-1}} - \frac{1}{\sigma_k^2 + n^{-1}}},
\]
where we recall that $c_t = 1 - \frac{t}{\sqrt{n}}$. Since $t$ is fixed in the argument we have, with all big $O$ being uniform in $n$ and $k \in \b{1,\ldots,K_n}$,
\[
\frac{n\sigma_k^2 + c_t^{-2}}{n\sigma_k^2 + 1} = \frac{n\sigma_k^2 + \p{1 - \frac{t}{\sqrt{n}}}^{-2}}{n\sigma_k^2 + 1} = \frac{n\sigma_k^2 + \p{1 + \frac{2t}{\sqrt{n}} + O\p{1/n}}}{n\sigma_k^2 + 1} = 1 + \frac{2}{\sqrt{n}\p{n\sigma_k^2 + 1}} t + O\p{\frac{\sigma_k^{-2}}{n^2}}.
\]
where we have used $\sigma_k \geq cn^{-1/2}$, a fact that we also use repeatedly below. Moreover
\begin{align*}
\frac{1}{c_t^2 \sigma_k^2 + n^{-1}} & = \frac{1}{\sigma_k^2 + n^{-1}} \frac{1}{1 + \frac{\sigma_k^2}{\sigma_k^2 + n^{-1}}\p{\p{1-\frac{t}{\sqrt{n}}}^2 - 1}} = \frac{1}{\sigma_k^2 + n^{-1}} \frac{1}{1 - \frac{\sigma_k^2}{\sigma_k^2 + n^{-1}}\p{\frac{2t}{\sqrt{n}} + O\p{\frac{1}{n}}}} \\
& = \frac{1}{\sigma_k^2 + n^{-1}} \p{1 + \frac{2\sigma_k^2}{\sqrt{n}\p{\sigma_k^2 + n^{-1}}}t + O\p{\frac{1}{n}} } = \frac{1}{\sigma_k^2 + n^{-1}} + \frac{2\sigma_k^2}{\sqrt{n}\p{\sigma_k^2 + n^{-1}}^2}t + O\p{\frac{\sigma_k^{-2}}{n}},
\end{align*}
therefore
\begin{align*}
\frac{\p{1 + \frac{t}{\sqrt{n}}}^2}{c_t^2 \sigma_k^2 + n^{-1}} - \frac{1}{\sigma_k^2 + n^{-1}} & = \p{1 + \frac{2t}{\sqrt{n}} + \frac{t^2}{n}}\p{\frac{1}{\sigma_k^2 + n^{-1}} + \frac{2\sigma_k^2}{\sqrt{n}\p{\sigma_k^2 + n^{-1}}^2}t + O\p{\frac{\sigma_k^{-2}}{n}}} - \frac{1}{\sigma_k^2 + n^{-1}} \\
& = \p{\frac{2}{\sqrt{n}} \frac{1}{\sigma_k^2 + n^{-1}} + \frac{2\sigma_k^2}{\sqrt{n}\p{\sigma_k^2 + n^{-1}}^2}} t + O\p{\frac{\sigma_k^{-2}}{n}} \\
& = \frac{2}{\sqrt{n}} \frac{2\sigma_k^2 + n^{-1}}{\p{\sigma_k^2 + n^{-1}}^2} t + O\p{\frac{\sigma_k^{-2}}{n}}.
\end{align*}
Let now $a_k,b_k \in \RR$ be defined as
\[
a_k := \frac{2}{\sqrt{n}\p{n\sigma_k^2 + 1}} \asymp \frac{\sigma_k^{-2}}{n^{3/2}} \text{ and } b_k := \frac{2}{\sqrt{n}} \frac{2\sigma_k^2 + n^{-1}}{\p{\sigma_k^2 + n^{-1}}^2} \asymp \frac{\sigma_k^{-2}}{n^{1/2}},
\]
for which given the calculations above we have
\[
\frac{n\sigma_k^2 + c_t^{-2}}{n\sigma_k^2 + 1} = 1 + a_k t + O\p{\frac{\sigma_k^{-2}}{n^2}} \text{ and } \frac{\p{1 + \frac{t}{\sqrt{n}}}^2}{c_t^2 \sigma_k^2 + n^{-1}} - \frac{1}{\sigma_k^2 + n^{-1}} = b_k t + O\p{\frac{\sigma_k^{-2}}{n}}.
\]
Then
\begin{align*}
\ln B_k\p{\epsilon_k} & = - \frac{1}{2} \ln \p{1 + a_k t + O\p{\frac{\sigma_k^{-2}}{n^2}}} - \frac{1}{2} \p{f_{0k} + n^{-1/2} \epsilon_k}^2 \b{b_k t + O\p{\frac{\sigma_k^{-2}}{n}}} \\
& = -\frac{1}{2} a_k t + O\p{\frac{\sigma_k^{-2}}{n^2}} - \frac{1}{2} f_{0k}^2 \b{b_k t + O\p{\frac{\sigma_k^{-2}}{n}}} - \frac{1}{\sqrt{n}}  f_{0k} \epsilon_k \b{b_k t + O\p{\frac{\sigma_k^{-2}}{n}}} - \frac{1}{2n} \epsilon_k^2 \b{b_k t + O\p{\frac{\sigma_k^{-2}}{n}}} \\
& = - \frac{1}{2} \b{a_k + f_{0k}^2 b_k + \frac{2}{\sqrt{n}} f_{0k} \epsilon_k b_k + \frac{1}{n} \epsilon_k^2 b_k} t + O\p{\frac{\sigma_k^{-2}}{n^2} + \frac{f_{0k}^2 \sigma_k^{-2}}{n}} + O\p{\frac{\sigma_k^{-2} f_{0k}}{n^{3/2}}} \epsilon_k + O\p{\frac{\sigma_k^{-2}}{n^2}} \epsilon_k^2.
\end{align*}
Because $\sigma_k \geq n^{-1/2} \vee \abs{f_{0k}}$ and $p \asymp n/\ln n$, we have
\[
\sum_{k = 1}^{K_n} \frac{\sigma_k^{-2}}{n^2} + \frac{f_{0k}^2 \sigma_k^{-2}}{n} = O\p{\frac{1}{\ln n}} = o\p{1},
\]
and also
\[
\sum_{k = 1}^{K_n} \frac{\sigma_k^{-2} f_{0k}}{n^{3/2}} \epsilon_k = O_p\p{\p{\sum_{k = 1}^{K_n} \frac{\sigma_k^{-4} f_{0k}^2}{n^3}}^{1/2}} = O_p\p{\p{n^{-3}\sum_{k = 1}^{K_n} \sigma_k^{-2}}^{1/2}} = O_p\p{\frac{1}{\p{n \ln n}^{1/2}}} = o_p\p{1},
\]
as well as
\[
\sum_{k = 1}^{K_n} \frac{\sigma_k^{-2}}{n^2}\epsilon_k^2 = O_p\p{\sum_{k = 1}^{K_n} \frac{\sigma_k^{-2}}{n^2}} = O_p\p{\frac{1}{\ln n}} = o_p\p{1}.
\]
As a result under the only assumptions $\sigma_k \geq n^{-1/2} \vee \abs{f_{0k}}, \quad p \asymp n/\ln n$ we obtain
\[
\prod_{k = 1}^{K_n} B_k\p{\epsilon_k} = \exp\p{\mu_n t + o_p(1)}, \quad \mu_n := - \frac{1}{2} \sum_{k = 1}^{K_n} \b{a_k + f_{0k}^2 b_k + \frac{2}{\sqrt{n}} f_{0k} \epsilon_k b_k + \frac{1}{n} \epsilon_k^2 b_k},
\]
for any $t$, which then implies using Theorem $(2.1)$ \cite{castilloBernsteinVonMisesTheorem2015} that
\[
d_{BL}\p{\Pi_1 \circ Z_n^{-1}, \NN\p{0,4\norm{f_0}_2^2}} = o_p(1), \quad Z_n := \sqrt{n}\p{\norm{f}_2^2 - \frac{2p}{n} - \frac{\mu_n}{\sqrt{n}} - \bar \psi},
\]
where we recall that $\Pi_1$ denotes the posterior distribution over $f$. Finally, using
\begin{align*}
\sum_{k = 1}^{K_n} \frac{1}{n}\epsilon_k^2 b_k & = \sum_{k = 1}^{K_n} \frac{1}{n} b_k + \sum_{k = 1}^{K_n} \frac{1}{n}\b{\epsilon_k^2 - 1} b_k = \sum_{k = 1}^{K_n} \frac{1}{n} b_k + O_p\p{\p{\sum_{k = 1}^{K_n} \frac{b_k^2}{n^2} }^{1/2}} \\
& = \sum_{k = 1}^{K_n} \frac{1}{n} b_k + O_p\p{\frac{1}{\sqrt{\ln n}}} = \sum_{k = 1}^{K_n} \frac{1}{n} b_k + o_p(1)
\end{align*}
and, since $\sigma_k \geq n^{-1/2} \vee \abs{f_{0k}}$,
\[
\sum_{k = 1}^{K_n} \frac{f_{0k} b_k}{\sqrt{n}} \epsilon_k = O_p\p{\p{\sum_{k = 1}^{K_n} \frac{f_{0k}^2 b_k^2}{n}}^{1/2}} = O_p\p{\p{\sum_{k = 1}^{K_n} \frac{f_{0k}^2 \sigma_k^{-4}}{n^2}}^{1/2}} = O_p\p{\frac{1}{\sqrt{\ln n}}} = o_p(1),
\]
the same BvM result holds with $\mu_n$ replaced with
\[
\mu_n = - \frac{1}{2} \sum_{k = 1}^{K_n} \b{a_k + \frac{b_k}{n} + f_{0k}^2 b_k}.
\]

\subsection{Proofs of Section \ref{section:histograms}}\label{section:proofs_histograms}

\begin{proposition}\label{proposition:linear_functionals_histograms}
	For a linear functional $\Phi\p{\theta} = \inner{a|f_\theta}_2, \quad a \in \FF_p$, we have
	\[
	\nabla \Phi = a - \int a f_\theta, \quad \nabla_{\mu \nu} \Phi = \frac{1}{2} T_{\mu \nu \rho} \nabla^\rho \Phi, \quad \abs{\nabla \Phi}^2 = \int \p{a - \int a f_\theta}^2 f_\theta,
	\]
	\[
	\Phi + \inner{\nabla \Phi|\nabla \ell} = a.
	\]
\end{proposition}
\begin{proof}{Proposition \ref{proposition:linear_functionals_histograms}}\\
	Since $\Phi = \int a + \int \p{a - \int a} f_\theta$ we have
	\[
	\partial_\mu \Phi = \partial_\mu \int \p{a - \int a} f_\theta = \int \p{a - \int a} \p{\partial_\mu \ell} f_\theta = \int \p{a - \int a} \p{u_\mu - \eta_\mu} f_\theta.
	\]
	We now define the coefficients $a_p^\nu = \inner{u_\nu | a}_2$ so that $a - \int a = a_p^\mu u_\mu$ (because $a \in \FF_p$) to get
	\[
	\partial_\mu \Phi = \int \p{a_p^\nu u_\nu} \p{u_\mu - \eta_\mu} f_\theta = a_p^\nu g_{\mu \nu},
	\]
	which implies that $\nabla^\mu \Phi = g^{\mu \nu} \partial_\nu \Phi = a_p^\mu = \inner{u_\mu | a}_2$. This implies that
	\[
	\nabla \Phi = \p{\nabla^\mu \Phi} \partial_\mu \ell = \sum_{\mu = 1}^{p-1} \inner{u_\mu|a}_2 \p{u_\mu - \eta_\mu} = a - \int a f_\theta,
	\]
	and in particular
	\[
	\abs{\nabla \Phi}^2 = P_\theta \c{\p{\nabla \Phi}^2} = \int \p{a - \int a f_\theta}^2 f_\theta.
	\]
	For the Hessian of $\Phi$ we have
	\[
	\partial_{\mu \nu} \Phi = \partial_\nu a_p^\rho g_{\mu \rho} = a_p^\rho T_{\mu \nu \rho},
	\]
	which implies that
	\[
	\nabla_{\mu \nu} \Phi = \partial_{\mu \nu} \Phi - \Gamma_{~\mu \nu}^\rho \partial_\rho \Phi = \partial_{\mu \nu} \Phi - \frac{1}{2} T_{\mu \nu \rho} \nabla^\rho \Phi = \frac{1}{2} T_{\mu \nu \rho} a_p^\rho.
	\]
	Finally using the results above we find $\Phi + \inner{\nabla \Phi | \nabla \ell} = \int a f_\theta + a - \int a f_\theta = a$.
\end{proof}

\begin{proposition}\label{proposition:linear_functionals_histograms_2}
	Let $a \in \FF_p$ and $\Phi$ be the functional $\Phi\p{\theta} = \int a f_\theta$. Then:
	\begin{enumerate}
		\item For a prior of type \ref{prior:independent_haar_series_histograms} on $\theta$, if $H$ is differentiable we have
		\[
		\Delta_\pi \Phi = - \sum_{\mu = 1}^{p-1} \sigma_\mu^{-1} H'\p{\theta^\mu/\sigma_\mu} \inner{u_\mu | a}_2,
		\]
		\item For a Dirichlet prior \ref{prior:dirichlet} with parameter $m \in \p{0,\infty}^p$ on $\omega$ we have
		\[
		\Delta_\pi \Phi = \norm{m}_1 \inner{a| h_{\bar m} - h_\omega}_2, \quad \nabla \Delta_\pi \Phi = - \norm{m}_1 \p{a - \int ah_\omega}.
		\]
	\end{enumerate}
\end{proposition}
\begin{proof}{Proposition \ref{proposition:linear_functionals_histograms_2}}\\
\begin{enumerate}
	\item We have
	\begin{align*}
		\Delta_\pi \Phi = \Delta \Phi + \inner{\nabla \Phi | \nabla \ln \pi} & = \Delta \Phi + \inner{\nabla \ln \pi | \nabla \Phi} = \nabla_\mu \nabla^\mu \Phi + \nabla_\mu \ln \pi \nabla \Phi \\
		& = \partial_\mu \nabla^\mu \Phi + \Gamma_{~\mu \nu}^\nu \nabla^\mu \Phi + \b{\partial_\mu \ln \Pi - \Gamma_{~\mu \nu}^\nu} \nabla^\mu \Phi \\
		& = \partial_\mu \nabla^\mu \Phi + \p{\partial_\mu \ln \Pi} \nabla^\mu \Phi,
	\end{align*}
	but since $\nabla^\mu \Phi = \inner{u_\mu|a}_2$ we have $\partial_\mu \nabla^\mu \Phi = 0$, therefore
	\[
	\Delta_\pi \Phi = \p{\partial_\mu \ln \Pi} \nabla^\mu \Phi = - \sum_{\mu = 1}^{p-1} \sigma_\mu^{-1} H'\p{\theta^\mu/\sigma_\mu} \inner{u_\mu | a}_2, \quad \partial_\mu \Delta_\pi \Phi_a = - \sigma_\mu^{-2} H''\p{\theta^\mu/\sigma_\mu} \inner{u_\mu | a}_2.
	\]
	\item By definition we have
	\[
	\Delta_\pi \Phi = \Delta \Phi + \inner{\nabla \ln \pi|\nabla \Phi} = g^{\mu \nu} \nabla_{\mu \nu} \Phi + \p{\nabla_\mu \ln \pi} \p{\nabla^\mu \Phi},
	\]
	but using the $\alpha-\omega$ coordinate system recall that $\nabla^\mu \Phi = p\p{\inner{\ind{I_\mu}|a}_2 - \inner{\ind{I_p}|a}_2}, \nabla_{\mu \nu} \Phi = \frac{1}{2} T_{\mu \nu \rho} \inner{u_\rho | a}_2$ and
	\[
	\nabla_\mu \ln \pi = \partial_\mu \ln \Pi - \Gamma_{~\mu \nu}^\nu = \partial_\mu \ln \Pi - \frac{1}{2} g^{\nu \rho} T_{\mu \nu \rho},
	\]
	therefore
	\[
	\Delta_\pi \Phi = \frac{1}{2} g^{\mu \nu} T_{\mu \nu \rho} \inner{u_\rho | a}_2 + \p{\partial_\mu \ln \Pi - \frac{1}{2} g^{\nu \rho} T_{\mu \nu \rho}} \inner{u_\mu | a}_2 = p\p{\partial_\mu \ln \Pi} \p{\inner{\ind{I_\mu}|a}_2 - \inner{\ind{I_p}|a}_2}.
	\]
	But by Proposition \ref{proposition:density_dirichlet_alpha} we have
	\begin{align*}
		\partial_\mu \ln \Pi & = \sum_{\nu = 1}^p m_\nu \partial_\mu \ln \frac{e^{\alpha^\nu}}{\sum_{\rho = 1}^p e^{\alpha^\rho}} = \sum_{\nu = 1}^p m_\nu \partial_\mu \b{\alpha^\nu - \ln \p{\sum_{\rho = 1}^p e^{\alpha^\rho}}} = \sum_{\nu = 1}^p m_\nu \b{\delta_\mu^\nu - \frac{e^{\alpha^\mu}}{\sum_{\rho = 1}^p e^{\alpha^\rho}}} \\
		& = m_\mu - \sum_{\nu = 1}^p \frac{m_\nu e^{\alpha^\mu}}{\sum_{\rho = 1}^p e^{\alpha^\rho}} = m_\mu - \norm{m}_1 \omega_\mu,
	\end{align*}
	therefore
	\begin{align*}
		\Delta_\pi \Phi & = p\sum_{\mu = 1}^{p-1} \p{\inner{\ind{I_\mu}|a}_2 - \inner{\ind{I_p}|a}_2} \p{m_\mu - \norm{m}_1 \omega_\mu} = p \sum_{\mu = 1}^{p-1} \p{\int_{I_\mu} a - \int_{I_p}a} \p{m_\mu - \norm{m}_1 \omega_\mu} \\
		& = p \sum_{\mu = 1}^{p-1} \p{\int_{I_\mu} a} \p{m_\mu - \norm{m}_1 \omega_\mu} - p \p{\int_{I_p}a} \sum_{\mu = 1}^{p-1} \p{m_\mu - \norm{m}_1 \omega_\mu} \\
		& = p \sum_{\mu = 1}^{p-1} \p{\int_{I_\mu} a} \p{m_\mu - \norm{m}_1 \omega_\mu} - p \p{\int_{I_p}a} \p{\norm{m}_1 - m_p - \norm{m}_1\p{1-\omega_p}} \\
		& = p \sum_{\mu = 1}^{p-1} \p{\int_{I_\mu} a} \p{m_\mu - \norm{m}_1 \omega_\mu} + p \p{\int_{I_p}a} \p{m_p - \norm{m}_1\omega_p} \\
		& = p \sum_{\mu = 1}^p \p{\int_{I_\mu} a} \p{m_\mu - \norm{m}_1 \omega_\mu} = \norm{m}_1 \sum_{\mu = 1}^p \frac{\inner{\ind{I_\mu}|a}_2}{\norm{\ind{I_\mu}}_2^2} \inner{\ind{I_\mu}| h_{\bar m} - h_\omega}_2 \\
		& = \norm{m}_1 \inner{a| h_{\bar m} - h_\omega}_2.
	\end{align*}
	As a result we also have (still in the $\alpha-\omega$ coordinate system)
	\[
	\nabla_\mu \Delta_\pi \Phi = \frac{\partial}{\partial \alpha^\mu} \b{\norm{m}_1 \inner{a| h_{\bar m} - h_\omega}_2} = - \norm{m}_1 \int a \p{\ind{I_\mu} - \omega_\mu} h_\omega,
	\]
	which implies that
	\[
	\nabla \Delta_\pi \Phi = - \norm{m}_1 g^{\mu \nu} \p{\int a\p{\ind{I_\mu} - \omega_\mu} h_\omega} \p{\ind{I_\nu} - \omega_\nu},
	\]
	but because $g_{\mu \nu}$ is the Gram matrix of the scores $\partial_\mu \ell = \ind{I_\mu} - \omega_\mu, \quad 1 \leq \mu \leq p-1$, the quantity
	\[
	g^{\mu \nu} \p{\int a\p{\ind{I_\mu} - \omega_\mu} h_\omega} \p{\ind{I_\nu} - \omega_\nu}
	\]
	is actually the $L^2\p{P_\omega}$ projection of $a$ on the tangent space
	\[
	T_\omega = \Span \b{\ind{I_\mu} - \omega_\mu : 1 \leq \mu \leq p-1} = \b{a \in \FF_p : \int a h_\omega = 0},
	\]
	which is equal to $a - \int ah_\omega$, and as a result
	\[
	\nabla \Delta_\pi \Phi = - \norm{m}_1 \p{a - \int ah_\omega}.
	\]
\end{enumerate}
\end{proof}

\begin{proposition}\label{proposition:bernstein_empirical_bins}
	If $p \ll \frac{n}{\ln n}$, for any $K>0$ there exists $M,n_->0$ such that
	\[
	\Pro_0^n \c{\norm{\hat f - \pi_p\c{f_0}}_\infty > M \sqrt{\frac{p \ln n}{n}}} \leq n^{-K}, \quad n \geq n_-.
	\]
\end{proposition}
\begin{proof}
	We have
	\[
	\hat f - \pi_p\c{f_0} = \sum_{j = 0}^J \sum_{k = 1}^{2^j} \p{ \hat f_{jk} - f_{0jk} } \psi_{jk},
	\]
	which implies
	\begin{align*}
		\norm{\hat f - \pi_p\c{f_0}}_\infty & \leq \sum_{j = 0}^J \norm{\sum_{k = 1}^{2^j} \p{\hat f_{jk} - f_{0jk}} \psi_{jk}}_\infty = \sum_{j=0}^J 2^{j/2} \max_{1 \leq k \leq 2^j} \abs{\hat f_{jk} - f_{0jk}} \\
		& \leq \p{2^{(J+1)/2} - 1} \max_{0 \leq j \leq J, 1 \leq k \leq 2^j} \abs{\hat f_{jk} - f_{0jk}} \leq \sqrt{2p} \max_{0 \leq j \leq J, 1 \leq k \leq 2^j} \abs{\hat f_{jk} - f_{0jk}},
	\end{align*}
	and as a result, by a union bound,
	\[
	\Pro_0^n \c{\norm{\hat f - \pi_p\c{f_0}}_\infty > t} \leq \sum_{j = 0}^J \sum_{k = 1}^{2^j} \Pro_0^n \c{\abs{\hat f_{jk} - f_{0jk}} > t/\sqrt{2p}}.
	\]
	Since
	\[
	\norm{\psi_{jk}}_\infty = 2^{j/2} \leq \sqrt{p}, \quad \Pro_0 \c{\p{\psi_{jk} - \Pro_0 \c{\psi_{jk}}}^2} \leq \Pro_0 \c{\psi_{jk}^2} \leq \norm{f_0}_\infty \int \psi_{jk}^2 = \norm{f_0}_\infty,
	\]
	by Bernstein's inequality we obtain, for any $j,k$,
	\[
	\Pro_0^n \c{\abs{\hat f_{jk} - f_{0jk}} > t/\sqrt{2p}} \leq 2 \exp \p{- \frac{1}{2} \frac{n\p{t/\sqrt{2p}}^2}{\norm{f_0}_\infty + \p{t/\sqrt{2p}} \times \sqrt{p}}} = 2\exp \p{-\frac{1}{2} \frac{nt^2/2p}{\norm{f_0}_\infty + t/\sqrt{2}}},
	\]
	and therefore
	\[
	\Pro_0^n \c{\norm{\hat f - \pi_p\c{f_0}}_\infty > t} \leq 8p\exp \p{-\frac{1}{2} \frac{nt^2/2p}{\norm{f_0}_\infty + t/\sqrt{2}}}.
	\]
	Now since $p \ll \frac{n}{\ln n}$, for any $K>0$ we can find $M>0$ such that setting $t = M\sqrt{\frac{p \ln n}{n}}$ makes the probability in the last display less than $n^{-K}$.
\end{proof}

\begin{lemma}\label{lemma:supremum_norm_contraction_rates_histograms}
	Assume that $X^n = \p{X_i}_{i=1}^n \iid f_0$ on $\c{0,1}$ for some $f_0 \in \CC^0, \quad p \ll \frac{n}{\ln n}$ and that $\Pi$ is either
	\begin{enumerate}
		\item a prior of type \ref{prior:independent_haar_series_histograms} with $H$ differentiable, $\norm{H'}_\infty < +\infty$ and $\sum_{j = 0}^J \sigma_j^{-1} 2^{j/2} \leq c\sqrt{np\ln n}$ for some constant $c>0$,
		\item or a Dirichlet prior \ref{prior:dirichlet} with parameter $m \in \p{0,\infty}^p$ such that $\norm{m}_\infty \leq c\sqrt{\frac{n \ln n}{p}}$ for a fixed constant $c>0$.
	\end{enumerate}
	Then for any $K>0$ there exists $M>0$ such that
	\[
	\ind{E_n} \Pi_n \c{\norm{f_\theta - \hat f}_\infty > M \sqrt{\frac{p \ln n}{n}}} \leq n^{-K}, \quad \Pro_0^n\c{E_n^c} \leq n^{-K},
	\]
	where $E_n$ is the event
	\[
	E_n = \b{\norm{\hat f - \pi_p\c{f_0}}_\infty \leq M \sqrt{\frac{p \ln n}{n}}}.
	\]
	In particular if $f_0 \in \CC^s, \quad 0 < s \leq 1$ then for $M>0$ large enough we have
	\[
	\ind{E_n} \Pi_n \c{\norm{f_\theta - f_0}_\infty > M\varepsilon_n } \leq n^{-K}, \quad \varepsilon_n := p^{-s} + \sqrt{\frac{p \ln n}{n}}.
	\]
\end{lemma}
\begin{proof}{Lemma \ref{lemma:supremum_norm_contraction_rates_histograms}}\\
	We show in Section \ref{section:justifications_lemma_laplace_transform} that we can apply Lemma \ref{lemma:laplace_transform} to the functional $\Phi_a = P_\theta \c{a}$ where $a = \ind{I_{\mu_0}} \in \FF_p$ for some $\mu_0 \in \b{1,\ldots,p}$ (in the prior \ref{prior:independent_haar_series_histograms} case) or $\tilde \Phi_a = \Phi_a - \frac{1}{n} \Delta_\pi \Phi_a$ (in the prior \ref{prior:dirichlet} case), the estimator $\hat \Psi_0 = \Pro_n\c{a} = \Pro_n \c{I_{\mu_0}}$ and the vector field $V = \nabla \Phi_a$. We will now see how this implies a differential inequality satisfied by the posterior cumulant generating function of $\sqrt{n}\p{\Phi_a - \Pro_n \c{a}}$ or $\sqrt{n}\p{\tilde \Phi_a - \Pro_n \c{a}}$ which somehow allows us to conclude.
	\begin{enumerate}
		\item Let $\mu_0 \in \b{1,\ldots,p}, \quad a = \ind{I_{\mu_0}} \in \FF_p, \quad \Phi_a$ the linear functional $\Phi_a = \int a f_\theta = P_\theta \c{I_{\mu_0}}$ and $V = \nabla \Phi_a$. Then as shown in Section \ref{section:justifications_lemma_laplace_transform} Lemma \ref{lemma:laplace_transform} applies here and therefore, for any $t \in \RR$,
		\begin{align*}
		\frac{d}{dt} \ln \Pi_n \c{e^{t\sqrt{n}\p{\Phi_a - \hat \Psi_0}}} & = \Pi_n \c{\p{t\abs{\nabla \Phi_a}^2 + \sqrt{n}\p{\hat \Phi_a - \hat \Psi_0} + \frac{1}{\sqrt{n}} \Delta_\pi \Phi_a}  \frac{e^{t\sqrt{n}\p{\Phi - \hat \Psi_0}}}{\Pi_n \c{e^{t\sqrt{n}\p{\Phi - \hat \Psi_0}}}}} \\
		& = \Pi_n \c{\p{t\Phi_a\p{1-\Phi_a} - \frac{1}{\sqrt{n}} \sum_{\mu = 1}^{p-1} \sigma_\mu^{-1} H'\p{\theta^\mu/\sigma_\mu} \inner{u_\mu | a}_2}  \frac{e^{t\sqrt{n}\p{\Phi - \hat \Psi_0}}}{\Pi_n \c{e^{t\sqrt{n}\p{\Phi - \hat \Psi_0}}}}}.
		\end{align*}
		Because for any $0 \leq j \leq J$ and $1 \leq k \leq 2^j$ we have
		\[
		\abs{\inner{\psi_{jk} | a}_2} = 2^{j/2}/p \text{ and } \# \b{1 \leq k \leq 2^j : \Supp\p{\psi_{jk}} \cap \Supp\p{a}} \leq 2,
		\]
		which implies
		\[
		\abs{\frac{1}{\sqrt{n}} \sum_{\mu = 1}^{p-1} \sigma_\mu^{-1} H'\p{\theta^\mu/\sigma_\mu} \inner{u_\mu | a}_2} \leq 2\norm{H'}_\infty \frac{1}{\sqrt{n}p} \sum_{j = 0}^J \sigma_j^{-1} 2^{j/2}.
		\]
		As a result
		\begin{align*}
			\frac{d}{dt} \ln \Pi_n \c{e^{t\sqrt{n}\p{\Phi_a - \hat \Psi_0}}} & \leq 2\norm{H'}_\infty \frac{1}{\sqrt{n}p} \sum_{j = 0}^J \sigma_j^{-1} 2^{j/2} + \Pi_n \c{t \Phi_a \frac{e^{t\sqrt{n}\p{\Phi - \hat \Psi_0}}}{\Pi_n \c{e^{t\sqrt{n}\p{\Phi - \hat \Psi_0}}}}} \\
			& = 2\norm{H'}_\infty \frac{1}{\sqrt{n}p} \sum_{j = 0}^J \sigma_j^{-1} 2^{j/2} + t\hat \Psi_0 + \frac{t}{\sqrt{n}} \Pi_n \c{\sqrt{n}\p{\Phi_a - \hat \Psi_0} \frac{e^{t\sqrt{n}\p{\Phi - \hat \Psi_0}}}{\Pi_n \c{e^{t\sqrt{n}\p{\Phi - \hat \Psi_0}}}}}.
		\end{align*}
		Since by Lemma \ref{lemma:laplace_transform} we also have
		\[
		\frac{d}{dt} \ln \Pi_n \c{e^{t\sqrt{n}\p{\Phi_a - \hat \Psi_0}}} = \Pi_n \c{\sqrt{n}\p{\Phi_a - \hat \Psi_0} \frac{e^{t\sqrt{n}\p{\Phi - \hat \Psi_0}}}{\Pi_n \c{e^{t\sqrt{n}\p{\Phi - \hat \Psi_0}}}}},
		\]
		we obtain
		\[
		\frac{d}{dt} \ln \Pi_n \c{e^{t\sqrt{n}\p{\Phi_a - \hat \Psi_0}}} \leq 2\norm{H'}_\infty \frac{1}{\sqrt{n}p} \sum_{j = 0}^J \sigma_j^{-1} 2^{j/2} + t\hat \Psi_0 + \frac{t}{\sqrt{n}} \frac{d}{dt} \ln \Pi_n \c{e^{t\sqrt{n}\p{\Phi - \hat \Psi_0}}},
		\]
		and therefore
		\[
		\frac{d}{dt} \ln \Pi_n \c{e^{t\sqrt{n}\p{\Phi_a - \hat \Psi_0}}} \leq 2\p{2\norm{H'}_\infty \frac{1}{\sqrt{n}p} \sum_{j = 0}^J \sigma_j^{-1} 2^{j/2}+ t\hat \Psi_0}, \quad 0 \leq t \leq \frac{\sqrt{n}}{2}.
		\]
		Integrating this inequality between $0$ and $t$ then yields
		\begin{align*}
			\ln \Pi_n \c{e^{t\sqrt{n}\p{\Phi_a - \hat \Psi_0}}} & \leq 2t\p{2\norm{H'}_\infty \frac{1}{\sqrt{n}p} \sum_{j = 0}^J \sigma_j^{-1} 2^{j/2} + \frac{1}{2}t\hat \Psi_0} \\
			& =: at + bt^2, \quad 0 \leq t \leq \frac{\sqrt{n}}{2}.
		\end{align*}
		Applying the same reasoning to $-\Phi_a$ we find
		\begin{align*}
			\frac{d}{dt} \ln \Pi_n \c{e^{-t\sqrt{n}\p{\Phi_a - \hat \Psi_0}}} & \leq \frac{a}{2} + \Pi_n \c{t\p{-\Phi_a} \frac{e^{-t\sqrt{n}\p{\Phi_a - \hat \Psi_0}}}{\Pi_n \c{e^{-t\sqrt{n}\p{\Phi_a - \hat \Psi_0}}}}} \\
			& = \frac{a}{2} - t\hat \Psi_0 + \frac{t}{\sqrt{n}} \Pi_n \c{\sqrt{n}\p{\hat \Psi_0 -\Phi_a} \frac{e^{-t\sqrt{n}\p{\Phi_a - \hat \Psi_0}}}{\Pi_n \c{e^{-t\sqrt{n}\p{\Phi_a - \hat \Psi_0}}}}} \\
			& \leq \frac{a}{2} + \frac{t}{\sqrt{n}} \frac{d}{dt} \ln \Pi_n \c{e^{-t\sqrt{n}\p{\Phi_a - \hat \Psi_0}}},
		\end{align*}
		which for $t \leq \frac{\sqrt{n}}{2}$ implies
		\[
		\frac{d}{dt} \ln \Pi_n \c{e^{-t\sqrt{n}\p{\Phi_a - \hat \Psi_0}}} \leq a,
		\]
		and therefore
		\[
		\ln \Pi_n \c{e^{-t\sqrt{n}\p{\Phi_a - \hat \Psi_0}}} \leq at, \quad 0 \leq t \leq \frac{\sqrt{n}}{2}.
		\]
		Using the numerical inequality $e^{\abs{x}} \leq e^x + e^{-x}$ we obtain
		\[
		\Pi_n \c{e^{t\sqrt{n}\abs{\Phi _a- \hat \Psi_0}}} \leq 2\exp\p{at + bt^2}, \quad 0 \leq t \leq \frac{\sqrt{n}}{2},
		\]
		which by Markov' inequality implies, for any $x>0$ and $0 \leq t \leq \frac{\sqrt{n}}{2}$,
		\[
		\Pi_n \c{\abs{\Phi_a - \hat \Psi_0} > x} \leq \exp\p{-tx\sqrt{n}} \Pi_n \c{e^{t\sqrt{n}\abs{\Phi_a - \hat \Psi_0}}} \leq 2\exp\p{at +bt^2 - tx\sqrt{n}}, \quad 0 \leq t \leq \frac{\sqrt{n}}{2}.
		\]
		For $x = C \sqrt{\frac{\ln n}{np}}$ and
		\[
		x \sqrt{n} \geq 2a \iff \sum_{j = 0}^J \sigma_j^{-1} 2^{j/2} \leq \frac{C}{8\norm{H'}_\infty} \sqrt{np\ln n}
		\]
		(which is satisfied for $C$ large enough by assumption) we get
		\[
		\Pi_n \c{\abs{\Phi_a - \hat \Psi_0} > x} \leq 2\exp\p{bt^2 - tx\sqrt{n}/2}, \quad 0 \leq t \leq \frac{\sqrt{n}}{2},
		\]
		and with $E_n$ the event
		\[
		E_n = \b{\norm{\hat f - \pi_p\c{f_0}}_\infty \leq M \sqrt{\frac{p \ln n}{n}}}
		\]
		(which by Proposition \ref{proposition:bernstein_empirical_bins} satisfies $\Pro_0\c{E_n^c} \leq n^{-K}$ for some $M>0$), since $p \ll n/\ln n$ we have for large $n$
		\[
		\hat \Psi_0 \ind{E_n} = \ind{E_n} \int \hat f \ind{I_{\mu_0}} \leq \frac{1}{p} \ind{E_n} \norm{\hat f}_\infty \leq \frac{1}{p}\p{M \sqrt{\frac{p \ln n}{n}} + \norm{\pi_p\c{f_0}}_\infty} \leq \frac{2\norm{f_0}_\infty}{p},
		\]
		so that
		\[
		\ind{E_n} \Pi_n \c{\abs{\Phi_a - \hat \Psi_0} > x} \leq 2\exp\p{\frac{2t^2 \norm{f_0}_\infty}{p} - tx\sqrt{n}/2}, \quad 0 \leq t \leq \frac{\sqrt{n}}{2},
		\]
		and choosing $t = \frac{xp\sqrt{n}}{8\norm{f_0}_\infty} = \frac{C}{8\norm{f_0}_\infty} \sqrt{p\ln n} \ll \sqrt{n}$ yields
		\[
		\ind{E_n} \Pi_n \c{\abs{\Phi_a - \hat \Psi_0} > x} \leq 2\exp\p{-\frac{npx^2}{32\norm{f_0}_\infty}} = 2\exp\p{- \frac{C^2}{32\norm{f_0}_\infty} \ln n},
		\]
		By a union bound over $\mu_0$ we obtain
		\begin{align*}
		\ind{E_n} \Pi_n \c{\norm{f_\theta - \hat f}_\infty > Cp\sqrt{\frac{\ln n}{np}}} & = \ind{E_n} \Pi_n \c{\max_{1 \leq \mu_0 \leq p} \abs{P_\theta \c{I_{\mu_0}} - \Pro_n\c{I_{\mu_0}}} > C\sqrt{\frac{\ln n}{np}}} \\
		& \leq 2p\exp\p{- \frac{C^2}{32\norm{f_0}_\infty} \ln n},
		\end{align*}
		therefore for some $C>0$ we get
		\[
		\ind{E_n} \Pi_n \c{\norm{f_\theta - \hat f} > C\sqrt{\frac{p\ln n}{n}}} \leq n^{-K}.
		\]
		\item For Dirichlet priors \ref{prior:dirichlet} we show in Section \ref{section:justifications_lemma_laplace_transform} that Lemma \ref{lemma:laplace_transform} applies to the functional $\tilde \Phi_a = \Phi_a - \frac{1}{n} \Delta_\pi \Phi_a$ and the vector field $V = \nabla \Phi_a$ and therefore, for any $t \in \RR$,
		\begin{align*}
		\frac{d}{dt} \ln \Pi_n \c{\exp\p{t\sqrt{n}\p{\tilde \Phi_a - \hat \Psi_0}}} & = \Pi_n \c{\p{t\inner{\nabla \tilde \Phi_a | V} + \sqrt{n} \p{\hat{\tilde \Phi}_a - \hat \Psi_0 + \frac{1}{n} \Div_\pi V}} \frac{e^{t\sqrt{n}\p{\tilde \Phi_a - \hat \Psi_0}}}{\Pi_n\c{e^{t\sqrt{n}\p{\tilde \Phi_a - \hat \Psi_0}}}}} \\
		& = \Pi_n \c{\p{t\inner{\nabla \tilde \Phi_a | \nabla \Phi_a} + \sqrt{n}\p{\hat{ \tilde \Phi}_a - \hat \Psi_0 + \frac{1}{n} \Delta_\pi \Phi_a}} \frac{e^{t\sqrt{n}\p{\tilde \Phi_a - \hat \Psi_0}}}{\Pi_n \c{e^{t\sqrt{n}\p{\tilde \Phi_a - \hat \Psi_0}}}}}.
		\end{align*}
		First, since as for priors of type \ref{prior:independent_haar_series_histograms} we have $\hat \Phi_a = \hat \Psi_0$, we have
		\begin{align*}
		\hat{\tilde \Phi}_a - \hat \Psi_0 + \frac{1}{n} \Delta_\pi \Phi_a & = \tilde \Phi_a + \frac{1}{n} \inner{V | \nabla \ell_n} - \hat \Psi_0 + \frac{1}{n} \Delta_\pi \Phi_a \\
		& = \Phi_a - \frac{1}{n}\Delta_\pi \Phi_a  + \frac{1}{n} \inner{\nabla \Phi_a | \nabla \ell_n} - \hat \Psi_0 + \frac{1}{n} \Delta_\pi \Phi_a \\
		& = \hat \Phi_a - \hat \Psi_0 = 0,
		\end{align*}
		therefore, using also $\nabla \tilde \Phi_a = \p{1 + \frac{\norm{m}_1}{n}} \nabla \Phi_a$ and $\abs{\nabla \Phi_a}^2 = \Phi_a\p{1-\Phi_a} \leq \Phi_a$,
		\begin{align*}
		\frac{d}{dt} \ln \Pi_n \c{\exp\p{t\sqrt{n}\p{\tilde \Phi_a - \hat \Psi_0}}} & = \Pi_n \c{t\inner{\nabla \tilde \Phi_a | \nabla \Phi_a} \frac{e^{t\sqrt{n}\p{\tilde \Phi_a - \hat \Psi_0}}}{\Pi_n \c{e^{t\sqrt{n}\p{\tilde \Phi_a - \hat \Psi_0}}}}} \\
		& = \Pi_n \c{t\p{1 + \frac{\norm{m}_1}{n}} \Phi_a\p{1-\Phi_a} \frac{e^{t\sqrt{n}\p{\tilde \Phi_a - \hat \Psi_0}}}{\Pi_n \c{e^{t\sqrt{n}\p{\tilde \Phi_a - \hat \Psi_0}}}}} \\
		& \leq t\Pi_n \c{\p{1 + \frac{\norm{m}_1}{n}}\Phi_a \frac{e^{t\sqrt{n}\p{\tilde \Phi_a - \hat \Psi_0}}}{\Pi_n \c{e^{t\sqrt{n}\p{\tilde \Phi_a - \hat \Psi_0}}}}} \\
		& = t\Pi_n \c{\b{\tilde \Phi_a + \frac{\norm{m}_1}{n} \int a h_{\bar m}} \frac{e^{t\sqrt{n}\p{\tilde \Phi_a - \hat \Psi_0}}}{\Pi_n \c{e^{t\sqrt{n}\p{\tilde \Phi_a - \hat \Psi_0}}}}} \\
		& = t\frac{\norm{m}_1}{n} \int a h_{\bar m} + t\Pi_n \c{\tilde \Phi_a \frac{e^{t\sqrt{n}\p{\tilde \Phi_a - \hat \Psi_0}}}{\Pi_n \c{e^{t\sqrt{n}\p{\tilde \Phi_a - \hat \Psi_0}}}}} \\
		& = t\p{\hat \Psi_0 + \frac{\norm{m}_1}{n} \int a h_{\bar m}} + \frac{t}{\sqrt{n}} \Pi_n \c{\sqrt{n}\p{\tilde \Phi_a - \hat \Psi_0} \frac{e^{t\sqrt{n}\p{\tilde \Phi_a - \hat \Psi_0}}}{\Pi_n \c{e^{t\sqrt{n}\p{\tilde \Phi_a - \hat \Psi_0}}}}} \\
		& = t\p{\hat \Psi_0 + \frac{\norm{m}_1}{n} \int ah_{\bar m}} + \frac{t}{\sqrt{n}} \frac{d}{dt} \ln \Pi_n \c{\exp\p{t\sqrt{n}\p{\tilde \Phi_a - \hat \Psi_0}}},
		\end{align*}
		which for $0 \leq t \leq \frac{\sqrt{n}}{2}$ implies that
		\[
		\frac{d}{dt} \ln \Pi_n \c{\exp\p{t\sqrt{n}\p{\tilde \Phi_a - \hat \Psi_0}}} \leq 2t\p{\hat \Psi_0 + \frac{\norm{m}_1}{n} \int ah_{\bar m}}.
		\]
		Because $\int a h_{\bar m} \leq \norm{h_{\bar m}}_\infty \int a = \frac{1}{p} \times \frac{p\norm{m}_\infty}{\norm{m}_1} = \frac{\norm{m}_\infty}{\norm{m}_1} \leq \frac{c}{\norm{m}_1} \sqrt{\frac{n \ln n}{p}}$, we obtain
		\[
		\frac{d}{dt} \ln \Pi_n \c{\exp\p{t\sqrt{n}\p{\tilde \Phi_a - \hat \Psi_0}}} \leq 4t\p{\hat \Psi_0 + c\sqrt{\frac{\ln n}{np}}}, \quad 0 \leq t \leq \frac{\sqrt{n}}{2},
		\]
		and as in the proof for priors of type \ref{prior:independent_haar_series_histograms}, using $p \ll n/\ln n$ for large $n$ we have $\hat \Psi_0 \ind{E_n} \leq 2\norm{f_0}_\infty/p$ and therefore, on the event $E_n$ and for large $n$,
		\[
		\sup_{0 \leq t \leq \sqrt{n}/2}\frac{d}{dt} \ln \Pi_n \c{\exp\p{t\sqrt{n}\p{\tilde \Phi_a - \hat \Psi_0}}} \leq 4t \p{\frac{2\norm{f_0}_\infty}{p} + \frac{c}{p} \sqrt{\frac{p\ln n}{n}}} \leq \frac{16t \norm{f_0}_\infty}{p},
		\]
		which yields, on $E_n$,
		\[
		\ln \Pi_n \c{\exp\p{t\sqrt{n}\p{\tilde \Phi_a - \hat \Psi_0}}} \leq \frac{32\norm{f_0}_\infty t^2}{p}, \quad 0 \leq t \leq \frac{\sqrt{n}}{2}.
		\]
		Applying the same reasoning to $-\Phi_a$ we obtain similarly, on $E_n$,
		\[
		\ln \Pi_n \c{\exp\p{-t\sqrt{n}\p{\tilde \Phi_a - \hat \Psi_0}}} \leq \frac{32\norm{f_0}_\infty t^2}{p}, \quad 0 \leq t \leq \frac{\sqrt{n}}{2},
		\]
		therefore, using the numerical inequality $e^{\abs{x}} \leq e^x + e^{-x}$,
		\[
		\ind{E_n} \Pi_n \c{\exp\p{t\sqrt{n}\abs{\tilde \Phi_a - \hat \Psi_0}}} \leq 2\exp\p{\frac{32\norm{f_0}_\infty t^2}{p}}, \quad 0 \leq t \leq \frac{\sqrt{n}}{2}.
		\]
		As a result, by Markov' inequality, with $x = C\sqrt{\frac{\ln n}{np}}$ for a fixed $C>0$,
		\[
		\ind{E_n} \Pi_n \c{\abs{\tilde \Phi_a - \hat \Psi_0} > x} \leq 2\exp\p{-t\sqrt{n} x + \frac{32\norm{f_0}_\infty t^2}{p}}, \quad 0 \leq t \leq \frac{\sqrt{n}}{2},
		\]
		and choosing $t = \frac{p\sqrt{n}x}{64\norm{f_0}_\infty} \lesssim \sqrt{p\ln n} \ll \sqrt{n}$ (proving in particular that $t \in \c{0,\frac{\sqrt{n}}{2}}$) yields
		\[
		\ind{E_n} \Pi_n \c{\abs{\tilde \Phi_a - \hat \Psi_0} > x} \leq 2\exp\p{- \frac{npx^2}{128\norm{f_0}_\infty}} = 2\exp\p{- \frac{C^2\ln n}{128\norm{f_0}_\infty}}.
		\]
		Finally $\tilde \Phi_a = \p{1 + \frac{\norm{m}_1}{n}} \Phi_a - \frac{\norm{m}_1}{n} \int ah_{\bar m}$, therefore $\abs{\tilde \Phi_a - \hat \Psi_0} \leq x$ implies
		\begin{align*}
		\abs{\Phi_a - \hat \Psi_0} & = \abs{\frac{\tilde \Phi_a + \frac{\norm{m}_1}{n} \int ah_{\bar m}}{1 + \frac{\norm{m}_1}{n}} - \hat \Psi_0} = \abs{\frac{\tilde \Phi_a + \frac{\norm{m}_1}{n} \int ah_{\bar m}}{1 + \frac{\norm{m}_1}{n}} - \frac{\hat \Psi_0}{1 + \frac{\norm{m}_1}{n}} + \frac{\frac{\norm{m}_1}{n}\hat \Psi_0}{1+\frac{\norm{m}_1}{n}}} \\
		& \leq x + \frac{\norm{m}_1}{n} \p{\int ah_{\bar m}  + \abs{\hat \Psi_0}} \leq x + \frac{\norm{m}_1}{n} \p{\frac{1}{p} \times \frac{p\norm{m}_\infty}{\norm{m}_1}  + \abs{\hat \Psi_0}} \\
		& \leq x + c \sqrt{\frac{\ln n}{np}} + \frac{cp}{n} \abs{\hat \Psi_0},
		\end{align*}
		so that, on $E_n$ and for large $n$,
		\[
		\abs{\tilde \Phi_a - \hat \Psi_0} \leq x \implies \abs{\Phi_a - \hat \Psi_0} \leq x + c \sqrt{\frac{\ln n}{np}} + \frac{cp}{n} \abs{\hat \Psi_0} \leq \frac{C+c}{p}\sqrt{\frac{p\ln n}{n}} + \frac{2c\norm{f_0}_\infty}{n} \leq \frac{2\p{C+c}}{p}\sqrt{\frac{p\ln n}{n}}.
		\]
		As a result, by a union bound over $\mu_0 \in \b{1,\ldots,p}$,
		\[
		\ind{E_n} \Pi_n \c{\max_{1 \leq \mu_0 \leq p} \abs{\int_{I_{\mu_0}} h_\omega - \Pro_n \c{I_{\mu_0}}} > \frac{2\p{C+c}}{p} \sqrt{\frac{p \ln n}{n}}} \leq 2p \exp\p{-\frac{C^2 \ln n}{8\p{\norm{f_0}_\infty + c}}},
		\]
		which is less than $n^{-K}$ for large choice of $C$ since $p \ll n/\ln n$. We conclude as for the priors of type \ref{prior:independent_haar_series_histograms} by noticing that $\norm{h_\omega - \hat f}_\infty = p \max_{1 \leq \mu_0 \leq p} \abs{\int_{I_{\mu_0}} h_\omega - \Pro_n\c{I_{\mu_0}}}$.
	\end{enumerate}
\end{proof}

\begin{proposition}\label{proposition:metric_tensor_bounds_histograms}
	In coordinates $\theta$ the metric tensor $g\p{\theta} = \c{g_{\mu \nu}\p{\theta}}_{\mu,\nu = 1}^{p-1}$ satisfies
	\[
	\p{\inf_{\c{0,1}} f_\theta} I_{p-1} \preceq g\p{\theta} \preceq \p{\sup_{\c{0,1}} f_\theta} I_{p-1}.
	\]
\end{proposition}
\begin{proof}
	If $a \in \RR^{p-1}$ then
	\begin{align*}
		a^\mu a^\nu g_{\mu \nu}\p{\theta} & = a^\mu a^\nu \int \p{u_\mu - \int u_\mu f_\theta} \p{u_\nu - \int u_\nu f_\theta} f_\theta \\
		& = \int \p{a^\mu u_\mu - \int a^\mu u_\mu f_\theta}^2 f_\theta,
	\end{align*}
	therefore by $L^2$ orthonormality of the family $\p{\ind{},u_1,\ldots,u_{p-1}}$,
	\[
	a^\mu a^\nu g_{\mu \nu}\p{\theta} = \int \p{a^\mu u_\mu}^2 f_\theta - \p{\int a^\mu u_\mu f_\theta}^2 \leq \int \p{a^\mu u_\mu}^2 f_\theta \leq \p{\sup_{\c{0,1}} f_\theta} \int \p{a^\mu u_\mu}^2 = \p{\sup_{\c{0,1}} f_\theta} \norm{a}_{\ell^2}^2,
	\]
	proving that $g\p{\theta} \preceq \p{\sup_{\c{0,1}} f_\theta} I_{p-1}$. Similarly, still by $L^2$ orthonormality of $\p{\ind{},u_1,\ldots,u_{p-1}}$,
	\begin{align*}
	a^\mu a^\nu g_{\mu \nu}\p{\theta} & = \int \p{a^\mu u_\mu - \int a^\mu u_\mu f_\theta}^2 f_\theta \\
	& \geq \p{\inf_{\c{0,1}} f_\theta} \int \p{a^\mu u_\mu - \int a^\mu u_\mu f_\theta}^2 = \p{\inf_{\c{0,1}} f_\theta} \b{\p{\int a^\mu u_\mu - \int a^\mu u_\mu f_\theta}^2 + \int \p{a^\mu u_\mu - \int a^\mu u_\mu}^2} \\
	& \geq \p{\inf_{\c{0,1}} f_\theta} \int \p{a^\mu u_\mu - \int a^\mu u_\mu}^2 = \p{\inf_{\c{0,1}} f_\theta} \norm{a}_{\ell^2}^2,
	\end{align*}
	proving that $g\p{\theta} \succeq \p{\inf_{\c{0,1}} f_\theta} I_{p-1}$.
\end{proof}

\begin{proposition}\label{proposition:squared_norm_functional_histograms}
	Let $\Phi = \int f_\theta^2 - \frac{p}{n}$ and $V$ be the vector field
	\[
	V^\mu = \nabla^\mu \Phi + \frac{1}{2} \p{\nabla^{\mu \nu} \Phi - \frac{1}{2} T^{\mu \nu \rho} \partial_\rho \Phi} \Pro_n \c{\partial_\nu \ell}.
	\]
	Then
	\[
	\nabla \Phi = 2\p{f_\theta - \int f_\theta^2} = 2\p{h_\omega - \int h_\omega^2}, \quad V = \hat f + f_\theta - P_\theta \c{\hat f + f_\theta} = \hat f + h_\omega - P_\omega \c{\hat f + h_\omega},
	\]
	\[
	\hat \Phi = \Phi + \frac{1}{n} \inner{V | \nabla \ell_n} = \int \hat f^2 - \frac{p}{n},
	\]
	and :
	\begin{enumerate}
		\item in $\theta - \eta$ coordinates we have
		\[
		\nabla^\mu \Phi = 2\eta_\mu, \quad V^\mu = \partial^\mu \Phi + \frac{1}{2} \p{\partial^{\mu \nu} \Phi} \Pro_n \c{\partial_\nu \ell} = \hat \eta_\mu + \eta_\mu,
		\]
		and if $\Pi$ is a prior of type \ref{prior:independent_haar_series_histograms} with $H$ twice differentiable, we have
		\[
		\Div_\pi V = p\p{1-\frac{1}{n}} - \Phi - \sum_{\mu = 1}^{p-1} \sigma_\mu^{-1} H'\p{\theta^\mu/\sigma_\mu}\p{\hat \eta_\mu + \eta_\mu}
		\]
		and
		\[
			\abs{\inner{\nabla \Div_\pi V | V}} \leq 4\p{1 + \sigma_-^{-2} \norm{H''}_\infty + \norm{H'}_\infty \sqrt{p} \sigma_-^{-1}} \p{\norm{f_0}_\infty + \norm{\hat f - \pi_p\c{f_0}}_\infty + \norm{f_\theta - \hat f}_\infty + 1}^2.
		\]
		\item In $\alpha - \omega$ coordinates we have
		\[
		\nabla^\mu \Phi = 2p\p{\omega_\mu - \omega_p}, \quad V^\mu = p\p{\hat \omega_\mu + \omega_\mu - \hat \omega_p - \omega_p},
		\]
		and if $\Pi$ is a Dirichlet prior \ref{prior:dirichlet} with parameter $m \in \p{0,\infty}^p$ we have
		\[
		\Div_\pi V = p\p{1 - \frac{1}{n}} - \Phi + \norm{m}_1 \int \p{h_{\bar m} - h_\omega}\p{\hat f + h_\omega}
		\]
		and
		\[
		\abs{\inner{\nabla \Div_\pi V | V}} \leq 6p \p{1 + \norm{m}_\infty} \p{1 + \norm{h_\omega - \hat f}_\infty + \norm{\hat f - \pi_p\c{f_0}}_\infty + \norm{f_0}_\infty}^2.
		\]
	\end{enumerate}
\end{proposition}
\begin{proof}
	We start the proof by using the $\theta - \eta$ coordinates. First, because $f_\theta \in \FF_p$ by Parseval's identity we have
	\[
	\Phi = \int f_\theta^2 - \frac{p}{n} = \inner{\ind{}|f_\theta}_2^2 + \sum_{\mu = 1}^{p-1} \inner{u_\mu|f_\theta}_2^2 - \frac{p}{n} = 1 + \sum_{\mu = 1}^{p-1} \eta_\mu^2 - \frac{p}{n},
	\]
	so that $\nabla^\mu \Phi = \partial^\mu \Phi = 2\eta_\mu, \quad \partial^{\mu \nu} \Phi = 2\delta^{\mu \nu}$, which implies that
	\[
	\nabla \Phi = 2\sum_{\mu = 1}^{p-1} \eta_\mu \partial_\mu \ell = 2\sum_{\mu = 1}^{p-1} \eta_\mu\p{u_\mu - \eta_\mu} = 2\p{f_\theta - \int f_\theta^2} = 2\p{h_\omega - \int h_\omega^2},
	\]
	and
	\begin{align*}
		V^\mu & = \nabla^\mu \Phi + \frac{1}{2} \p{\nabla^{\mu \nu} \Phi - \frac{1}{2} T^{\mu \nu \rho} \partial_\rho \Phi} \Pro_n \c{\partial_\nu \ell} \\
		& = \partial^\mu \Phi + \frac{1}{2} \p{\partial^{\mu \nu} \Phi - \b{- \frac{1}{2} T_{~\mu \nu}^\rho\p{\eta} \partial^\rho \Phi} - \frac{1}{2} T_{~~\rho}^{\mu \nu} \partial^\rho \Phi} \Pro_n \c{\partial_\nu \ell} \\
		& = \partial^\mu \Phi + \frac{1}{2} \p{\partial^{\mu \nu} \Phi - \b{- \frac{1}{2} T_{\rho~~}^{~\mu \nu} \partial^\rho \Phi} - \frac{1}{2} T_{~~\rho}^{\mu \nu} \partial^\rho \Phi} \Pro_n \c{\partial_\nu \ell} \\
		& = \partial^\mu \Phi + \frac{1}{2} \p{\partial^{\mu \nu} \Phi} \Pro_n \c{\partial_\nu \ell} = 2\eta_\mu + \Pro_n\c{\partial_\mu \ell} \\
		& = 2\eta_\mu + \Pro_n \c{u_\mu - \eta_\mu} = \Pro_n \c{u_\mu} + \eta_\mu = \hat \eta_\mu + \eta_\mu.
	\end{align*}
	In particular this also implies
	\[
	V = V^\mu \partial_\mu \ell = \sum_{\mu = 1}^{p-1} \p{\hat \eta_\mu + \eta_\mu}\p{u_\mu - \eta_\mu} = \hat f + f_\theta - P_\theta\c{\hat f + f_\theta}.
	\]
	Because $f_\theta = h_\omega$ we also have $V = \hat f + h_\omega - P_\omega\c{\hat f + h_\omega}$. Furthermore
	\begin{align*}
		\hat \Phi = \Phi + \frac{1}{n} \inner{V | \nabla \ell_n} & = \int f_\theta^2 - \frac{p}{n} + \frac{1}{n} V^\mu \partial_\mu \ell_n = 1 - \frac{p}{n} + \sum_{\mu = 1}^{p-1} \eta_\mu^2 + \frac{1}{n} \p{\hat \eta_\mu + \eta_\mu} \sum_{i = 1}^n \p{u_\mu(x_i) - \eta_\mu} \\
		& = 1 - \frac{p}{n} + \sum_{\mu = 1}^{p-1} \eta_\mu^2 + \p{\hat \eta_\mu + \eta_\mu} \p{\hat \eta_\mu - \eta_\mu} = 1 - \frac{p}{n} + \sum_{\mu = 1}^{p-1} \hat \eta_\mu^2 = \int \hat f^2 - \frac{p}{n}.
	\end{align*}
	Moreover
	\begin{align*}
		\Div_\pi V & = \Div V + \inner{\nabla \ln \pi | V} = \nabla_\mu V^\mu + \p{\partial_\mu \ln \pi} V^\mu = \partial_\mu V^\mu + \Gamma_{~\mu \nu}^\nu V^\nu  + \p{\partial_\mu \ln \Pi} V^\mu - \Gamma_{~\mu \nu}^\nu V^\mu \\
		& = \sum_{\mu = 1}^{p-1} \b{\partial_\mu V^\mu + \p{\partial_\mu \ln \Pi} V^\mu},
	\end{align*}
	and since $V^\mu = \hat \eta_\mu + \eta_\mu$, we have $\partial_\mu V^\mu = \frac{\partial \eta_\mu}{\partial \theta^\mu} = g_{\mu \mu}$ which implies that
	\[
	\sum_{\mu = 1}^{p-1} \partial_\mu V^\mu = \sum_{\mu = 1}^{p-1} g_{\mu \mu} = \sum_{\mu = 1}^{p-1} \int u_\mu^2 f_\theta - \eta_\mu^2 = p-1 - \p{\Phi + \frac{p}{n} - 1} = p\p{1-\frac{1}{n}} - \Phi,
	\]
	where we have used the equality $\sum_{\mu = 1}^{p-1} u_\mu^2 = p-1$ which is satisfied by the Haar basis. For a prior of type \ref{prior:independent_haar_series_histograms} we have $\partial_\mu \ln \Pi = - \sigma_\mu^{-1} H'\p{\theta^\mu/\sigma_\mu}$, therefore
	\[
	\Div_\pi V = p\p{1-\frac{1}{n}} - \Phi - \sum_{\mu = 1}^{p-1} \sigma_\mu^{-1} H'\p{\theta^\mu/\sigma_\mu}\p{\hat \eta_\mu + \eta_\mu}.
	\]
	Using $\partial_\mu \Phi = 2 \eta_\mu$ and $\partial \eta_\nu/\partial \theta^\mu = g_{\mu \nu}$ this also implies
	\[
	\partial_\mu \Div_\pi V = - 2\eta_\mu - \sigma_\mu^{-2} H''\p{\theta^\mu/\sigma_\mu} \p{\hat \eta_\mu + \eta_\mu} - \sum_{\nu = 1}^{p-1} \sigma_\nu^{-1} H'\p{\theta^\nu/\sigma_\nu} g_{\mu \nu},
	\]
	and therefore
	\begin{align*}
		\inner{\nabla \Div_\pi V | V} & = V^\mu \partial_\mu \Div_\pi V = \sum_{\mu = 1}^{p-1} \p{\hat \eta_\mu + \eta_\mu}\p{- 2\eta_\mu - \sigma_\mu^{-2} H''\p{\theta^\mu/\sigma_\mu} \p{\hat \eta_\mu + \eta_\mu} - \sum_{\nu = 1}^{p-1} \sigma_\nu^{-1} H'\p{\theta^\nu/\sigma_\nu} g_{\mu \nu}} \\
		& = -2 \int \p{\hat f + f_\theta - 2}\p{f_\theta - 1} - \sum_{\mu = 1}^{p-1} \sigma_\mu^{-2} H''\p{\theta^\mu/\sigma_\mu}\p{\hat \eta_\mu + \eta_\mu}^2 - \sum_{\mu,\nu = 1}^{p-1} \sigma_\nu^{-1} H'\p{\theta^\nu/\sigma_\nu} \p{\hat \eta_\mu + \eta_\mu} g_{\mu \nu}.
	\end{align*}
	First
	\begin{align*}
		\abs{\int \p{\hat f + f_\theta - 2}\p{f_\theta - 1}} & = \abs{\int \p{\hat f + f_\theta - 2}f_\theta} \leq \int \abs{\hat f + f_\theta - 2}f_\theta \\
		& \leq 2 + \norm{f_\theta - \hat f}_\infty + 2\norm{\hat f - \pi_p\c{f_0}}_\infty + 2 \norm{\pi_p\c{f_0}}_\infty \\
		& \leq 2 + 2 \norm{f_0}_\infty + \norm{f_\theta - \hat f}_\infty + 2\norm{\hat f - \pi_p\c{f_0}}_\infty.
	\end{align*}
	Moreover by Proposition \ref{proposition:metric_tensor_bounds_histograms} and Cauchy-Schwarz inequality we have
	\begin{align*}
		\abs{\sum_{\mu,\nu = 1}^{p-1} \sigma_\nu^{-1} H'\p{\theta^\nu/\sigma_\nu} \p{\hat \eta_\mu + \eta_\mu} g_{\mu \nu}} & \leq \norm{H'}_\infty \norm{f_\theta}_\infty \p{\sum_{\mu = 1}^{p-1} \sigma_\mu^{-2}}^{1/2} \p{\sum_{\mu = 1}^{p-1} \p{\hat \eta_\mu + \eta_\mu}^2 }^{1/2} \\
		& \leq \norm{H'}_\infty \b{\norm{f_\theta - \hat f}_\infty + \norm{\hat f - \pi_p\c{f_0}}_\infty + \norm{\pi_p\c{f_0}}_\infty} \sqrt{p} \sigma_-^{-1} \norm{\hat f + f_\theta - 2}_2 \\
		& \leq \norm{H'}_\infty \b{\norm{f_\theta - \hat f}_\infty + \norm{\hat f - \pi_p\c{f_0}}_\infty + \norm{f_0}_\infty} \sqrt{p} \sigma_-^{-1} \\
		& \times \b{2 + 2\norm{f_0}_\infty + 2\norm{\hat f - \pi_p\c{f_0}}_\infty + \norm{f_\theta - \hat f}_\infty},
	\end{align*}
	and similarly
	\begin{align*}
	\abs{\sum_{\mu = 1}^{p-1} \sigma_\mu^{-2} H''\p{\theta^\mu/\sigma_\mu}\p{\hat \eta_\mu + \eta_\mu}^2} & \leq \sigma_-^{-2} \norm{H''}_\infty \norm{\hat f + f_\theta - 2}_2^2 \\
	& = \sigma_-^{-2} \norm{H''}_\infty \p{\norm{2\pi_p\c{f_0} + 2\p{\hat f - \pi_p\c{f_0}} + f_\theta - \hat f - 2}_2}^2 \\
	& \leq \sigma_-^{-2} \norm{H''}_\infty \p{2\norm{\pi_p\c{f_0}}_2 + 2\norm{\hat f - \pi_p\c{f_0}}_2 + \norm{f_\theta - \hat f}_2 + 2}^2 \\
	& \leq 4\sigma_-^{-2} \norm{H''}_\infty \p{\norm{f_0}_\infty + \norm{\hat f - \pi_p\c{f_0}}_\infty + \norm{f_\theta - \hat f}_\infty + 1}^2,
	\end{align*}
	which yields
	\begin{align*}
		& \abs{\inner{\nabla \Div_\pi V | V}} \leq 4 + 4 \norm{f_0}_\infty + 2\norm{f_\theta - \hat f}_\infty + 4\norm{\hat f - \pi_p\c{f_0}}_\infty \\
		& + 4\sigma_-^{-2} \norm{H''}_\infty \p{\norm{f_0}_\infty + \norm{\hat f - \pi_p\c{f_0}}_\infty + \norm{f_\theta - \hat f}_\infty + 1}^2 \\
		& + 2\norm{H'}_\infty \b{\norm{f_\theta - \hat f}_\infty + \norm{\hat f - \pi_p\c{f_0}}_\infty + \norm{f_0}_\infty}^2 \sqrt{p} \sigma_-^{-1} \\
		& \leq 4\p{1 + \sigma_-^{-2} \norm{H''}_\infty + \norm{H'}_\infty \sqrt{p} \sigma_-^{-1}} \p{\norm{f_0}_\infty + \norm{\hat f - \pi_p\c{f_0}}_\infty + \norm{f_\theta - \hat f}_\infty + 1}^2.
	\end{align*}
	We now switch to the $\alpha - \omega$ coordinate system. We have $\Phi = \int f_\theta^2 - \frac{p}{n} = p\sum_{\mu = 1}^p \omega_\mu^2 - \frac{p}{n}$, therefore for any $1 \leq \mu \leq p-1$, since $\omega_p = 1 - \sum_{\mu = 1}^{p-1} \omega_\mu$,
	\[
	\partial^\mu \Phi = \frac{\partial \Phi}{\partial \omega_\mu} = 2p \omega_\mu + 2p \omega_p \frac{\partial \omega_p}{\partial \omega_\mu} = 2p \p{\omega_\mu - \omega_p}, \quad \partial^{\mu \nu} \Phi = 2p \frac{\partial}{\partial \omega_\nu} \p{\omega_\mu - \omega_p} = 2p \p{\delta^{\mu \nu} + 1}.
	\]
	As for the $\theta-\eta$ coordinate system this implies that
	\begin{align*}
		V^\mu & = \partial^\mu \Phi + \frac{1}{2} \p{\partial^{\mu \nu} \Phi} \Pro_n \c{\partial_\nu \ell} = 2p \p{\omega_\mu - \omega_p} + \sum_{\nu = 1}^{p-1} p \p{\delta^{\mu \nu} + 1} \p{\hat \omega_\nu - \omega_\nu} \\
		& = 2p \p{\omega_\mu - \omega_p} + p \p{\hat \omega_\mu - \omega_\mu} + p \sum_{\nu = 1}^{p-1} \hat \omega_\nu - \omega_\nu = 2p \p{\omega_\mu - \omega_p} + p \p{\hat \omega_\mu - \omega_\mu} + p \p{\omega_p - \hat \omega_p} \\
		& = p \p{\hat \omega_\mu + \omega_\mu} - p \p{\hat \omega_p + \omega_p},
	\end{align*}
	and in particular, for any $1 \leq \mu \leq p-1$,
	\[
	\partial_\mu V^\mu = p \frac{\partial}{\partial \alpha^\mu} \p{\hat \omega_\mu + \omega_\mu - \hat \omega_p - \omega_p} = p \p{\frac{\partial \omega_\mu}{\partial \alpha^\mu} + \sum_{\nu = 1}^{p-1} \frac{\partial \omega_\nu}{\partial \alpha^\mu}} = p \p{g_{\mu \mu} + \sum_{\nu = 1}^{p-1} g_{\mu \nu}},
	\]
	which implies that
	\begin{align*}
		\sum_{\mu = 1}^{p-1} \partial_\mu V^\mu & = p \sum_{\mu = 1}^{p-1} \p{g_{\mu \mu} + \sum_{\nu = 1}^{p-1} g_{\mu \nu}} = p \sum_{\mu = 1}^{p-1} \p{\omega_\mu\p{1-\omega_\mu} + \sum_{\nu = 1}^{p-1} \delta_{\mu \nu} \omega_\mu - \omega_\mu \omega_\nu} \\
		& = p \sum_{\mu = 1}^{p-1} \p{\omega_\mu\p{1-\omega_\mu} + \omega_\mu - \omega_\mu \p{1-\omega_p}} = p \sum_{\mu = 1}^{p-1} \p{\omega_\mu\p{1-\omega_\mu} + \omega_p \omega_\mu} \\
		& = p\sum_{\mu = 1}^{p-1} \omega_\mu\p{1-\omega_\mu} + p \omega_p \sum_{\mu = 1}^{p-1} \omega_\mu = p\sum_{\mu = 1}^{p-1} \omega_\mu\p{1-\omega_\mu} + p \omega_p\p{1-\omega_p} \\
		& = p\sum_{\mu = 1}^p \omega_\mu\p{1-\omega_\mu} = p\p{1 - \frac{1}{n}} - \Phi.
	\end{align*}
	Therefore, using also Proposition \ref{proposition:density_dirichlet_alpha} we obtain
	\begin{align*}
		\Div_\pi V & = \Div V + \inner{\nabla \ln \pi | V} = \nabla_\mu V^\mu + \p{\partial_\mu \ln \pi} V^\mu = \partial_\mu V^\mu + \Gamma_{~\mu \nu}^\nu V^\nu  + \p{\partial_\mu \ln \Pi} V^\mu - \Gamma_{~\mu \nu}^\nu V^\mu \\
		& = \sum_{\mu = 1}^{p-1} \b{\partial_\mu V^\mu + \p{\partial_\mu \ln \Pi} V^\mu} = p\p{1 - \frac{1}{n}} - \Phi + \sum_{\mu = 1}^{p-1} \norm{m}_1 \p{\bar m_\mu - \omega_\mu} p\p{\hat \omega_\mu + \omega_\mu - \hat \omega_p - \omega_p} \\
		& = p\p{1 - \frac{1}{n}} - \Phi + p\norm{m}_1 \sum_{\mu = 1}^{p-1} \p{\bar m_\mu - \omega_\mu}\p{\hat \omega_\mu + \omega_\mu} - p\norm{m}_1 \p{\hat \omega_p + \omega_p} \sum_{\mu = 1}^{p-1} \p{\bar m_\mu - \omega_\mu} \\
		& = p\p{1 - \frac{1}{n}} - \Phi + p\norm{m}_1 \sum_{\mu = 1}^{p-1} \p{\bar m_\mu - \omega_\mu}\p{\hat \omega_\mu + \omega_\mu} + p\norm{m}_1 \p{\hat \omega_p + \omega_p} \p{\bar m_p - \omega_p} \\
		& = p\p{1 - \frac{1}{n}} - \Phi + p\norm{m}_1 \sum_{\mu = 1}^p \p{\bar m_\mu - \omega_\mu}\p{\hat \omega_\mu + \omega_\mu} \\
		& = p\p{1 - \frac{1}{n}} - \Phi + \norm{m}_1 \int \p{h_{\bar m} - h_\omega}\p{\hat f + h_\omega}.
	\end{align*}
	As a result we also have
	\begin{align*}
		\partial^\mu \Div_\pi V & = - \partial^\mu \Phi + p\norm{m}_1 \sum_{\nu = 1}^p \frac{\partial}{\partial \omega_\mu} \p{\bar m_\nu - \omega_\nu}\p{\hat \omega_\nu + \omega_\nu} \\
		& = - 2p\p{\omega_\mu - \omega_p} + p\norm{m}_1 \p{\p{\bar m_\mu - \omega_\mu} - \p{\hat \omega_\mu + \omega_\mu}} + p\norm{m}_1 \frac{\partial}{\partial \omega_\mu} \p{\bar m_p - \omega_p}\p{\hat \omega_p + \omega_p} \\
		& = - 2p\p{\omega_\mu - \omega_p} + p\norm{m}_1 \p{\p{\bar m_\mu - \omega_\mu} - \p{\hat \omega_\mu + \omega_\mu}} + p\norm{m}_1 \p{\p{\hat \omega_p + \omega_p} - \p{\bar m_p - \omega_p}},
	\end{align*}
	which using $\sum_{\mu = 1}^{p-1} \ind{I_\mu} = 1 - \ind{I_p}$ implies that
	\begin{align*}
	\nabla \Div_\pi V & = \p{\partial^\mu \Div_\pi V} \p{\partial_\mu \ell} = -2p \sum_{\mu = 1}^{p-1} \p{\omega_\mu - \omega_p} \p{\ind{I_\mu} - \omega_\mu} \\
	& + p \norm{m}_1 \sum_{\mu = 1}^{p-1} \p{\p{\bar m_\mu - \omega_\mu} - \p{\hat \omega_\mu + \omega_\mu}}\p{\ind{I_\mu} - \omega_\mu} - p \norm{m}_1 \sum_{\mu = 1}^{p-1} \p{\p{\bar m_p - \omega_p} - \p{\hat \omega_p + \omega_p}}\p{\ind{I_\mu} - \omega_\mu} \\
	& = -2p \sum_{\mu = 1}^p \omega_\mu \p{\ind{I_\mu} - \omega_\mu} + p \norm{m}_1 \sum_{\mu = 1}^p \p{\p{\bar m_\mu - \omega_\mu} - \p{\hat \omega_\mu + \omega_\mu}}\p{\ind{I_\mu} - \omega_\mu} \\
	& = -2\p{h_\omega - \int h_\omega^2} + \norm{m}_1\p{\p{h_{\bar m} - h_\omega} - \p{\hat f + h_\omega} -  \int \p{\p{h_{\bar m} - h_\omega} - \p{\hat f + h_\omega}} h_\omega}
	\end{align*}
	and therefore, since $V = \hat f + h_\omega - \int \p{\hat f + h_\omega} h_\omega$,
	\begin{align*}
	\inner{\nabla \Div_\pi V | V} & = P_\omega \c{\p{\Div_\pi V} V} \\
	& = \int \p{-2h_\omega + \norm{m}_1 \p{\p{h_{\bar m} - h_\omega} - \p{\hat f + h_\omega}}}\p{\hat f + h_\omega - \int \p{\hat f + h_\omega} h_\omega} h_\omega.
	\end{align*}
	In particular
	\begin{align*}
	\abs{\inner{\nabla \Div_\pi V | V}} & = \abs{\int \p{-2h_\omega + \norm{m}_1 \p{\p{h_{\bar m} - h_\omega} - \p{\hat f + h_\omega}}}\p{\hat f + h_\omega - \int \p{\hat f + h_\omega} h_\omega} h_\omega} \\
	& \leq \int \p{2h_\omega + \norm{m}_1 \abs{\p{h_{\bar m} - h_\omega} - \p{\hat f + h_\omega}}} \abs{\hat f + h_\omega - \int \p{\hat f + h_\omega} h_\omega} h_\omega \\
	& \leq 4\norm{h_\omega}_\infty + \norm{m}_1 \norm{\p{h_{\bar m} - h_\omega} - \p{\hat f + h_\omega}}_\infty \norm{\hat f + h_\omega}_\infty \\
	& \leq 4\norm{h_\omega}_\infty + \norm{m}_1 \p{\norm{\hat f + h_\omega}_\infty + \norm{h_{\bar m}}_\infty + \norm{h_\omega}_\infty} \norm{\hat f + h_\omega}_\infty,
	\end{align*}
	which combined with
	\[
	\norm{h_\omega}_\infty \leq \norm{h_\omega - \hat f}_\infty + \norm{\hat f - \pi_p\c{f_0}}_\infty + \norm{\pi_p\c{f_0}}_\infty \leq \norm{h_\omega - \hat f}_\infty + \norm{\hat f - \pi_p\c{f_0}}_\infty + \norm{f_0}_\infty,
	\]
	\[
	\norm{\hat f + h_\omega}_\infty \leq \norm{h_\omega - \hat f}_\infty + 2\norm{\hat f - \pi_p\c{f_0}}_\infty + 2\norm{\pi_p\c{f_0}}_\infty \leq \norm{h_\omega - \hat f}_\infty + 2\norm{\hat f - \pi_p\c{f_0}}_\infty + 2\norm{f_0}_\infty,
	\]
	\[
	\norm{h_{\bar m}}_\infty = p \norm{\bar m}_\infty = p \frac{\norm{m}_\infty}{\norm{m}_1}
	\]
	implies
	\begin{align*}
	& \abs{\inner{\nabla \Div_\pi V | V}} \\
	& \leq 4\p{\norm{h_\omega - \hat f}_\infty + \norm{\hat f - \pi_p\c{f_0}}_\infty + \norm{f_0}_\infty} \\
	& + p\norm{m}_\infty\p{1 + 2\norm{h_\omega - \hat f}_\infty + 3\norm{\hat f - \pi_p\c{f_0}}_\infty + 3\norm{f_0}_\infty} \p{\norm{h_\omega - \hat f}_\infty + 2\norm{\hat f - \pi_p\c{f_0}}_\infty + 2\norm{f_0}_\infty} \\
	& \leq \p{4 + 6p\norm{m}_\infty} \p{1 + \norm{h_\omega - \hat f}_\infty + \norm{\hat f - \pi_p\c{f_0}}_\infty + \norm{f_0}_\infty}^2 \\
	& \leq 6p \p{1 + \norm{m}_\infty} \p{1 + \norm{h_\omega - \hat f}_\infty + \norm{\hat f - \pi_p\c{f_0}}_\infty + \norm{f_0}_\infty}^2.
	\end{align*}
\end{proof}

\begin{proposition}\label{proposition:density_dirichlet_alpha}
	The Lebesgue density of the Dirichlet distribution $\Dir\p{m}$ in $\alpha$ coordinates is given by
	\[
	\Pi \p{\alpha} = \frac{1}{B(m)} \prod_{\mu = 1}^p \c{\frac{e^{\alpha^\mu}}{\sum_{\nu = 1}^p e^{\alpha^\nu}}}^{m_\mu}, \quad B\p{m} = \frac{1}{\Gamma \p{\norm{m}_1}} \prod_{\mu = 1}^p \Gamma\p{m_\mu}
	\]
	over the admissibility set $\b{\alpha \in \RR^p : \alpha^p = 0}$. In particular
	\[
	\frac{\partial}{\partial \alpha^\mu} \ln \Pi\p{\alpha} = m_\mu - \norm{m}_1 \omega_\mu = \norm{m}_1 \p{\bar m_\mu - \omega_\mu}, \quad \bar m := \frac{m}{\norm{m}_1}.
	\]
\end{proposition}
\begin{proof}{Proposition \ref{proposition:density_dirichlet_alpha}}\\
	By the change of variable formula we have
	\[
	\Pi \c{\alpha} = \Pi \p{\omega\p{\alpha}} \cdot \abs{\frac{\partial \omega}{\partial \alpha}},
	\]
	but using $\frac{\partial \omega_\mu}{\partial \alpha^\nu} = g_{\mu \nu}$ (in $\alpha$ coordinates), this implies
	\[
	\Pi \p{\alpha} = \Pi \p{\omega\p{\alpha}} \cdot \abs{g(\alpha)}.
	\]
	First
	\[
	\Pi \p{\omega\p{\alpha}} = \frac{\prod_{\mu = 1}^p \omega_\mu^{m_\mu - 1}}{B(m)} = \frac{1}{B(m)} \prod_{\mu = 1}^p \c{\frac{e^{\alpha^\mu}}{\sum_{\nu = 1}^p e^{\alpha^\nu}}}^{m_\mu - 1}
	\]
	and using Lemma \ref{lemma:determinant_metric_tensor} below we also have
	\[
	\abs{g(\alpha)} = \prod_{\mu = 1}^p \frac{e^{\alpha^\mu}}{\sum_{\nu = 1}^p e^{\alpha^\nu}},
	\]
	so that
	\[
	\Pi \p{\alpha} = \frac{1}{B(m)} \prod_{\mu = 1}^p \c{\frac{e^{\alpha^\mu}}{\sum_{\nu = 1}^p e^{\alpha^\nu}}}^{m_\mu - 1} \cdot \prod_{\mu = 1}^p \frac{e^{\alpha^\mu}}{\sum_{\nu = 1}^p e^{\alpha^\nu}} = \frac{1}{B(m)} \prod_{\mu = 1}^p \c{\frac{e^{\alpha^\mu}}{\sum_{\nu = 1}^p e^{\alpha^\nu}}}^{m_\mu}.
	\]
\end{proof}

\begin{lemma}\label{lemma:determinant_metric_tensor}
	In the $\alpha$ coordinate system the determinant $\abs{g}$ of the metric tensor is given by
	\[
	\abs{g(\alpha)} = \prod_{\mu = 1}^p \omega_\mu = \prod_{\mu = 1}^p \frac{e^{\alpha^\mu}}{\sum_{\nu = 1}^p e^{\alpha^\nu}}.
	\]
\end{lemma}
\begin{proof}{Lemma \ref{lemma:determinant_metric_tensor}}\\
	We have $g_{\mu \nu} = \delta_{\mu \nu} \omega_\mu - \omega_\mu \omega_\nu = \c{\Diag \p{\omega} - \omega \omega^T}_{\mu \nu}$, so that the matrix determinant lemma implies $\abs{g} = \abs{\Diag \p{\omega} \p{1 - \omega^T \Omega^{-1} \omega}}$, but $\omega^T \Omega^{-1} \omega = \sum_{\mu = 1}^{p-1} \omega_\mu \frac{1}{\omega_\mu} \omega_\mu = 1 - \omega_p$, therefore
	\[
	\abs{g(\alpha)} = \abs{\Diag \p{\omega} \p{1 - \p{1-\omega_p}}} = \abs{\prod_{\mu = 1}^{p-1} \omega_\mu \times \omega_p} = \prod_{\mu = 1}^p \omega_\mu.
	\]
	Now we use the formula
	\[
	\omega_\mu = \frac{e^{\alpha^\mu}}{\sum_{\nu = 1}^p e^{\alpha^\nu}}, \quad 1 \leq \mu \leq p
	\]
	to conclude.
\end{proof}

\begin{proof}{Proposition \ref{proposition:efficiency_estimator_squared_norm_histograms}}\\
	With $T(f) = f^2$ it suffices to show that $\hat \Psi_0 = \hat \Phi_1 + o_p\p{n^{-1/2}}$ where
	\[
	\hat \Phi_1 = \int T\p{f_0} + \p{\Pro_n - \Pro_0}\c{T'\p{f_0}}.
	\]
	For this we define
	\[
	f_p = \pi_p\c{f_0}, \quad \hat \Phi_2 =  \int T\p{f_p} + \p{\Pro_n - P_{f_p}}\c{T'\p{f_p}}
	\]
	and we show that
	\[
	\hat \Phi_2 - \hat \Phi_1 = o_p\p{n^{-1/2}}, \quad \hat \Psi_0 - \hat \Phi_2 = o_p\p{n^{-1/2}}.
	\]
	We have
	\begin{align*}
		\hat \Phi_1 - \b{ \int T\p{f_0} + \p{\Pro_n - \Pro_0}\c{T'\p{f_p}}} & = \p{\Pro_n - \Pro_0}\c{T'\p{f_0} - T'\p{f_p}} = 2\p{\Pro_n - \Pro_0}\c{f_0 - f_p} \\
		& = O_p\p{\frac{\norm{f_0 - f_p}_\infty}{\sqrt{n}}} = O_p\p{\frac{p^{-s}}{\sqrt{n}}} = o_p\p{n^{-1/2}}.
	\end{align*}
	Moreover because $T'\p{f_p} \in \FF_p$ we have $\Pro_n \c{T'\p{f_p}} = \int \hat f T'\p{f_p}$, therefore
	\begin{align*}
		\hat \Phi_2 - \b{\int T\p{f_0} + \p{\Pro_n - \Pro_0}\c{T'\p{f_p}}} & = \int T\p{f_p} + T'\p{f_p}\p{\hat f - f_p} - T\p{f_0} - T'\p{f_p}\p{\hat f - f_0} \\
		& = \int T\p{f_p} - T\p{f_0} - T'\p{f_p}\p{f_p - f_0} \\
		& = \norm{f_0 - f_p}_2^2 = O\p{p^{-2s}} = o\p{n^{-1/2}},
	\end{align*}
	since $p \gg n^{\frac{1}{4s}}$, which implies that $\hat \Phi_2 - \hat \Phi_1 = o_p\p{n^{-1/2}}$. We now show that $\hat \Phi_2 - \hat \Psi_0 = o_p\p{n^{-1/2}}$. We have
	\begin{align*}
		\int \hat f^2 & = \int \p{f_p + \hat f - f_p}^2 = \int f_p^2 + 2f_p\p{\hat f - f_p} + \p{\hat f - f_p}^2 \\
		& = \hat \Phi_2 + \int \p{\hat f - f_p}^2,
	\end{align*}
	and
	\[
	\int \p{\hat f - f_p}^2 = p \sum_{\mu = 1}^p \p{\hat \omega_\mu - \omega_{0\mu}}^2 = \frac{1}{n^2} \sum_{i,j} K\p{x_i,x_j},
	\]
	where $K$ is the kernel
	\[
	K\p{x,y} = p \sum_{\mu = 1}^p \p{\ind{I_\mu}\p{x} - \omega_{0\mu}} \p{\ind{I_\mu}\p{y} - \omega_{0\mu}}.
	\]
	We have
	\[
	\int \p{\hat f - f_p}^2 = \underbrace{\frac{1}{n} \int K\p{x,x} f_0(x)}_{=: U_0} + \underbrace{\frac{1}{n^2} \sum_{i = 1}^n \p{K\p{x_i,x_i} - \int K\p{x,x} f_0(x)}}_{=: U_1} + \underbrace{\frac{1}{n^2} \sum_{i \neq j} K\p{x_i,x_j}}_{=: U_2}.
	\]
	First
	\[
	U_0 = \frac{p}{n} \sum_{\mu = 1}^p \int \p{\ind{I_\mu} - \omega_{0\mu}}^2 f_0 = \frac{p}{n} \sum_{\mu = 1}^p \omega_{0\mu}\p{1-\omega_{0\mu}} = \frac{p}{n} - \frac{1}{n} \int f_p^2 = \frac{p}{n} + o\p{n^{-1/2}}.
	\]
	Moreover
	\[
	U_1 = O_p\p{n^{-3/2} \p{\int \p{K(x,x) - \int K(x,x)f_0(x)}^2 f_0(x)}^{1/2}},
	\]
	but
	\[
	K\p{x,x} = p \sum_{\mu = 1}^p \p{\ind{I_\mu}(x) - \omega_{0\mu}}^2 = p \sum_{\mu = 1}^p \b{\ind{I_\mu} + \omega_{0\mu}^2 - 2 \omega_{0\mu} \ind{I_\mu}} = p + \int f_p^2 - 2f_p,
	\]
	so that
	\[
	\int \p{K(x,x) - \int K(x,x)f_0(x)}^2 f_0(x) = 4 \int \p{f_p - \int f_pf_0}^2 f_0 \leq 4 \int f_p^2 f_0 \leq 4 \norm{f_p}_\infty^2 \leq 4 \norm{f_0}_\infty^2,
	\]
	which implies that $U_1 = O_p\p{n^{-3/2}} = o_p\p{n^{-1/2}}$. Finally because $U_2$ is a degenerate second order $U-$statistics we have
	\begin{align*}
		P_0 \c{U_2^2} & = \p{\frac{n-1}{n}}^2 \times \frac{1}{\binom{n}{2}} \times \binom{2}{2} \times \binom{n-2}{2-2} \times \int K^2\p{x,y} f_0(x)f_0(y) \\
		& \leq \frac{2}{n\p{n-1}} \int K^2\p{x,y} f_0(x)f_0(y),
	\end{align*}
	but
	\begin{align*}
		K\p{x,y} & = p \sum_{\mu = 1}^p \p{\ind{I_\mu}(x) - \omega_{0\mu}} \p{\ind{I_\mu}(y) - \omega_{0\mu}} = p \sum_{\mu = 1}^p \ind{x,y \in I_\mu} - f_p(x) - f_p(y) + \int f_p^2 \\
		& = p \ind{x,y \text{ in the same bin}} - p P_0 \c{\text{x,y in the same bin}} + \int f_p^2 - f_p(x) + \int f_p^2 - f_p(y),
	\end{align*}
	so that
	\[
	K^2\p{x,y} \leq 3 \p{p^2 \p{\ind{x,y \text{ in the same bin}} - P_0 \c{\text{x,y in the same bin}}}^2 + \p{\int f_p^2 - f_p(x)}^2 + \p{\int f_p^2 - f_p(y)}^2},
	\]
	and therefore
	\begin{align*}
		& \int K^2\p{x,y}f_0(x)f_0(y) \leq 3\p{p^2 P_0\c{\text{x,y in the same bin}} + 2 \int \p{f_p - \int f_p^2}^2 f_0} \\
		& \lesssim p^2 \times \sum_{\mu = 1}^p \omega_{0\mu}^2 = p \int f_p^2 \lesssim p,
	\end{align*}
	which implies $U_2 = O_p\p{\sqrt{p}/n} = o_p\p{n^{-1/2}}$, since $p \ll n$ by assumption. All in all this proves that $\hat \Psi_0 = \hat \Phi_1 + o_p\p{n^{-1/2}}$ and therefore the result.
\end{proof}

\begin{proposition}\label{proposition:nabla_phi_integral_functionals_histograms}
	With $\Phi = \int T(x,f(x))dx, \quad f \in \HH_p$ we have
	\[
	\frac{\partial \Phi}{\partial \omega_\mu} = p\int_{I_\mu} \frac{\partial T}{\partial f}\p{x,p\omega_\mu}dx - p\int_{I_p} \frac{\partial T}{\partial f}\p{x,p\omega_p}dx, \quad \mu = 1,\ldots,p-1
	\]
	and
	\[
	\nabla \Phi = \pi_p \c{ \frac{\partial T}{\partial f}\p{\cdot,f(\cdot)} } - \int \pi_p \c{ \frac{\partial T}{\partial f}\p{\cdot,f(\cdot)} }(x)f(x)dx.
	\]
\end{proposition}
\begin{proof}
	By definition $\nabla \Phi = \p{\partial^\mu \Phi}\p{\partial_\mu \ell}$, but since for a histogram $f = h_\omega \in \HH_p$ we have
	\[
	\Phi = \int T\p{x,f(x)}dx = \sum_{\mu = 1}^p \int_{I_\mu} T\p{x,p\omega_\mu}dx,
	\]
	this implies (in $\alpha-\omega$ coordinates)
	\[
	\partial^\mu \Phi = \frac{\partial \Phi}{\partial \omega_\mu} = p\int_{I_\mu} \frac{\partial T}{\partial f}\p{x,p\omega_\mu}dx - p\int_{I_p} \frac{\partial T}{\partial f}\p{x,p\omega_p}dx
	\]
	(because $\omega_p = 1 - (\omega_1 + \ldots + \omega_{p-1})$) and therefore, using also $\partial_\mu \ell = \ind{I_\mu} - \omega_\mu$,
	\begin{align*}
		\nabla \Phi & = p\sum_{\mu = 1}^{p-1} \b{\int_{I_\mu} \frac{\partial T}{\partial f}\p{x,p\omega_\mu}dx - \int_{I_p} \frac{\partial T}{\partial f}\p{x,p\omega_p}dx} \p{\ind{I_\mu} - \omega_\mu} \\
		& = p\sum_{\mu = 1}^p \b{\int_{I_\mu} \frac{\partial T}{\partial f}\p{x,p\omega_\mu}dx} \p{\ind{I_\mu} - \omega_\mu} \\
		& = \pi_p \c{ \frac{\partial T}{\partial f}\p{\cdot,f(\cdot)} } - \int \pi_p \c{ \frac{\partial T}{\partial f}\p{\cdot,f(\cdot)} }(x)f(x)dx,
	\end{align*}
	where we have used $\sum_{\mu = 1}^{p-1} \ind{I_\mu} - \omega_\mu = \omega_p - \ind{I_p}$.
\end{proof}

\begin{proposition}\label{proposition:delta_pi_Phi_integral_functionals_histograms}
	With $\Phi = \int T(x,f(x))dx, \quad f \in \HH_p$ we have
	\begin{enumerate}
		\item $\Delta_\pi \Phi = \int \frac{\partial^2 T}{\partial f^2}\p{x,f(x)} f(x)\p{p-f(x)} dx - \int \frac{\partial T}{\partial f}\p{x,f(x)}\b{\sum_{\mu = 1}^{p-1} \sigma_\mu^{-1} H'\p{\theta^\mu/\sigma_\mu}u_\mu(x)}dx$ for a prior of type \ref{prior:independent_haar_series_histograms},
		\item $\Delta_\pi \Phi = \int f(x)\p{p - f(x)} \frac{\partial^2 T}{\partial f^2}\p{x,f(x)}dx + \norm{m}_1 \int \p{h_{\bar m}(x) - f(x)} \frac{\partial T}{\partial f}\p{x,f(x)}dx$ for a prior of type \ref{prior:dirichlet}, where $\bar m = m/\norm{m}_1$.
	\end{enumerate}
\end{proposition}
\begin{proof}
	As proved in Section \ref{section:general_results}, for either the $\alpha-\omega$ or the $\theta - \eta$ coordinate system we have
	\[
	\Delta_\pi \Phi = \partial_\mu \nabla^\mu \Phi + \p{\partial_\mu \ln \Pi}\p{\nabla^\mu \Phi} = \partial_\mu \partial^\mu \Phi + \p{\partial_\mu \ln \Pi}\p{\partial^\mu \Phi}.
	\]
	\begin{enumerate}
		\item In the $\theta-\eta$ coordinate system, since $f = 1 + \sum_{\mu = 1}^{p-1} \eta_\mu u_\mu$ we have $\partial^\mu f = u_\mu$ and
		\[
		\partial^\mu \Phi = \frac{\partial}{\partial \eta_\mu} \int T\p{x,f(x)}dx = \int \frac{\partial T}{\partial f}\p{x,f(x)}u_\mu(x)dx,
		\]
		which implies that
		\[
		\partial_\mu \partial^\mu \Phi = g_{\mu \nu} \partial^\nu \partial^\mu \Phi = \int \frac{\partial^2 T}{\partial f^2}\p{x,f(x)} \b{\sum_{\mu,\nu = 1}^{p-1} g_{\mu \nu} u_\mu(x) u_\nu(x)} dx.
		\]
		Now since for any $x,y$ we have $1 + \sum_{\mu = 1}^{p-1} u_\mu(x) u_\mu(y) = p \sum_{\mu = 1}^p \ind{I_\mu}(x) \ind{I_\mu}(y)$, we obtain
		\begin{align*}
		\sum_{\mu,\nu = 1}^{p-1} g_{\mu \nu} u_\mu(x) u_\nu(x) & = P_\theta \c{\p{\sum_{\mu = 1}^{p-1} u_\mu(x) u_\mu}^2} - \p{\sum_{\mu = 1}^{p-1} u_\mu(x) \eta_\mu}^2 \\
		& = \Var_{y \sim P_\theta} \c{\sum_{\mu = 1}^{p-1} u_\mu(x) u_\mu(y)} = \Var_{y \sim P_\theta} \c{p \sum_{\mu = 1}^p \ind{I_\mu}(x)\ind{I_\mu}(x)} \\
		& = p^2\Var_{y \sim P_\theta} \c{\ind{y \text{ in the same bin as } x}} = p^2 \omega_{\mu(x)}\p{1-\omega_{\mu(x)}}
		\end{align*}
		where $\mu(x)$ denotes the index of the bin containing $x$. Therefore
		\begin{align*}
			\partial_\mu \partial^\mu \Phi & = g_{\mu \nu} \partial^\nu \partial^\mu \Phi = \int \frac{\partial^2 T}{\partial f^2}\p{x,f(x)} \b{\sum_{\mu,\nu = 1}^{p-1} g_{\mu \nu} u_\mu(x) u_\nu(x)} dx \\
			& = \sum_{\mu = 1}^p \int_{I_\mu} \frac{\partial^2 T}{\partial f^2}\p{x,f(x)} p^2 \omega_\mu\p{1-\omega_\mu} dx \\
			& = \sum_{\mu = 1}^p \int_{I_\mu} \frac{\partial^2 T}{\partial f^2}\p{x,f(x)} p f(x)\p{1-p^{-1}f(x)} dx \\
			& = \int \frac{\partial^2 T}{\partial f^2}\p{x,f(x)} f(x)\p{p-f(x)} dx.
		\end{align*}
		Moreover
		\[
		\p{\partial_\mu \ln \Pi}\p{\partial^\mu \Phi} = - \sum_{\mu = 1}^{p-1} \sigma_\mu^{-1} H'\p{\theta^\mu/\sigma_\mu} \int \frac{\partial T}{\partial f}\p{x,f(x)}u_\mu(x)dx,
		\]
		therefore
		\[
		\Delta_\pi \Phi = \int \frac{\partial^2 T}{\partial f^2}\p{x,f(x)} f(x)\p{p-f(x)} dx - \int \frac{\partial T}{\partial f}\p{x,f(x)}\b{\sum_{\mu = 1}^{p-1} \sigma_\mu^{-1} H'\p{\theta^\mu/\sigma_\mu}u_\mu(x)}dx.
		\]
		\item In the $\alpha-\omega$ coordinate system, by Proposition \ref{proposition:nabla_phi_integral_functionals_histograms} we have
		\[
		\partial^\mu \Phi = \frac{\partial \Phi}{\partial \omega_\mu} = p\int_{I_\mu} \frac{\partial T}{\partial f}\p{x,p\omega_\mu}dx - p\int_{I_p} \frac{\partial T}{\partial f}\p{x,p\omega_p}dx,
		\]
		therefore
		\begin{align*}
			\partial_\mu \partial^\mu \Phi & = g_{\mu \nu} \partial^\nu \partial^\mu \Phi = \sum_{\mu,\nu = 1}^{p-1} g_{\mu \nu} \frac{\partial^2 \Phi}{\partial \omega_\mu \partial \omega_\nu} = \sum_{\mu,\nu = 1}^{p-1} g_{\mu \nu} \frac{\partial}{\partial \omega_\nu} \b{p\int_{I_\mu} \frac{\partial T}{\partial f}\p{x,p\omega_\mu}dx - p\int_{I_p} \frac{\partial T}{\partial f}\p{x,p\omega_p}dx} \\
			& = p \sum_{\mu,\nu = 1}^{p-1} g_{\mu \nu} \b{ p\delta_{\mu \nu} \int_{I_\mu} \frac{\partial^2 T}{\partial f^2}\p{x,p\omega_\mu}dx + p \int_{I_p} \frac{\partial^2 T}{\partial f^2}\p{x,p\omega_p}dx} \\
			& = \p{p^2 \sum_{\mu = 1}^{p-1} g_{\mu \mu} \int_{I_\mu} \frac{\partial^2 T}{\partial f^2}\p{x,p\omega_\mu}dx} + \p{p^2 \int_{I_p} \frac{\partial^2 T}{\partial f^2}\p{x,p\omega_p}dx \sum_{\mu,\nu = 1}^{p-1} g_{\mu \nu}},
		\end{align*}
		but $g_{\mu \mu} = \omega_\mu\p{1-\omega_\mu}$ and
		\[
		\sum_{\mu,\nu = 1}^{p-1} g_{\mu \nu} = \sum_{\mu,\nu = 1}^{p-1} \delta_{\mu \nu} \omega_\mu - \omega_\mu \omega_\nu = \sum_{\mu = 1}^{p-1} \omega_\mu - \p{\sum_{\mu = 1}^{p-1} \omega_\mu}^2 = 1-\omega_p - \p{1-\omega_p}^2 = \omega_p\p{1-\omega_p},
		\]
		therefore
		\[
		\partial_\mu \partial^\mu \Phi = p^2 \sum_{\mu = 1}^p \omega_\mu\p{1-\omega_\mu} \int_{I_\mu} \frac{\partial^2 T}{\partial f^2}\p{x,p\omega_\mu}dx = \int f(x)\p{p - f(x)} \frac{\partial^2 T}{\partial f^2}\p{x,f(x)}dx.
		\]
		Furthermore, using Proposition \ref{proposition:density_dirichlet_alpha} we have $\partial_\mu \ln \Pi = \norm{m}_1 \p{\bar m_\mu - \omega_\mu}$ where $\bar m = m/\norm{m}_1$. As a result we obtain
		\begin{align*}
			\p{\partial_\mu \ln \Pi}\p{\partial^\mu \Phi} & = \sum_{\mu = 1}^{p-1} \norm{m}_1 \p{\bar m_\mu - \omega_\mu} \b{p\int_{I_\mu} \frac{\partial T}{\partial f}\p{x,p\omega_\mu}dx - p\int_{I_p} \frac{\partial T}{\partial f}\p{x,p\omega_p}dx} \\
			& = p \norm{m}_1 \sum_{\mu = 1}^p \p{\bar m_\mu - \omega_\mu} \int_{I_\mu} \frac{\partial T}{\partial f}\p{x,p\omega_\mu}dx \\
			& = \norm{m}_1 \int \p{h_{\bar m}(x) - f(x)} \frac{\partial T}{\partial f}\p{x,f(x)}dx,
		\end{align*}
		and therefore
		\[
		\Delta_\pi \Phi = \int f(x)\p{p - f(x)} \frac{\partial^2 T}{\partial f^2}\p{x,f(x)}dx + \norm{m}_1 \int \p{h_{\bar m}(x) - f(x)} \frac{\partial T}{\partial f}\p{x,f(x)}dx.
		\]
	\end{enumerate}
\end{proof}

\begin{lemma}\label{lemma:bound_hat_Phi_integral_functionals_histograms}
	Let $\varepsilon_n = p^{-s} + \sqrt{\frac{p \ln n}{n}}$ and $B_n = \b{f \in \HH_p : \norm{f - f_0}_\infty < C\varepsilon_n}$ for a fixed $C>0$. Then with $\hat \Phi = \Phi + \frac{1}{n} \inner{\nabla \Phi | \nabla \ell_n}$ we have, for any $M_n \to \infty$,
	\[
	\sup_{f \in B_n} \abs{\hat \Phi - \hat \Psi_0} = O_p\p{p^{-2s} + p^{-(s+\gamma)} + M_n p^{-s \wedge \gamma} \sqrt{\frac{p \ln n}{n}} + M_n \frac{p \ln n}{n}}.
	\]
	whenever $s \wedge \gamma > 1/2$ and $\p{\ln n}^{\frac{1}{2(s \wedge \gamma) - 1}} \ll p \ll \p{n/\ln^2 n}^{1/2}$.
\end{lemma}
\begin{proof}
	With $f_p = \pi_p\c{f_0}$, since by definition $\frac{1}{n} \inner{\nabla \Phi| \nabla \ell_n} = \Pro_n \c{\nabla \Phi}$ using Proposition \ref{proposition:nabla_phi_integral_functionals_histograms} we have
	\begin{align*}
		\hat \Phi - \hat \Psi_0 & = \int T\p{x,f(x)}dx + \Pro_n \c{ \pi_p\c{\frac{\partial T}{\partial f}\p{\cdot,f(\cdot)} }} - \int f(x) \frac{\partial T}{\partial f}\p{x,f(x)}dx \\
		& - \b{\int T\p{x,f_0(x)}dx + \p{\Pro_n - \Pro_0}\c{\frac{\partial T}{\partial f}\p{\cdot,f_0(\cdot)}}} \\
		& = \int T\p{x,f(x)}dx + \Pro_n \c{\frac{\partial T}{\partial f}\p{\cdot,f(\cdot)}} - \int f(x) \frac{\partial T}{\partial f}\p{x,f(x)}dx - \b{\int T\p{x,f_0(x)}dx + \p{\Pro_n - \Pro_0}\c{\frac{\partial T}{\partial f}\p{\cdot,f_0(\cdot)}}} \\
		& + \Pro_n \c{ \p{\pi_p - I}\c{\frac{\partial T}{\partial f}\p{\cdot,f(\cdot)}}} \\
		& = \underbrace{\int \b{T\p{x,f(x)} - T(x,f_0(x)) - (f(x)-f_0(x)) \frac{\partial T}{\partial f}(x,f(x))}dx}_{=I} + \underbrace{\Pro_0\c{ \p{\pi_p - I}\c{\frac{\partial T}{\partial f}\p{\cdot,f(\cdot)}}}}_{=II} \\
		& + \underbrace{\p{\Pro_n - \Pro_0} \c{\frac{\partial T}{\partial f}\p{\cdot,f(\cdot)} - \frac{\partial T}{\partial f}\p{\cdot,f_0(\cdot)} + \p{\pi_p - I}\c{\frac{\partial T}{\partial f}\p{\cdot,f(\cdot)}}}}_{=III}.
	\end{align*}
	For $f$ such that $\norm{f-f_0}_\infty < \eta$ the term $I$ can be bounded by a constant multiple of $\varepsilon_n^2$ using assumption \ref{assumption:T} and a second order Taylor expansion of $T$ in its second argument. Moreover by $L^2$ orthogonality we have
	\[
	\Pro_0\c{ \p{\pi_p - I}\c{\frac{\partial T}{\partial f}\p{\cdot,f(\cdot)}}} = \int \p{\pi_p - I}\c{\frac{\partial T}{\partial f}\p{\cdot,f(\cdot)}} f_0 = \int \p{\pi_p - I}\c{\frac{\partial T}{\partial f}\p{\cdot,f(\cdot)}} \p{f_0 - \pi_p\c{f_0}},
	\]
	so that
	\[
	\abs{\Pro_0\c{ \p{\pi_p - I}\c{\frac{\partial T}{\partial f}\p{\cdot,f(\cdot)}}}} \leq \norm{\p{\pi_p - I}\c{\frac{\partial T}{\partial f}\p{\cdot,f(\cdot)}}}_\infty \norm{f_0 - \pi_p\c{f_0}}_\infty.
	\]
	First because $f_0 \in \CC^s$ we have $\norm{f_0 - \pi_p\c{f_0}}_\infty \lesssim p^{-s}$. Moreover, for any $\mu \in \b{1,\ldots,p}$ and $x \in I_\mu$ we have
	\[
	\p{\pi_p - I}\c{\frac{\partial T}{\partial f}\p{\cdot,f(\cdot)}}(x) = p \int_{I_\mu} \frac{\partial T}{\partial f}\p{y,p\omega_\mu}dy - \frac{\partial T}{\partial f}\p{x,p\omega_\mu} = p \int_{I_\mu} \b{\frac{\partial T}{\partial f}\p{y,p\omega_\mu} - \frac{\partial T}{\partial f}\p{x,p\omega_\mu}} dy,
	\]
	so that using assumption \ref{assumption:T} we obtain
	\[
	\norm{\p{\pi_p - I}\c{\frac{\partial T}{\partial f}\p{\cdot,f(\cdot)}}}_\infty \leq \max_{1 \leq \mu \leq p} \max_{x,y \in I_\mu} \abs{\frac{\partial T}{\partial f}\p{y,p\omega_\mu} - \frac{\partial T}{\partial f}\p{x,p\omega_\mu}} \leq Cp^{-\gamma},
	\]
	which implies that  the term $II$ can be bounded by a multiple of $p^{-(s+\gamma)}$. Finally we bound the empirical process term $III$ using a suitable maximal inequality, for which a simple $\epsilon$-net argument suffices here : for $f \in B_n$ define $u(f)$ as
	\[
	u(f) = \frac{\partial T}{\partial f}\p{\cdot,f(\cdot)} - \frac{\partial T}{\partial f}\p{\cdot,f_0(\cdot)} + \p{\pi_p - I}\c{\frac{\partial T}{\partial f}\p{\cdot,f(\cdot)}}.
	\]
	Then using assumption \ref{assumption:T} along with the definition of $B_n$ we have $\sup_{f \in B_n} \norm{\frac{\partial T}{\partial f}\p{\cdot,f(\cdot)} - \frac{\partial T}{\partial f}\p{\cdot,f_0(\cdot)}}_\infty \lesssim \varepsilon_n$, and as we have shown in the calculations above that $\sup_{f \in B_n} \norm{\p{\pi_p - I}\c{\frac{\partial T}{\partial f}\p{\cdot,f(\cdot)}}}_\infty \lesssim p^{-\gamma}$, we obtain
	\[
	\sup_{f \in B_n} \norm{u(f)}_\infty \lesssim \varepsilon_n' := \varepsilon_n + p^{-\gamma} \asymp p^{- \gamma \wedge s} + \sqrt{\frac{p \ln n}{n}}.
	\]
	Moreover using assumption \ref{assumption:T} and the bound $\norm{\pi_p\c{g}}_\infty \leq \norm{g}_\infty$ for any $g \in L^\infty$, we obtain $\norm{u(f) - u(g)}_\infty \leq \norm{f-g}_\infty$ for any $f,g \in B_n$. As a result if $0 < \delta < 1$ and $f_1,\ldots,f_L$ form a $\delta$-net of $B_n$ with respect to the supremum norm the functions $u(f_1),\ldots,u(f_L)$ are themselves a $\delta$-net of $u\p{B_n} = \b{u(f) : f \in B_n}$ with respect to the supremum norm. Because $B_n \subset B_n' := \b{f \in \FF_p : \norm{f-f_0}_\infty < C \varepsilon_n}$ and that $B_n'$ has a $\norm{\cdot}_\infty$-covering number bounded by $\p{C'\varepsilon_n/\delta}^p$ where $C'>0$ is a fixed constant (take for instance elements of the form $f_0 + \sum_{\mu = 1}^p f^\mu \ind{I_\mu}$ where $f^\mu \in \b{l \delta : l \in \mathbb Z, \abs{l} < C\varepsilon_n/\delta}$), this implies that $L$ can be chosen such that $\ln L \asymp p \ln \p{\varepsilon_n/\delta}$. Thus we obtain
	\[
	\sup_{f \in B_n} \abs{\p{\Pro_n - \Pro_0} \c{u(f)}} \leq \delta + \max_{1 \leq l \leq L} \abs{\p{\Pro_n - \Pro_0}\c{u(f_l)}}.
	\]
	Because the variables $\sqrt{n}\p{\Pro_n - \Pro_0} \c{u(f_l)}, \quad 1 \leq l \leq L$ are all subgaussian with variance proxy bounded by a multiple of $\max_{1 \leq l \leq L} \norm{u(f_l)}_\infty^2 \lesssim \varepsilon_n'^2$, the maximum $\max_{1 \leq l \leq L} \sqrt{n} \abs{\p{\Pro_n - \Pro_0}\c{u(f_l)}}$ is itself subgaussian with variance proxy bounded by a multiple of $\varepsilon_n'^2 \ln L$ (see e.g. \cite{vershyninHighDimensionalProbabilityIntroduction2018}) and therefore, for a universal constant $c>0$,
	\[
	\Pro_0 \c{\max_{1 \leq l \leq L} \sqrt{n} \abs{\p{\Pro_n - \Pro_0}\c{u(f_l)}} > t} \leq \exp\p{-c \frac{t^2}{\varepsilon_n'^2 \ln L}}, \quad t > 0.
	\]
	This proves that $\max_{f \in B_n} \abs{\p{\Pro_n - \Pro_0} \c{u(f)}} = O_p\p{\delta + t/\sqrt{n}}$ whenever $t \gtrsim \varepsilon_n' \sqrt{\ln L}$. Notice that $\varepsilon_n' \sqrt{p\ln n} = o(1)$ since $\p{\ln n}^{\frac{1}{2(\gamma \wedge s) - 1}} \ll p \ll \sqrt{n/\ln^2 n}$; thus choosing $\delta = t/\sqrt{n}$ and $t = \varepsilon_n' \sqrt{p \ln n}$ yields $t = o(1)$ and $t \geq \varepsilon_n' \sqrt{p \ln n} \asymp \varepsilon_n' \sqrt{\ln L}$, because $\ln L \asymp p \ln n$. We thus get $\max_{f \in B_n} \abs{\p{\Pro_n - \Pro_0} \c{u(f)}} = O_p\p{\delta + t/\sqrt{n}} = O_p\p{t/\sqrt{n}} = O_p\p{\varepsilon_n' \sqrt{\frac{p \ln n}{n}}}$.
	Combining the bounds obtained for the three different terms we obtain
	\[
	\sup_{f \in B_n} \abs{\hat \Phi - \hat \Psi_0} = O_p\p{\varepsilon_n^2 + p^{-(s+\gamma)} + \varepsilon_n' \sqrt{\frac{p \ln n}{n}}} = O_p\p{p^{-2s} + p^{-(s+\gamma)} + p^{-s \wedge \gamma} \sqrt{\frac{p \ln n}{n}} + \frac{p \ln n}{n}}.
	\]
\end{proof}

\subsection{Additional technical proofs}\label{section:justifications_lemma_laplace_transform}

\begin{proof}{Additional technical justifications for the proof of Lemma \ref{lemma:contraction_rates_white_noise}}\\
First, given a fixed choice of $\chi : \Theta \to \c{0,1}$ of class $\CC^1$ that is compactly supported and identically equal to $1$ on $\b{\theta : \norm{\theta}_2 \leq 1}$ we define $\chi_l := \chi\p{2^{-l} \cdot}$. We then have $\chi_l \to 1$ pointwise and we now prove that, for any $t>0$,
\begin{align*}
	& e^{t\abs{\Phi_a}} \b{1 + \abs{\Div_\pi V} + \abs{\nabla \Phi_a} + \abs{V} + \abs{\inner{V|\nabla \Phi_a}} + \abs{\hat \Phi_a - \hat \Psi_{a0}} + \abs{\inner{\nabla \chi_l | V}}} \\
	& = e^{t\abs{\Phi_a}} \b{1 + \abs{\Delta_\pi \Phi_a} + 2\abs{\nabla \Phi_a} + \abs{\nabla \Phi_a}^2 + \abs{\hat \Phi_a - \hat \Psi_{a0}} + \abs{\inner{\nabla \chi_l | \nabla \Phi_a}}} \in L^1\p{\Pi_n}.
\end{align*}
First Proposition \ref{proposition:linear_functionals_white_noise} yields $\hat \Phi_a = \hat \Psi_{a0}$ and $\abs{\nabla \Phi_a}^2 = N^{-1} \norm{a}_2^2$, so that it suffices to show that
\[
e^{t\abs{\Phi_a}} \b{1 + \abs{\Delta_\pi \Phi_a} + \abs{\inner{\nabla \chi_l | \nabla \Phi_a}}} \in L^1\p{\Pi_n} \text{ and } \Pi_n \c{e^{t\tau} \inner{\nabla \chi_l | \nabla \Phi_a}} \to 0.
\]
First
\[
\abs{\inner{\nabla \chi_l | \nabla \Phi_a}} = \abs{N^{-1} \sum_{\mu = 1}^p \p{\partial_\mu \chi_l}\p{\partial_\mu \Phi_a}} \leq N^{-1} 2^{-l} \norm{\abs{\nabla \chi}}_\infty \norm{a}_2,
\]
moreover for priors of type \ref{prior:independent_series_white_noise} Proposition \ref{proposition:linear_functionals2_white_noise} together with Cauchy--Schwarz inequality yield
\[
\abs{\Delta_\pi \Phi_a} = N^{-1} \abs{\sum_{\mu = 1}^p a^\mu \sigma_\mu^{-1} H'\p{\theta^\mu/\sigma_\mu}} \leq N^{-1} \norm{H'}_\infty \norm{a}_2 \p{\sum_{\mu = 1}^p \sigma_\mu^{-2}}^{1/2},
\]
so that $e^{t\abs{\Phi_a}} \b{1 + \abs{\Delta_\pi \Phi_a} + \abs{\inner{\nabla \chi_l | \nabla \Phi_a}}}$ is bounded by a constant multiple of $e^{t\abs{\Phi_a}}$ which itself is an element of $L^1\p{\Pi_n}$ by the same reasoning made in Section \ref{section:white_noise} to show that $\Pi_n$ is well defined for every $x_1,\ldots,x_n$, i.e. that $e^{\ell_n} \in L^1\p{\Pi}$, using the quadratic decay of the log-likelihood. This also implies that $\abs{\Pi_n \c{e^{t\tau} \inner{\nabla \chi_l | \nabla \Phi_a}}}$ is bounded by a constant multiple of $N^{-1} 2^{-l} \norm{\abs{\nabla \chi}}_\infty \norm{a}_2$ which converges to $0$ as $l \to \infty$, thus proving that Lemma \ref{lemma:laplace_transform} applies in that case. For priors of type \ref{prior:gaussian} by Proposition \ref{proposition:linear_functionals2_white_noise} and Cauchy--Schwarz' again,
\[
\abs{\Delta_\pi \Phi_a} = \abs{N^{-1} \sum_{\mu = 1}^p \sigma_\mu^{-2} \theta^\mu a^\mu} \leq N^{-1} \norm{a}_2 \p{\max_{1 \leq \mu \leq p} \sigma_\mu^{-2}} \norm{\theta}_2,
\]
so that the same reasoning applies here using instead $e^{t\abs{\Phi_a}}\p{1 + \norm{\theta}_2} \in L^1\p{\Pi_n}$, which follows again by using the quadratic decay of either the log-likelihood or, in this case, the prior. We conclude that Lemma \ref{lemma:laplace_transform} applies in both cases.
\end{proof}

\begin{proof}{Additional technical justifications for the proof of Lemma \ref{lemma:supremum_norm_contraction_rates_histograms}}\\
For a fixed $\CC^1$ function $\chi : \RR \to \c{0,1}$ such that $\chi \equiv 1$ on $\c{-1,1}$ and $\Supp\p{\chi} \subset ]-2,2[$ define
\[
\chi_l = \prod_{\mu = 1}^{p-1} \chi\p{2^{-l}\theta^\mu}.
\]
Then $\chi_l$ is compactly supported, converges to $1$ pointwise and, using the $\theta-\eta$ coordinate system,
\[
\inner{\nabla \chi_l | V} = \inner{\nabla \chi_l | \nabla \Phi_a} = \p{\partial_\mu \chi_l}\p{\partial^\mu \Phi_a} = \sum_{\mu = 1}^{p-1} 2^{-l} \chi'\p{2^{-l}\theta^\mu} \p{\prod_{\nu \neq \mu} \chi\p{2^{-l} \theta^\nu}} \inner{u_\mu | a}_2,
\]
so that
\[
\abs{\inner{\nabla \chi_l | V}} \leq 2^{-l} \norm{\chi'}_\infty \sum_{\mu = 1}^{p-1} \abs{\inner{u_\mu | a}_2}.
\]
To conclude we distinguish between the two priors.
\begin{enumerate}
	\item We start with a prior of type \ref{prior:independent_haar_series_histograms}. Let $\mu_0 \in \b{1,\ldots,p}, \quad a = \ind{I_{\mu_0}} \in \FF_p, \quad \Phi_a$ the functional $\Phi = \Phi_a = \int a f_\theta = P_\theta \c{I_{\mu_0}}$ and $V = \nabla \Phi_a$. Then by Proposition \ref{proposition:linear_functionals_histograms} we have
	\[
	\abs{\nabla \Phi_a}^2 = P_\theta \c{\p{\ind{I_{\mu_0}} - P_\theta \c{I_{\mu_0}}}^2} = \Phi_a \p{1-\Phi_a} \leq 1
	\]
	and
	\[
	\hat \Phi_a = \Phi_a + \frac{1}{n} \inner{\nabla \Phi_a | \nabla \ell_n} = \Pro_n \c{a} =: \hat \Psi_0.
	\]
	Moreover using Proposition \ref{proposition:linear_functionals_histograms_2} we have
	\[
	\abs{\Div_\pi V} = \abs{\Delta_\pi \Phi_a} = \abs{- \sum_{\mu = 1}^{p-1} \sigma_\mu^{-1} H'\p{\theta^\mu/\sigma_\mu} \inner{u_\mu | a}_2} \leq \sigma_-^{-1} \norm{H'}_\infty \sum_{\mu = 1}^{p-1} \abs{\inner{u_\mu |a}_2},
	\]
	therefore, for any $t>0$,
	\begin{align*}
	& e^{t\abs{\Phi}} \b{1 + \abs{\Div_\pi V} + \abs{\nabla \Phi} + \abs{V} + \abs{\inner{V|\nabla \Phi}} + \abs{\hat \Phi - \hat \Psi_0} + \abs{\inner{\nabla \chi_l | V}}} \\
	& \leq e^t \b{4 + \sigma_-^{-1} \norm{H'}_\infty \sum_{\mu = 1}^{p-1} \abs{\inner{u_\mu |a}_2} + 2^{-l} \norm{\chi'}_\infty \sum_{\mu = 1}^{p-1} \abs{\inner{u_\mu | a}_2}} \in L^1\p{\Pi_n}
	\end{align*}
	as the latter is a constant. The same bound on $\inner{\nabla \chi_l | V}$ also yields, for any $t \in \RR$,
	\[
	\abs{\Pi_n \c{e^{t\tau}\inner{\nabla \chi_l | \nabla \Phi_a}}} \leq e^{2\abs{t}\sqrt{n}}2^{-l} \norm{\chi'}_\infty \sum_{\mu = 1}^{p-1} \abs{\inner{u_\mu | a}_2} \to 0
	\]
	as $l \to \infty$ so that Lemma \ref{lemma:laplace_transform} applies here.
	\item For Dirichlet priors \ref{prior:dirichlet} we verify that Lemma \ref{lemma:laplace_transform} applies but to the functional $\Phi = \Phi_a - \frac{1}{n} \Delta_\pi \Phi_a$ and the vector field $V = \nabla \Phi_a$. We still have $\abs{V} = \abs{\nabla \Phi_a} = \sqrt{\Phi_a\p{1-\Phi_a}} \leq 1$, and using Proposition \ref{proposition:linear_functionals_histograms_2} we have $\Div_\pi V = \Delta_\pi \Phi_a = \norm{m}_1 \int a\p{h_{\bar m} - h_\omega}$, therefore $\abs{\Div_\pi V} = \abs{\Delta_\pi \Phi_a} \leq 2\norm{m}_1 \norm{a}_\infty$ and
	\[
	\Phi = \Phi_a - \frac{1}{n} \Delta_\pi \Phi_a = \int a h_\omega - \frac{\norm{m}_1}{n} \int a\p{h_{\bar m} - h_\omega} = \p{1 + \frac{\norm{m}_1}{n}} \Phi_a - \frac{\norm{m}_1}{n} \int ah_{\bar m},
	\]
	which also implies
	\[
	\nabla \Phi = \p{1 + \frac{\norm{m}_1}{n}} \nabla \Phi_a, \quad \abs{\nabla \Phi} = \p{1 + \frac{\norm{m}_1}{n}} \abs{\nabla \Phi_a} \leq 1 + \frac{\norm{m}_1}{n}.
	\]
	We deduce (using also Cauchy--Schwarz inequality) that for any $t > 0$ we have
	\begin{align*}
		& e^{t\abs{\Phi}} \b{1 + \abs{\Div_\pi V} + \abs{\nabla \Phi} + \abs{V} + \abs{\inner{V|\nabla \Phi}} + \abs{\hat \Phi - \hat \Psi_0} + \abs{\inner{\nabla \chi_l | V}}} \\
		& \leq e^{t \b{1 + \frac{\norm{m}_1}{n} \norm{a}_\infty}} \b{4 + 2\norm{m}_1 \norm{a}_\infty + \frac{\norm{m}_1}{n} + 2^{-l} \norm{\chi'}_\infty \sum_{\mu = 1}^{p-1} \abs{\inner{u_\mu | a}_2}} \in L^1\p{\Pi_n}
	\end{align*}
	as the latter is a constant. The same bound on $\inner{\nabla \chi_l | V}$ also yields, for any $t \in \RR$,
	\[
	\abs{\Pi_n \c{e^{t\sqrt{n}\p{\Phi - \hat \Psi_0}} \inner{\nabla \chi_l | V}}} \leq e^{\abs{t}\sqrt{n}\p{1 + \frac{\norm{m}_1}{n} \norm{a}_\infty}} 2^{-l} \norm{\chi'}_\infty \sum_{\mu = 1}^{p-1} \abs{\inner{u_\mu | a}_2} \to 0
	\]
	as $l \to \infty$ so that Lemma \ref{lemma:laplace_transform} applies here.
\end{enumerate}
\end{proof}

\begin{proposition}\label{proposition:liouvilles_formula}{Liouville's formula}\\
	If $V$ is a vector field and $\p{\varphi_t}$ its generated flow (that we assume to be smooth and globally defined for simplicity), then for any distribution $\Pi$ that is absolutely continuous with respect to the Riemannian volume form with relative density $\pi$ we have
	\[
	\frac{d \p{\Pi \circ \varphi_t^{-1}}}{d\Pi} = \exp \p{-\int_0^t \p{\Div_\pi V} \p{\varphi_{-s}\p{\cdot}} ds}.
	\]
	Moreover, if $\p{t,\theta} \mapsto \varphi_t\p{\theta}$ is any family of diffeomorphisms (not necessarily a vector flow) that satisfies $\varphi_0 = I$ and $\dot \varphi_0 = V$, then
	\[
	\p{\frac{d}{dt}}_{t = 0} \frac{d \p{\Pi \circ \varphi_t^{-1}}}{d\Pi} = - \Div_\pi V.
	\]
\end{proposition}
\begin{proof}{Proposition \ref{proposition:liouvilles_formula}}\\
	Let $F : \Theta \to \RR$ be an arbitrary nonnegative and measurable function and $\Pi_t := \Pi \circ \varphi_t^{-1}$. Then by definition and by the change of variable formula we have
	\[
	\Pi_t \c{F} = \int F\p{\varphi_t\p{\theta'}} \Pi\p{\theta'}d\theta' = \int F\p{\theta} \Pi \p{\varphi_{-t}\p{\theta}} \abs{\frac{\partial \varphi_{-t}(\theta)}{\partial \theta}}d\theta,
	\]
	(we remind the reader that as in Section \ref{section:information_geometry} $\Pi\p{\cdot}$ denotes the Lebesgue density of $\Pi$ in local coordinates while $\pi$ is the one with respect to the Riemannian volume form), but using Jacobi's formula for the derivative of the determinant we obtain
	\[
	\frac{\partial}{\partial t} \ln J(t,\theta) = \Tr \c{ \c{ \frac{\partial}{\partial t} \frac{\partial \varphi_{-t}}{\partial \theta} } \c{\p{\frac{\partial \varphi_{-t}}{\partial \theta}}^{-1}}}
	\]
	where
	\[
	J\p{t,\theta} := \abs{\frac{\partial \varphi_{-t}\p{\theta}}{\partial \theta}}
	\]
	is the Jacobian determinant. Since
	\[
	\frac{\partial}{\partial t} \varphi_{-t}^\mu\p{\theta} = - V^\mu\p{\varphi_t\p{\theta}}
	\]
	which implies
	\[
	\frac{\partial}{\partial t} \frac{\partial}{\partial \theta^\nu} \varphi_{-t}^\mu\p{\theta} = \frac{\partial}{\partial \theta^\nu} \frac{\partial}{\partial t} \varphi_{-t}^\mu\p{\theta} = - \frac{\partial}{\partial \theta^\nu} V^\mu\p{\varphi_{-t}\p{\theta}} = - \frac{\partial V^\mu}{\partial \theta^\tau}\p{\varphi_{-t}\p{\theta}} \frac{\partial \varphi_{-t}^\tau}{\partial \theta^\nu}\p{\theta}
	\]
	and therefore
	\[
	\p{\c{ \frac{\partial}{\partial t} \frac{\partial \varphi_{-t}}{\partial \theta} } \c{\p{\frac{\partial \varphi_{-t}}{\partial \theta}}^{-1}}}_{\mu,\nu} = - \frac{\partial V^\mu}{\partial \theta^\nu}\p{\varphi_{-t}\p{\theta}}.
	\]
	As a result
	\[
	\frac{\partial}{\partial t} \ln J(t,\theta) = - \sum_{\tau = 1}^p \frac{\partial V^\tau}{\partial \theta^\tau}\p{\varphi_{-t}\p{\theta}} = - \b{ \Div V - \Gamma_{~\mu \nu}^\nu V^\mu } \p{\varphi_{-t}\p{\theta}}.
	\]
	Furthermore by definition of $\pi$
	\[
	\frac{\partial}{\partial t} \ln \Pi \p{\varphi_{-t}\p{\theta}} = - \b{V^\mu \partial_\mu \ln \Pi} \p{\varphi_{-t}\p{\theta}} = - \b{\inner{\nabla \ln \pi | V} + \Gamma_{~\mu \nu}^\nu V^\mu} \p{\varphi_{-t}\p{\theta}},
	\]
	therefore
	\[
	\frac{\partial}{\partial t} \b{\ln J(t,\theta) + \ln \Pi \p{\varphi_{-t}\p{\theta}}} = - \b{\Div V + \inner{\nabla \ln \pi | V}} \p{\varphi_{-t}\p{\theta}} = - \p{\Div_\pi V} \p{\varphi_{-t}\p{\theta}}.
	\]
	Finally since $J(0) = 1$ we obtain
	\[
	J\p{t,\theta} \Pi \p{\varphi_t\p{\theta}} = \Pi \p{\theta} \exp \p{- \int_0^t \p{\Div_\pi V} \p{\varphi_{-s}\p{\theta}} ds}.
	\]
	As a result
	\[
	\Pi_t \c{F} = \int F\p{\theta} \exp \p{- \int_0^t \p{\Div_\pi V} \p{\varphi_{-s}\p{\theta}} ds} \Pi \c{d\theta},
	\]
	which concludes the proof of the first statement since $F$ is arbitrary. The second one follows by the same reasoning, using the fact that $\dot \varphi_t = V\p{\varphi_t\p{\cdot}}$ but for $t = 0$ only.
\end{proof}

\end{document}